\documentclass{amsart}

\usepackage[nonewpage]{imakeidx}
\makeindex[columns=3, title=Index, options= -s example_style.ist]
\let\oldtocsection=\tocsection
\let\oldtocsubsection=\tocsubsection
\let\oldtocsubsubsection=\tocsubsubsection
\renewcommand{\tocsection}[2]{\hspace{0em}\oldtocsection{#1}{#2}}
\renewcommand{\tocsubsection}[2]{\hspace{1em}\oldtocsubsection{#1}{#2}}
\renewcommand{\tocsubsubsection}[2]{\hspace{2em}\oldtocsubsubsection{#1}{#2}}

\usepackage{float}

\usepackage[%
 bookmarks=true,%
 bookmarksdepth=tocdepth,%
 bookmarksnumbered=true,%
 colorlinks=false,%
 pdftitle={},%
 pdfsubject={},%
 pdfauthor={},%
 pdfkeywords={}%
]{hyperref}

\usepackage{lineno}

\usepackage{graphicx}
\usepackage{enumerate}
\usepackage{amsmath,amssymb,amscd,amsthm}
\usepackage{mathrsfs}
\usepackage{bm}
\usepackage{color}
\usepackage{ascmac}
\usepackage{here}

\newtheorem{mainresult}{Theorem}

\newtheorem{thm}{Theorem}[section]
\newtheorem{lem}[thm]{Lemma}
\newtheorem{prop}[thm]{Proposition}

\theoremstyle{definition}

\newtheorem{dfn}[thm]{Definition}
\newtheorem{op}[thm]{Question}

\newtheorem{rmk}[thm]{Remark}

\numberwithin{equation}{section}
\numberwithin{figure}{section}

\def\smfd{W^{\mathrm{s}}}
\def\umfd{W^{\mathrm{u}}}
\def\lsmfd{W^{\mathrm{s}}_{\mathrm{loc}}}

\def\lumfd{W^{\mathrm{u}}_{\mathrm{loc}}}

\def\setl{\setlength{\leftskip}{-18pt}}

\title[Pluripotency  of wandering dynamics]
{Pluripotency of wandering dynamics}

\author{\sc Shin Kiriki}
\address[Shin Kiriki]{Department of Mathematics, Tokai University, 4-1-1 Kitakaname, Hiratuka, Kanagawa, 259-1292, JAPAN}
\email{kiriki@tokai.ac.jp}

\author{\sc Yushi Nakano}
\address[Yushi Nakano]{Department of Mathematics, Hokkaido University, Kita 10, Nishi 8, Kita-Ku,
Sapporo, Hokkaido 060-0810, JAPAN}
\email{yushi.nakano@tsc.u-tokai.ac.jp}

\author{\sc Teruhiko Soma}
\address[Teruhiko Soma]{Department of Mathematical Sciences, Tokyo Metropolitan University, 1-1 Minami-Ohsawa, Hachioji, Tokyo, 192-0397, JAPAN}
\email{tsoma@tmu.ac.jp}


\date{\today}

\begin{document}

\begin{abstract}
This paper isolates a perturbative mechanism, which we call \emph{pluripotency},
by which the symbolic and statistical behavior of prescribed orbits
in a uniformly hyperbolic set can be realized, after an arbitrarily small perturbation,
along the forward orbits of all points in a set of positive Lebesgue measure.

In this sense, pluripotency provides a way of reprogramming dynamics from both statistical and geometric viewpoints: the empirical measures of all points in a positive-measure set can be made to asymptotically follow those of a prescribed orbit in the hyperbolic set. We first give an abstract criterion, formulated in terms of symbolic itinerary descriptions, which is equivalent to a strong form of pluripotency. We then prove that this mechanism occurs robustly in higher-dimensional non-hyperbolic dynamics. More precisely, for every $2\le r<\infty$ and $\dim M\ge 3$, there exists a $C^r$-open set of diffeomorphisms with wild blender-horseshoes such that every diffeomorphism in this open set is strongly pluripotent for a dense invariant subset of the blender-horseshoe. 

As applications, this yields dense classes of diffeomorphisms with non-trivial Dirac physical measures and with historic wandering domains inside the same open set, providing a new mechanism related to Takens' last problem.

\end{abstract}
\maketitle

\tableofcontents

\section{Introduction}\label{S_Introduction}

\subsection{Motivations}\label{ss_motivation}
Birkhoff's ergodic theorem implies that if 
$\mu$ is 
an ergodic invariant probability measure
for a continuous map $f$ on a compact manifold $M$, 
then
$\mu$-almost every point $x$ of $M$ has the limit of 
time averages for any continuous function $\varphi:M\longrightarrow \mathbb{R}$, that is, 
\begin{equation}\label{eqn_BET}
\lim_{n\rightarrow \infty} \frac{1}{n}
\sum_{i=0}^{n-1}
\varphi\circ f^{i}(x)=\int_{M}\varphi\ d\mu.
\end{equation}
If $f$ is an element of $\mathrm{Diff}^{2}(M)$ with an Axiom A attractor, then there exists 
an ergodic invariant measure $\mu$ whose support is the attractor and such that
the set of initial points $x$ for which \eqref{eqn_BET} holds has positive Lebesgue measure; 
see \cite{Si72,R76,Bo08}. 
Under hyperbolicity assumptions weaker than Axiom~A, 
the existence of such measures is non-trivial, but the theory of SRB measures has substantially advanced the situation; 
see \cite{BDV05} for a comprehensive account.

From the point of view of observable statistics, two kinds of phenomena have played a central role in the study of non-hyperbolic dynamics. The first one concerns physical measures whose supports are not attractors. This leads to the notion of \emph{non-trivial Dirac physical measures}, studied for instance in \cite{CV01, SSV10, SV13, Sa18, GGS22}. The second one concerns the failure of convergence of empirical measures. Namely, for a point $x$ of $M$ and an element $f$ of $\mathrm{Diff}^{r}(M)$, let
\begin{equation}\label{af}
\delta^{n}_{x,f}=
\frac{1}{n}\sum_{i=0}^{n-1}\delta_{f^i(x)}
\end{equation}
be the empirical measure of the first $n$ iterates of $x$. It is natural to ask whether there are abundant dynamical systems for which the sequence $(\delta^n_{x,f})_{n\geq 0}$ fails to converge on a set of positive Lebesgue measure. Orbits with such initial points are said to have \emph{historic behavior}. Questions of this type were raised by Ruelle and later developed by Takens \cite{R01,T08}; they are now commonly referred to as Takens' last problem.

Let us recall the notions of persistence and robustness, following the terminology used in \cite[Section 11]{PR83} and \cite[Section 3]{T08}.

\begin{dfn}[persistent and robust properties]\label{PR}
Let $\mathcal{C}$ be a non-empty subset of
$\mathrm{Diff}^r(M)$, which is called a \emph{class}.
We say that a property $\mathscr{A}$ is $C^r$-\emph{persistent} relative to $\mathcal{C}$ if every $f\in \mathcal{C}$ has the property $\mathscr{A}$. Such a property is said to be $C^r$-\emph{robust} when $\mathcal{C}$ is open.
\end{dfn}

There are many examples in which the set of initial points with historic behavior is residual but has zero Lebesgue measure; see, for instance, \cite{T08, BLV14, KLS16, Ya20, BKNRS20, AP21, CCSV24, P}. There are also examples for which this set has positive Lebesgue measure \cite{CV01, CYZ20, Sb21, Tal22}, although the abundance of the corresponding systems is not addressed there. The first affirmative result in the direction of Takens' problem was obtained in \cite{KS17}, where historic behavior was shown to be $C^r$-persistent relative to a dense subset of the Newhouse domain for two-dimensional diffeomorphisms, $2\le r<\infty$. Further results were obtained by Barrientos \cite{B22}, in a higher-dimensional setting reducible to two-dimensional dynamics, and by Berger--Biebler \cite{BB23} for two-dimensional diffeomorphisms in the $C^\infty$ and real-analytic categories. However, whether historic behavior itself can be a robust property remains open.

The problems of non-trivial Dirac physical measures and historic behavior are closely related. They have been treated simultaneously in several works, beginning with the affine horseshoe model with homoclinic tangency of Colli--Vargas \cite{CV01}, and more recently in the setting of full branch maps including perturbed Lorenz-like maps by Coates--Luzzatto \cite{CL}. There are also related studies on statistical stability and instability of invariant measures; see \cite{DHL06, AGL26,Tal}.

The purpose of this paper is to isolate a mechanism underlying these phenomena. 
Rather than treating non-trivial Dirac physical measures and historic behavior as separate outcomes, we formalize a perturbative principle by which the statistical behavior of a hyperbolic symbolic subsystem can be transferred to a set of positive Lebesgue measure. 
We call this mechanism the \emph{pluripotent property}\footnote{
The terminology is inspired by induced pluripotent stem cells (iPSCs), 
which are somatic cells reprogrammed into a pluripotent state by introducing a small number of factors. 
Shinya Yamanaka and John B.\! Gurdon received the 2012 Nobel Prize in Physiology or Medicine 
for the discovery that mature cells can be reprogrammed to become pluripotent \cite{Yam12}. 
}. 
It is extracted from the geometric models of \cite{CV01,KNS23}, but it is formulated independently of those particular models. 
This viewpoint was further developed in \cite{KLNSV25}, where strong pluripotency was applied to Takens' last problem in a related setting. 
We first give a practical criterion for a strong form of pluripotency in terms of itinerary descriptions, and then prove that this strong pluripotency occurs $C^r$-robustly in a class of higher-dimensional non-hyperbolic dynamical systems for $2\le r<\infty$.

\subsection{Pluripotency}\label{ss_pluri_lem}

A uniformly hyperbolic horseshoe contains a rich collection of symbolic and statistical behaviors. However, such behavior is usually confined to a hyperbolic set of zero Lebesgue measure, and hence is not observable in the measure-theoretic sense. The role of pluripotency is to bridge this gap. It asks whether, after an arbitrarily small perturbation, the statistical behavior of a prescribed orbit in the horseshoe can be reproduced along the forward orbits of all points in a set of positive Lebesgue measure.

We formulate this property for a uniformly hyperbolic invariant set with a symbolic description. In the definition below, the horseshoe may be replaced by a uniformly hyperbolic set $\Lambda$ which is a maximal $f$-invariant set in the disjoint union of $n\,(\geq 2)$ open sets such that $f|_{\Lambda}$ is topologically conjugate to the full two-sided shift on $n$ symbols. To state the definition, we use the first Wasserstein metric.

For any two Borel probability measures $\mu$ and $\nu$ on a compact Riemannian manifold $M$, 
we consider the \emph{first Wasserstein metric} $d_W$ given by
\[
d_W(\mu ,\nu )= \sup_\varphi \left\vert \int _M \varphi \, d\mu - \int _M \varphi \, d\nu \right\vert,
\]
where the supremum is taken over all Lipschitz continuous functions $\varphi : M \longrightarrow [-1, 1]$ whose Lipschitz constants are bounded by $1$. 
See \cite{V09} for its basic properties. In particular, 
since $M$ is compact, $d_W$ metrizes the weak topology on the space of Borel probability measures on $M$.

\begin{dfn}[pluripotency]\label{dfn:0624}
Let $M$ be a compact Riemannian manifold with $\dim M\geq 2$.
Suppose that $f$ is a $C^{r}$ $(r\geq 1)$ diffeomorphism on $M$ with 
a uniformly hyperbolic invariant set $\Lambda$, and let $\Lambda^{\prime}$ be a non-empty subset of $\Lambda$. 
\begin{enumerate}[(1)]
\item
$f$ is \emph{pluripotent} for $\Lambda'$ if, for every $x\in \Lambda^{\prime}$, 
there exist $g\in \mathrm{Diff}^r(M)$ arbitrarily $C^r$-close to $f$ 
and a subset $D$ of $M$ with positive Lebesgue measure
such that, for every $y\in D$ and the continuation $x_g$ of $x$,
\begin{equation}\label{eq:0625}
\lim _{n\rightarrow \infty} d_W( \delta_{y,g}^n, \delta_{x_g,g}^n) =0,
\end{equation}
where $\delta_{y,g}^n$ and $\delta_{x_g,g}^n$ are the empirical measures given by \eqref{af}.
\item  
$f$ is \emph{strongly pluripotent} for $\Lambda'$ if the following stronger condition holds, for $g$ and $D$ as in (1), instead of \eqref{eq:0625}:  
\begin{equation}\label{eq:0625_s}
\lim _{n\rightarrow \infty} \frac1{n}\sum_{i=0}^{n-1}
\sup_{y\in D}\mathrm{dist}(g^i(y),g^i(x_g))=0.
\end{equation}
\end{enumerate}
\end{dfn}

Condition \eqref{eq:0625} means that, after an arbitrarily small perturbation, 
the empirical measures of all points in the positive-measure set $D$ asymptotically follow those of the prescribed orbit of $x_g$. 
Thus the statistical behavior originally carried by an orbit in the hyperbolic set $\Lambda_g$ becomes observable on a set of positive Lebesgue measure. 
The stronger condition \eqref{eq:0625_s} says more: the forward orbits of all points in $D$ spend a density-one set of times uniformly close to the prescribed forward orbit of $x_g$.

It is immediate that \eqref{eq:0625_s} implies \eqref{eq:0625}. Indeed, for any Lipschitz function $\varphi:M\longrightarrow [-1,1]$ with $\mathrm{Lip}(\varphi)\leq 1$, we have
\[
|\varphi(g^i(y))-\varphi(g^i(x_g))|
\leq \mathrm{dist}(g^i(y),g^i(x_g)).
\]
However, Theorem \ref{thm_new} below shows that the converse is not true in general.

Let us spell out two basic consequences of the definition. If $x_g$ is a saddle fixed point, then \eqref{eq:0625} implies that the corresponding Dirac measure is a non-trivial physical measure for $g$. On the other hand, \eqref{eq:0625} can also hold when the sequence $(\delta _{x_g, g}^n)_{n\in \mathbb N}$ does not converge. In this case, every point of $D$ has historic behavior. See Theorem \ref{DrcHst} for precise statements.

\medskip

In this paper, we are mainly concerned with the case where the set $D$ can be chosen to be a non-empty open set; see the third item of Remark \ref{mainrmk}. 
At first sight, the form of \eqref{eq:0625} or \eqref{eq:0625_s} might suggest that such a set $D$ should be a neighborhood of $x_g$. 
This is not the case. The following proposition shows that any open set $D$ satisfying Definition \ref{dfn:0624}\,(1) must be disjoint from the continuation $\Lambda_g$ of $\Lambda$. In particular, $x_g$ never belongs to $D$. 
This observation emphasizes that pluripotency is not a local shadowing property near the prescribed orbit, but an observable realization mechanism outside the hyperbolic set.

\begin{prop}\label{p_D_open}
Let $g$ be an element of $\mathrm{Diff}^r(M)$ and let $D$ be any open subset of $M$ satisfying the conditions in Definition \ref{dfn:0624}\,(1).
Then $D\cap \Lambda_g$ is empty.
\end{prop}

See Section \ref{S_Pf_Pl} for the proof.

\subsection{Describability}

We next introduce a symbolic criterion which will be used to detect strong pluripotency. 
The point is that the observable realization required in Definition \ref{dfn:0624} can be checked by coding the forward orbit of the positive-measure set $D$ along a density-one set of times. 
This converts the statistical condition in \eqref{eq:0625_s} into a symbolic itinerary problem.

A pair $\{\mathbb{U}_0,\mathbb{U}_1\}$ of disjoint open sets in $M$ is
called a \emph{coding pair} of a horseshoe $\Lambda$ if 
\[
\Lambda =\bigcap_{i\in \mathbb{Z}} f^{i}(\mathbb U_0 \sqcup \mathbb U_1)
\]
and the restriction $f|_\Lambda$ is topologically conjugate 
to the shift map on $\{0,1\}^{\mathbb Z}$ by the coding map  
$\mathcal I: \Lambda  \longrightarrow  \{0,1\}^{\mathbb Z}$ satisfying
\[
\left(\mathcal I(x) \right)_j = v 
\quad \text{if } f^j(x) \in \mathbb U_{v},
\]
where $(\mathcal I(x) )_j$ denotes the $j$th entry of $\mathcal I(x)$. 

Unlike the elements of a usual Markov partition, the sets in a coding pair are open and are not assumed to be compact. 
This openness will be essential in the proof of Lemma \ref{l_pluri_1} and related arguments.

The existence of the coding pair above implies that, for each 
$\underline v=(v_j)_{j\in \mathbb Z}\in\{0,1\}^{\mathbb Z}$, 
\[
\{ \mathcal I^{-1}(\underline v)\}=
\bigcap _{j\in \mathbb Z} f^{-j}(\mathbb U_{v_j}).
\]
Moreover, if $g\in \mathrm{Diff}^r(M)$ is sufficiently $C^r$-close to $f$, then the coding pair also gives a natural coding map for the continuation $\Lambda_g$, which we denote by $\mathcal I_g$.

\begin{dfn}[describable property]\label{describable}
Let $\Sigma$ be a subset of $\{0,1\}^{\mathbb N_0}$ and 
let $f$ be an element of $\mathrm{Diff}^{r}(M)$ with a horseshoe $\Lambda$ associated with a coding pair $\{\mathbb U_0,\mathbb U_1\}$, 
where $\mathbb N_0=\{0,1,\ldots\}$.  
We say that $f$ is $\Sigma$-\emph{describable} over $\Lambda$  
if every element $\underline{v}=(v_{0}v_{1}v_{2}\ldots\,)$ of $\Sigma$
satisfies the following two conditions:
\begin{description}
\setl
\item[(DEI) \rm(Dominance of Encoded Intervals)]
There exists a sequence of integer intervals
$
\mathbb{I}_k=[\alpha_k, \alpha_k+\beta_k]\cap \mathbb Z
$,
where $(\alpha_k)_{k\in\mathbb N}$ is a strictly increasing sequence of non-negative integers and each $\beta_k$ $(k\in\mathbb N)$ is a non-negative integer satisfying
\[
\alpha_k+\beta_k+1\leq \alpha_{k+1},
\]
such that
\[
\lim_{N\rightarrow\infty}
\frac{
\#\left\{0\leq n\leq N-1\,;\,
n\in \bigcup_{k=1}^{\infty}\mathbb I_k
\right\}
}{N}
=1.
\]
\item[(OCD) \rm(Observable Coded Description)]
There exist an element $g$ of $\mathrm{Diff}^{r}(M)$ arbitrarily $C^r$-close to $f$ and a subset $D$ of $M$ with positive Lebesgue measure such that  
\[
g^n(D)\subset \mathbb U_{v_n}
\]
for every
$
n\in \bigcup_{k\in\mathbb N}\mathbb I_k.
$
\end{description}
\end{dfn}

Condition (DEI) says that the prescribed code is required only on a union of time intervals whose asymptotic density is one. 
Condition (OCD) says that, after an arbitrarily small perturbation, the forward images of the positive-measure set $D$ follow this prescribed code on those dominant time intervals. 
Thus $\Sigma$-describability is a symbolic form of observable orbit programming.

The following result is the first main step of the paper. 
It shows that strong pluripotency is equivalent to this symbolic describability condition. 
Thus the problem of proving strong pluripotency is reduced to constructing positive-measure sets whose forward images realize prescribed symbolic itineraries on a density-one set of times.

We set
\[
\mathbb{Z}_{(-)}=\mathbb{Z}\cap(-\infty,0)
\]
and, for a given $\Sigma\subset \{0,1\}^{\mathbb N_0}$, define
\[
\widehat \Sigma
=
\bigl\{\underline u\underline v\in \{0,1\}^{\mathbb Z}\,;\,
\underline u\in \{0,1\}^{\mathbb Z_{(-)}},\ 
\underline v\in \Sigma
\bigr\},
\]
where $\underline u\underline v$ is the element of $\{0,1\}^{\mathbb Z}$ defined by
\[
(\underline u\underline v)_j=
\begin{cases}
u_j & \text{if } j\leq -1,\\
v_j & \text{if } j\geq 0.
\end{cases}
\]

\begin{mainresult}[Pluripotency Lemma]\label{p-lemma}
Suppose that $f$ is an element of $\mathrm{Diff}^r(M)$ with a horseshoe 
$\Lambda$ associated with a coding pair $\{\mathbb U_0,\mathbb U_1\}$, and let 
$\Sigma$ be a non-empty subset of $\{0,1\}^{\mathbb N_0}$.
Then $f$ is $\Sigma$-describable over $\Lambda$ if and only if 
$f$ is strongly pluripotent for $\mathcal I^{-1}(\widehat \Sigma)$.
\end{mainresult}

In fact, the proof of Theorem \ref{p-lemma} shows that the perturbation $g$ and the positive-measure set $D$ appearing in Definition \ref{dfn:0624} can be chosen to be exactly those appearing in Definition \ref{describable}.

\subsection{Robustness of pluripotent property}\label{ss_robust_pluri}

We now turn to the robust realization of the pluripotent mechanism introduced above. 
The main point of this subsection is that strong pluripotency is not merely a formal consequence of symbolic hyperbolicity. 
It becomes robustly available when the underlying horseshoe is placed in a genuinely non-hyperbolic configuration, namely in a wild blender-horseshoe. 
We first recall the relevant notions.

A (cs-)\emph{blender} is a transitive hyperbolic set $\Lambda$ of 
$f\in \mathrm{Diff}^{1}(M)$ with s-index $k\geq 2$ and having a superposition region. 
Here, the \emph{superposition region} of $\Lambda$ means a $C^{1}$-open set $\mathcal{D}$ of embeddings of $(k-1)$-dimensional disks $D$ into $M$ such that, for every diffeomorphism $g$ in some $C^{1}$-neighborhood $\mathcal{U}$ of $f$, every disk $D\in \mathcal{D}$ intersects the local unstable manifold $\lumfd(\Lambda_{g})$ of the continuation $\Lambda_{g}$ of $\Lambda$.
The blender $\Lambda$ is called a \emph{blender-horseshoe} if the restriction $f|_{\Lambda}$ is topologically conjugate to the restriction of some diffeomorphism on a horseshoe. 
See, for example, \cite{BD96,BD12}. 
Moreover, a blender-horseshoe $\Lambda$ is called \emph{wild} if $\umfd(\Lambda)$ and $\smfd(\Lambda)$ have a homoclinic tangency in the closure of the superposition region of $\Lambda$. 
Such a non-hyperbolic situation already appeared in \cite{BD12}, although the terminology ``wild'' is not used there.

A blender-horseshoe alone does not give the observable coding required in Definition \ref{describable}. 
Indeed, if $\Lambda$ is a blender-horseshoe far from any homoclinic tangency, then $f|_{\Lambda}$ is topologically conjugate to the shift map on $\{0,1\}^{\mathbb{Z}}$. 
Thus one has a coding pair $\{\mathbb U_0,\mathbb U_1\}$ of $\Lambda$, and the density condition (DEI) in Definition \ref{describable} can be imposed on a prescribed sequence of integer intervals $\mathbb I_k$ for any strictly increasing sequence $(\alpha_k)_{k\in\mathbb N}$ of positive integers and non-negative integers $\beta_k$ with $\alpha_k+\beta_k+1\leq \alpha_{k+1}$. 
However, without a suitable non-hyperbolic mechanism, there is no reason for the observable condition (OCD) to hold on a set $D$ of positive Lebesgue measure. 
The role of wildness is precisely to supply this missing observable component. 
The following theorem shows that, for wild blender-horseshoes constructed in this paper, the symbolic coding can be realized on positive-measure sets in a robust way.

To state the theorem, we introduce two definitions.

\begin{dfn}[majority condition for codes]\label{m-cnd}
For a given binary code $\underline{v}=(v_{j})_{j\in\mathbb{Z}}\in \{0,1\}^{\mathbb{Z}}$ and $n\in \mathbb{N}$, define
\begin{equation}\label{eqn_pnv}
p_{n}(\underline{v})=\frac{
\#\left\{
j\in \mathbb{N}\ ;\ 
n-(3n)^{2/3}< j\leq n,\ 
v_{j}=0
\right\}
}{(3n)^{2/3}}.
\end{equation}
We say that $\underline{v}$ satisfies the \emph{majority condition} if
\[
\liminf_{n\rightarrow \infty}p_{n}(\underline{v})\geq \frac{1}{2}.
\]\index{p_nv@$p_{n}(\underline{v})$}
\end{dfn}

\noindent
We denote by $\widehat\Sigma^{(\mathrm{mj})}$ the subset of $\{0,1\}^{\mathbb{Z}}$ consisting of elements satisfying the majority condition. 
The relation between this majority condition for binary codes and the majority condition for diffeomorphisms used in \cite{KNS23} will be discussed in Remark \ref{r_majority}. 
We set $\Lambda^{(\mathrm{mj})}=\mathcal{I}^{-1}(\widehat\Sigma^{(\mathrm{mj})})$, which is an $f$-invariant dense subset of $\Lambda$.

\begin{dfn}[non-trivial wandering domain]
A \emph{non-trivial wandering domain} for $f$ in $\mathrm{Diff}^{r}(M)$ is a connected non-empty open subset $D$ of $M$ satisfying the following conditions: 
\begin{enumerate}[(1)]
\item $f^{i}(D)\cap f^{j}(D)=\emptyset$ for any integers $i,j\geq 0$ with $i\neq j$, 
\item \label{non-triviality} the union of the $\omega$-limit sets of all $x\in D$, namely $\omega(D,f)=\bigcup_{x\in D}\omega(x,f)$, is not equal to a single periodic orbit.
\end{enumerate}
\end{dfn}

\noindent 
This definition follows the terminology of \cite{dMvS}. 
Instead of \eqref{non-triviality}, one may adopt the stronger non-triviality condition used in \cite{CV01}, requiring that $D$ is not contained in the basin of a weak attractor. 
With this stronger convention, a wandering domain of the classical Denjoy counterexample is no longer non-trivial, since the basin of the weak attractor is the whole circle $S^{1}$. 
The wandering domains constructed in this paper are non-trivial in this stronger sense as well.

In one-dimensional dynamics, the absence of wandering domains is related to several regularity and non-flatness conditions; see \cite{vS10}. 
In higher dimensions, by contrast, phenomena caused by non-hyperbolicity may occur robustly or generically, and wandering domains have played an important role in this context; see \cite{CV01, KS17, B22, BB23}. 
In the present paper, the construction of non-trivial wandering domains is also a key ingredient in proving Theorem \ref{mainthm}; see Remark \ref{mainrmk}.

Let $\Lambda^{\prime}$ be any non-empty subset of a blender-horseshoe for $f$. 
For any diffeomorphism $g$ sufficiently $C^{r}$-close to $f$, we denote by $\Lambda^{\prime}_{g}$ the continuation of $\Lambda^{\prime}$; see \cite{BD12}. 
The following is the main theorem of this paper. 
Its proof occupies Sections \ref{S_WBH}--\ref{S_Proof_thms}.

\begin{mainresult}[$C^{r}$-robustness of strong pluripotency]\label{mainthm}
Let $\dim M\geq 3$ and $2\leq r<\infty$. 
Then 
there exists an $f_{0}\in \mathrm{Diff}^{r}(M)$ having a wild blender-horseshoe $\Lambda$ 
and an open neighborhood $\mathcal{O}$ of $f_0$ in $\mathrm{Diff}^{r}(M)$  
such that any element $f$ of $\mathcal{O}$ is strongly pluripotent for the continuation $\Lambda_{f}^{(\mathrm{mj})}$ of $\Lambda^{(\mathrm{mj})}$.
\end{mainresult}

\begin{rmk}\label{mainrmk}
\begin{enumerate}[$\bullet$]
\item For the proof of Theorem \ref{mainthm}, we show that every element $f$ of $\mathcal{O}$ is $\Sigma^{(\mathrm{mj})}$-describable, which is equivalent to $f$ being strongly pluripotent by Theorem \ref{p-lemma}.
\item The set in the proof of Theorem \ref{mainthm} corresponding to $D$ in Definition \ref{dfn:0624} is connected and open. 
Hence, by Proposition \ref{p_D_open}, any such $D$ is disjoint from $\Lambda_g^{(\mathrm{mj})}$. 
Moreover, $D$ is constructed so as to be a non-trivial wandering domain for some $g$ arbitrarily $C^{r}$-close to $f$; see Theorem \ref{thmWD} for details.
\end{enumerate}
\end{rmk}

We finally explain that the Wasserstein condition \eqref{eq:0625} does not, by itself, force the stronger orbit-wise condition \eqref{eq:0625_s}. 
Although \eqref{eq:0625_s} immediately implies \eqref{eq:0625}, the converse implication fails for a given realizing perturbation and positive-measure set. 
The following theorem shows that any diffeomorphism $f$ as in Theorem \ref{mainthm} can be approximated by another diffeomorphism for which the pluripotency condition \eqref{eq:0625} is realized on a contracting wandering domain, while the corresponding strong condition \eqref{eq:0625_s} fails for the same prescribed orbit and domain.

\begin{thm}\label{thm_new}
Under the assumptions of Theorem \ref{mainthm}, 
for any element $f$ of $\mathcal{O}$, there exist 
$x\in \Lambda_f^{(\mathrm{mj})}$, a diffeomorphism $g\in \mathcal{O}$ arbitrarily $C^r$-close to $f$, and a contracting wandering domain $D$ for $g$ such that
\[
\liminf_{n\rightarrow\infty}\frac1{n}\sum_{j=0}^{n-1}\inf_{y\in D}\mathrm{dist}(g^j(y),g^j(x_g))>0
\quad
\text{and}
\quad
\lim_{n\rightarrow\infty} \sup_{y\in D}d_W(\delta_{y,g}^n,\delta_{x_g,g}^n)=0,
\]
where $x_g\in \Lambda_g^{(\mathrm{mj})}$ is the continuation of $x$.
\end{thm}

\subsection{Pluripotency and Takens' last problem}

We next explain how the pluripotent mechanism yields non-trivial Dirac physical measures and historic behavior, the two phenomena discussed in Subsection \ref{ss_motivation}. 
This is the point at which the abstract mechanism introduced above is converted into concrete statistical consequences related to Takens' last problem.

In \cite{KNS23}, we studied a three-dimensional diffeomorphism similar to the concrete model of $f_0$ used in the proof of Theorem \ref{mainthm}. 
It was shown in \cite[Theorem A]{KNS23} that there exist diffeomorphisms $g$ arbitrarily $C^{r}$-close to $f_0$ satisfying either of two contrasting properties: the existence of a non-trivial Dirac physical measure, or the existence of a historic wandering domain. 
However, the density of such diffeomorphisms in an ambient open set was not addressed there. 
Using the model and arguments developed for Theorem \ref{mainthm}, we prove that these properties are $C^{r}$-persistent relative to dense classes inside the same $C^r$-open set $\mathcal{O}$.

Let $\mathscr{D}_{\underline{0}}$ denote the property that a diffeomorphism $g\in \mathcal{O}$ has a non-trivial Dirac physical measure supported at the saddle fixed point $\mathcal{I}_{g}^{-1}(\underline{0})$, where $\underline{0}$ is the two-sided infinite sequence all of whose entries are zero. 
Let $\mathscr{H}$ denote the property that $g$ has a non-trivial wandering domain such that the forward orbit of every point in this domain has historic behavior.

\begin{thm}\label{DrcHst}
Suppose that $\mathcal{O}$ is the $C^{r}$-open set in Theorem \ref{mainthm}. 
Then there exist disjoint dense classes $\mathcal{D}$ and $\mathcal{H}$ in $\mathcal{O}$ satisfying the following conditions.
\begin{enumerate}[\rm (1)]
\item \label{DrcHst-1} $\mathscr{D}_{\underline{0}}$ is $C^{r}$-persistent relative to $\mathcal{D}$.
\item \label{DrcHst-2} $\mathscr{H}$ is $C^{r}$-persistent relative to $\mathcal{H}$.
\end{enumerate}
\end{thm}

Theorem \ref{DrcHst}\,\eqref{DrcHst-2} gives an affirmative answer to 
Takens' last problem in the sense of dense persistence inside the non-empty $C^r$-open set $\mathcal{O}$ for diffeomorphisms of dimension at least three. 
This should be distinguished from the higher-dimensional result of Barrientos \cite{B22}. 
The method of \cite{B22} is based on reducing the dynamics to an appropriate two-dimensional setting and applying the result of \cite{KS17}. 
Such an approach may yield a similar conclusion for a dense subset of some open set arbitrarily close to $f_0$, but it does not directly give density throughout the specific neighborhood $\mathcal{O}$ obtained in Theorem \ref{mainthm}. 
In contrast, the present approach uses the pluripotent mechanism associated with wild blender-horseshoes and yields a dense class $\mathcal{H}$ in the whole of $\mathcal{O}$.

\subsection{Further discussions and outline of this paper}
We end this introduction with a further question and an outline of the paper.

The question concerns the extent to which pluripotency can hold on the whole hyperbolic set. 
Theorem \ref{mainthm} gives strong pluripotency for the dense invariant subset $\Lambda_f^{(\mathrm{mj})}$ of the blender-horseshoe $\Lambda_f$. 
In related two-dimensional settings, it follows from \cite{KS17} that any diffeomorphism in a Newhouse domain is strongly pluripotent for a certain proper subset $\Lambda'$ of the corresponding basic set $\Lambda$. 
However, due to technical restrictions, this does not immediately imply strong pluripotency for the whole basic set. 
This leads to the following question.

\begin{op}\label{op_pluri2} 
Is any diffeomorphism in every Newhouse domain strongly pluripotent for the corresponding basic set?
\end{op}

\medskip

We now give an outline of the paper.
Theorem \ref{p-lemma} is independent of the remaining main results, and hence Section \ref{S_Pf_Pl}, which contains its proof, can be read without referring to the later sections. 
The proof of Theorem \ref{mainthm}, given in Section \ref{prThmB}, uses several lemmas and propositions from Sections \ref{S_WBH} through \ref{S_CWD}; these are proved by geometric arguments. 
In the first half of Section \ref{S_Proof_thms}, we prove Theorem \ref{thm_new} by combining ideas from the proofs of Theorems \ref{p-lemma} and \ref{mainthm}. 
In the second half of Section \ref{S_Proof_thms}, we prove Theorem \ref{DrcHst} using combinatorial descriptions of statistical behavior.

As a guide to the reader, a table of contents is provided at the beginning of the paper, and each subsection begins with a brief explanation of its contents. 
Appendix \ref{Ap_curvature} contains indispensable, though somewhat technical, differential-geometric estimates. 
An index is included at the end of the paper for the reader's convenience.

\section{Proof of Pluripotency lemma}\label{S_Pf_Pl}

The main aim of this section is to prove Pluripotency Lemma (Theorem \ref{p-lemma})  under the notations in Section \ref{S_Introduction}.
In addition we prove Proposition \ref{p_D_open}.

Let ${\mathbb{V}}_{0},{\mathbb{V}}_{1}$ be compact  subsets of ${\mathbb{U}}_{0}$ and ${\mathbb{U}}_{1}$ respectively such that, for any 
$g$ sufficiently close to $f$ in $\mathrm{Diff}^r(M)$, 
$\Lambda_g=\bigcap_{n\in\mathbb{Z}}g^n({\mathbb{V}}_{0}\cup {\mathbb{V}}_{1})$ 
is the continuation of $\Lambda$.
See Subsection \ref{ss_affine_model} for a practical example of such a compact set pair.

\begin{lem}\label{l_pluri_1}
Suppose that $f$ is strongly pluripotent for a subset $\Lambda'$ of a horseshoe $\Lambda$.
For any $x\in \Lambda'$, let $g$ be an element of $\mathrm{Diff}^r(M)$ arbitrarily $C^r$-close to $f$ and $D$ a subset of $M$ satisfying the conditions in (2) of Definition \ref{dfn:0624}.
Then, for $(v_i)_{i\in\mathbb{Z}}=\mathcal{I}(x)$,  
$$
\lim_{n\rightarrow \infty}
\frac{\#\bigl\{0\leq i\leq n-1;\, g^i(D)\subset {\mathbb{U}}_{v_i}\bigr\}}{n}=1
$$
holds.
\end{lem}
\begin{proof}
Let $x_g$ be the continuation of any element $x$ of $\Lambda'$.
We consider the positive number $d_{0}$ defined as
$$
d_0=\min\left\{\mathrm{dist}_M(\partial {\mathbb{V}}_0,\partial {\mathbb{U}}_0), \mathrm{dist}_M(\partial {\mathbb{V}}_1,\partial {\mathbb{U}}_1)\right\}.
$$
For $j=0,1$, let $(i_k^{(j)})_{k\geq 1}$ be the maximal sequence of strictly increasing non-negative integers with 
$g^{i_k^{(j)}}(x_g)\in {\mathbb{U}}_j$.
Take an arbitrarily small $\delta>0$.
For the proof, it suffices to show the following inequalities
\begin{equation}\label{eqn_2delta}
\begin{split}
&\frac{\#\bigl\{k\geq 1;\,0\leq i_k^{(0)}\leq n-1, g^{i_k^{(0)}}(D)\not\subset {\mathbb{U}}_0\bigr\}}{n}< \delta,\\
&\frac{\#\bigl\{k\geq 1;\,0\leq i_k^{(1)}\leq n-1, g^{i_k^{(1)}}(D)\not\subset {\mathbb{U}}_1\bigr\}}{n}< \delta
\end{split}
\end{equation}
hold for all sufficiently large $n$.
If the first inequality of \eqref{eqn_2delta} did not hold, then 
there would exist a strictly  
increasing sequence $(n_m)_{m\in \mathbb N}$ of positive integers satisfying 
$$
\frac{\#\bigl\{k\geq 1;\, 0\leq i_k\leq n_m-1, g^{i_k^{(0)}}(D)\not\subset {\mathbb{U}}_0\bigr\}}{n_m}\geq \delta.
$$
Since $\mathrm{dist}(g^{i_k^{(0)}}(y),g^{i_k^{(0)}}(x_g))\geq d_{0}$ for any $y \in D$ with $g^{i_k^{(0)}}(y)\not\in \mathbb{U}_{0}$, we have
\begin{align*}
\frac1{n_m}\sum_{i=0}^{n_m-1}\sup_{y\in D}\left\{\mathrm{dist}(g^i(y),g^i(x_g))\right\}&
\geq \frac1{n_m}\sum_{k=1}^{m'}\sup_{y\in D}\left\{\mathrm{dist}(g^{i_k^{(0)}}(y),g^{i_k^{(0)}}(x_g))\right\}\\
&\geq \frac{m'd_0}{n_m}\geq d_0\delta, 
\end{align*}
where $m'=\#\bigl\{0\leq i_k^{(0)}\leq n_m-1;\, g^{i_k^{(0)}}(D)\not\subset {\mathbb{U}}_0\bigr\}$.
This contradicts \eqref{eq:0625_s} and hence the first inequality of \eqref{eqn_2delta} holds.
The second inequality is proved quite similarly, so the proof is complete.
\end{proof}

\begin{lem}\label{l_beta_kL}
Suppose that $f$ is $\Sigma$-describable with respect to the intervals ${\mathbb{I}}_k=[\alpha_k,\alpha_k+\beta_k]\cap \mathbb{N}$ 
satisfying the conditions of Definition \ref{describable}.
Then one can suppose that, for any $L>0$, the following equation
\begin{equation}\label{eqn_beta_kL}
\lim_{n\rightarrow\infty}\frac{\#\{0\leq i\leq n-1;\, i\in {\mathbb{I}}_k \text{\rm\ with $\beta_k\geq L$}\}}{n}=1
\end{equation}
holds if necessary redefining ${\mathbb{I}}_k$'s.
\end{lem}
\begin{proof}
One can reconstruct the intervals ${\mathbb{I}}_k=[\alpha_k, \alpha_k+\beta_k]\cap \mathbb{Z}$ so that they satisfy the following conditions. 
\begin{itemize}
\setl
\item
$g^n(D)$ is contained in ${\mathbb{U}}_{v_n}$ if and only if $n$ is an element of some ${\mathbb{I}}_k$.
\item
$\alpha_k+\beta_k+2\leq \alpha_{k+1}$ for any $k$.
\end{itemize}
In the case when $\alpha_k+\beta_k+1=\alpha_{k+1}$, we consider the new interval $[\alpha_k,\alpha_{k+1}+\beta_{k+1}]\cap \mathbb{Z}$ instead of ${\mathbb{I}}_k\cup {\mathbb{I}}_{k+1}$.
From the construction, we know that the sequence of the new intervals, still denoted by $({\mathbb{I}}_k)$, satisfies (DEI) and (OCD).
Here we need to consider the following two cases.

\smallskip

\noindent{\bf Case 1.} $({\mathbb{I}}_k)$ consists of finitely many intervals. 
Then the last entry ${\mathbb{I}}_{k_0}$ is a half-open interval $[\alpha_{k_0},\infty)\cap \mathbb{Z}$.
 We split ${\mathbb{I}}_{k_0}$ into infinitely many intervals such that 
 ${\mathbb{I}}_{k_0}^{\mathrm{new}}=[\alpha_{k_0},\alpha_{k_0}+2]\cap \mathbb{Z}$ and 
 ${\mathbb{I}}_{k_0+i}^{\mathrm{new}}=[\alpha_{k_0}+2^{i+1},\alpha_{k_0}+2^{i+2}-2]\cap \mathbb{Z}$ for $i\geq 1$.
 It is not hard to see that the sequence of the new intervals satisfies \eqref{eqn_beta_kL}.
  
\smallskip

\noindent{\bf Case 2.} $({\mathbb{I}}_k)$ consists of infinitely many intervals. 
If \eqref{eqn_beta_kL} did not hold, then there would exist $\delta>0$ and a strictly increasing sequence $\{n_m\}$ 
of positive integers satisfying the following condition.
\begin{equation}\label{eqn_Nm}
\frac{\#\left\{0\leq i\leq n_m-1; i\in {\mathbb{I}}_k\,\ \text{for some $k$ with $\beta_k< L$}\right\}}{n_m}>\delta.
\end{equation}
Let $k_1<k_2<\cdots<k_p$ be the positive integers with $\beta_{k_j}<L$ and ${\mathbb{I}}_{k_j}\cap [0,n_m-1]\neq \emptyset$.
By \eqref{eqn_Nm}, we have $\dfrac{pL}{n_m}\geq \delta$ or equivalently $p\geq L^{-1}n_m\delta$.
Note that $[0,n_m-1]\setminus \bigcup_{k=1}^\infty[\alpha_k,\alpha_k+\beta_k]$ consists of at least $p-1$ connected components, each of which is 
either an open or half-open interval.
Since $\alpha_k+\beta_k+2\leq \alpha_{k+1}$, each of these intervals contains at least one positive integer.
It follows that 
$$
\liminf_{m\rightarrow \infty}
\frac{\#\left\{0\leq i\leq n_m-1; i\not\in \bigcup_{k=1}^\infty{\mathbb{I}}_k\right\}}{n_m}
\geq 
\lim_{m\rightarrow \infty}\frac{L^{-1}n_m\delta-1}{n_m}=L^{-1}\delta.$$
This contradicts that $({\mathbb{I}}_k)$ satisfies (DEI) and hence \eqref{eqn_beta_kL} holds.
\end{proof}

Now we are ready to prove Theorem \ref{p-lemma}.

\begin{proof}[Proof of Theorem \ref{p-lemma}]
Under the assumptions of Theorem \ref{p-lemma}, we suppose that $f$ is strongly pluripotent for $\Lambda'=\mathcal{I}^{-1}(\widehat\Sigma)$ and $g$ is an element of $\mathrm{Diff}^r(M)$ arbitrarily $C^r$-close $f$ and satisfying 
\eqref{eq:0625_s} for any $y\in D$.
By Lemma \ref{l_pluri_1}, 
$$
\lim_{n\rightarrow \infty}
\frac{\#\bigl\{0\leq i\leq n-1;\, g^i(D)\subset {\mathbb{U}}_{v_i}\bigr\}}{n}=1.
$$
Then one can construct a sequence $({\mathbb{I}}_k)_{k\in\mathbb{N}}$ satisfying (DEI) and (OCD) as in the proof of Lemma \ref{l_beta_kL}.
Thus $f$ is $\Sigma$-describable over $\Lambda$.

Conversely, we suppose that $f$ is $\Sigma$-describable over $\Lambda$. 
Fix $\underline v=(v_0v_1v_2\dots)\in \Sigma$ and choose $g$  arbitrarily $C^{r}$-close to $f$ such that $g$ satisfies the two conditions in Definition \ref{describable}  for increasing sequences $(\alpha _k)_{k\in \mathbb N}$, 
non-negative integers $\beta _k$ $(k\in \mathbb N)$ as in Lemma \ref{l_beta_kL} and a positive Lebesgue measure set $D$.
Let $\underline u=(\ldots u_{-3} u_{-2} u_{-1})\in \{0,1\}^{\mathbb{Z}_{(-)}}$, and denote 
$\mathcal I_g^{-1}(\underline u \underline v)$ by $x_g$ for simplicity. 
Fix $\varepsilon >0$ and $y\in D$ arbitrarily.

For a fixed positive integer $N$, consider any $\beta_i$ with $\beta_i\geq 2N+1$.
For any $0\le j\le \beta _i- 2N$, 
\begin{equation}\label{eq:0808b}
g^{\alpha _i + N+ j}(D \cup \{x_g\})\subset \bigcap _{k=- N}^{N} g^{-k}(\mathbb U_{v_{\alpha _i +N+j+k}})
\end{equation}
holds.
Indeed, it follows from  the choice of $j$ that 
\begin{equation}\label{eq:0808}
\alpha _i \le \alpha _i + N+ j+k \le \alpha _i+\beta _i
\end{equation}
if $-N\le k\le N$, so that $g^{\alpha _i + N+ j+k}(D)\subset \mathbb U_{v_{\alpha _i +N+j+k}}$ by 
(OCD) of Definition \ref{describable}. 
On the other hand, 
$$\{x_g\} = \bigcap _{n\ge 0}g^{-n} (\mathbb U_{v_{n}}) \cap \bigcap _{n<0} g^{-n} (\mathbb U_{u_{n}})$$
because $x_g=\mathcal I_g^{-1}(\underline u \underline v)$, 
so that $g^n(x_g)\in \mathbb U_{v_{n}}$ for all $n\ge 0$.
In particular, $g^{\alpha _i + N +j +k}(x_g)\in  \mathbb U_{v_{\alpha _i +N+j+k}}$ for any $-N\le k\le N$ because $\alpha _i +N+j+k \ge \alpha _i\ge 0$ by \eqref{eq:0808}. 
 That is, we have \eqref{eq:0808b}.  
Hence, since 
\[
 \lim _{N\rightarrow \infty} \sup _{(w_k)_{k\in \mathbb Z}\in \{0,1\}^{\mathbb Z}} \mathrm{diam}\left(\bigcap _{k=- N}^{N} g^{-k}(\mathbb U_{w_{k}})\right)  =0,
\]
 one can find $N\in \mathbb N$ such that for any $\beta_k$ with $\beta _k \ge  2N+1$ 
 and 
 any $j\in \mathbb I_k':=[\alpha_k +N+1, \alpha _k +\beta _k -N]\cap \mathbb{Z}$,
\begin{equation}\label{eq:1103}
\mathrm{diam}(g^j(D \cup \{ x_g\})) \leq \varepsilon.
\end{equation}

Let $N_0$ be the smallest integer with $N_0\geq \dfrac{4N\mathrm{diam}(M)}{\varepsilon}$ 
and denote by $({\mathbb{I}}_{k_a})_{a\in\mathbb{N}}$ the subsequence of $({\mathbb{I}}_k)_{k\in\mathbb{N}}$ consisting of all 
intervals $[\alpha_{k_a},\alpha_{k_a}+\beta_{k_a}]\cap \mathbb{Z}$ with $\beta_{k_a}\geq 2N+N_0$.
We set simply ${\mathbb{I}}_{k_a}={\mathbb{I}}_{(a)}$ and ${\mathbb{I}}_{k_a}'={\mathbb{I}}_{(a)}'$ 
and consider the splitting of   
$\dfrac{1}{n}\sum _{j=0}^{n-1} \sup_{y\in D}\left\{\mathrm{dist}(g^j(y),g^j(x_g))\right\}$ 
as follows.
\begin{align*}
\frac{1}{n}\sum _{j=0}^{n-1} \sup_{y\in D} \{ \mathrm{dist}(g^j(y),& g^j(x_g)) \} 
=  \frac{1}{n}\sum _{j\in [0,n-1] \cap ( \bigcup_a\mathbb I_{(a)}')} \sup_{y\in D} \{ \mathrm{dist}(g^j(y),g^j(x_g)) \}\\
&+\frac{1}{n}\sum _{j\in [0,n-1] \cap ( \bigcup_a\mathbb I_{(a)}\setminus \mathbb I_{(a)}')} 
\sup_{y\in D}\{\mathrm{dist}(g^j(y),g^j(x_g))\}\\
&+\frac{1}{n}\sum _{j\in [0,n-1] \cap (\mathbb{N}_{0}\setminus \bigcup_{a}\mathbb{I}_{(a)})} 
\sup_{y\in D}\{\mathrm{dist}(g^j(y),g^j(x_g))\}.
\end{align*}
By \eqref{eq:1103}, the first term of the right-hand side is bounded by $\varepsilon$.
Let $a_0$ be the greatest integer among $a\in \mathbb{N}$ with ${\mathbb{I}}_{(a)}\cap [0,n-1]\neq \emptyset$.
Since $\beta_{k_a}\geq 2N+N_0$, $a_0-1\leq \dfrac{\sum_{a=1}^{a_0-1}\beta_{k_a}}{2N+N_0}$.
Since moreover $\sum_{a=1}^{a_0-1}\beta_{k_a}<n$, 
the second term is bounded by 
\begin{align*}
\frac{2Na_0}{n}\,\mathrm{diam}(M)&\leq 
\frac{2N}{n}\left(\frac{\sum_{a=1}^{a_0-1}\beta_{k_a}}{2N+N_0}+1\right)\mathrm{diam}(M)\\
&< 
2N\left(\frac1{2N+N_0}+\frac1{n}\right)\mathrm{diam}(M)\\
&\leq \frac{4N}{2N+N_0}\,\mathrm{diam}(M)
< \frac{\varepsilon}{\mathrm{diam}(M)}\,\mathrm{diam}(M)=\varepsilon
\end{align*}
for any $n\geq 2N+N_0$.
By \eqref{eqn_beta_kL}, 
there exists $n_0$ such that the third term is bounded 
by $\varepsilon$ for any $n\geq n_0$.
It follows that
$$
\limsup_{n\rightarrow\infty} 
\frac1{n} \sum_{j=0}^{n-1}\sup _{y\in D}\{\mathrm{dist}\left(g^j(y),g^j(x_g)\right)\}< 3\varepsilon.
$$
Since $\varepsilon$ is arbitrary, $f$ is strongly pluripotent for $\Lambda'$.
This completes the proof of Theorem \ref{p-lemma}.
\end{proof}

As it is seen in the proof, actually we have shown a conclusion stronger than \eqref{eq:0625_s}. 
More precisely, 
since  
$g^n({\mathcal I_g^{-1}\bigl(\underline u\underline v)})$ and $g^n(x_g)$ are contained in the same ${\mathbb{U}}_{v_n}$ 
for any $\underline u \in \{0,1\}^{ \mathbb{Z}_{(-)}}$ and $n\geq 0$, it follows that 
\[
\lim _{n\rightarrow \infty} \sup _{\underline u \in \{0,1\}^{ \mathbb{Z}_{(-)}} }  
\frac1{n} \sum_{j=0}^{n-1}\sup_{ y\in D}\left\{\mathrm{dist}\left(g^j(y),g^j({\mathcal I_g^{-1}\bigl(\underline u\underline v)}\bigr)\right)\right\}=0,
\]
where $\mathcal I_g$ is the coding map of $g$ corresponding to $\mathcal I$ for $f$.

\bigskip

\begin{proof}[Proof of Proposition \ref{p_D_open}]
Suppose the contrary that $D\cap \Lambda_g$ would contain an element $y$.
Since $D$ is an open set, one can choose $\varepsilon>0$ sufficiently small so that the $\varepsilon$-neighborhood $O_\varepsilon(y)$ of $y$ in $M$ is contained in $D$.
We may assume that there exists a strictly increasing sequence $(n_m)_{m\in\mathbb{N}}$ of positive integers 
such that 
$$\#\{0\leq j\leq n_m\,;\,g^j(x_g)\in {\mathbb{V}}_0\}\geq \frac{n_m}2$$ 
if necessary replacing ${\mathbb{V}}_0$ with ${\mathbb{V}}_1$.
We set $\mathcal{I}_g(y)=(v_n)_{n\in \mathbb{Z}}$.
Then there exists a positive integer $n_0$ such that 
$$\mathcal{I}_g^{-1}
\bigl\{(v_n')_{n\in\mathbb{Z}}\,;\, v_n'=v_n \text{ for $|n|\leq n_0$}
\bigr\}\subset O_\varepsilon(y).$$
In particular, for $\underline w=(w_n)_{n\in \mathbb{Z}}$ with $w_n=v_n$ for $n\leq n_0$ and $w_n=1$ for $n\geq n_0+1$, 
$z=\mathcal{I}_g^{-1}(\underline w)$ is an element of $\Lambda_g$ contained in $O_\varepsilon(y)\subset D$.

Since $\mathrm{dist}(\mathbb{V}_0,\mathbb{V}_1)>0$, there exist a Lipschitz 
map $\varphi:M\longrightarrow [-1,1]$ and a constant $0<L\leq 1$ with 
$\mathrm{Lip}(\varphi)\leq 1$, $\varphi(M)\subset [0,L]$ and such that
$\varphi(x)=L$ for $x\in \mathbb{V}_0$ and $\varphi(x)=0$ for $x\in \mathbb{V}_1$. 
Then we have 
\begin{align*}
\limsup_{m\rightarrow \infty}d_W(\delta_{z,g}^{n_m}, \delta_{x_g,g}^{n_m})&\geq 
\limsup_{m\rightarrow \infty}\frac1{n_m}\left|\,\sum_{j=0}^{n_m-1}(\varphi\circ g^j(z)-\varphi\circ g^j(x_g))\right|\\
&\geq \lim_{m\rightarrow \infty}\frac1{n_m}\left(\frac{n_m}2-n_0\right)L=\frac{L}2.
\end{align*}
This contradicts \eqref{eq:0625}.
Thus we have $D\cap \Lambda_g=\emptyset$.
\end{proof}

\section{Wild blender-horseshoes}\label{S_WBH}

For simplicity, in this section, we only consider the case of $n=3$ in Theorem \ref{mainthm}.
So one can suppose that the manifold $M$ has a coordinate neighborhood which is identified with 
the sub-space $(-1,2)^3$ of $\mathbb{R}^3$.
We will see in Section \ref{S_Proof_thms} that  
our arguments here still hold in the case of $n>3$ 
for certain elements $f_0$ of $\mathrm{Diff}^r(M)$ having a horseshoe 
$\Lambda$ with $\dim W^{\mathrm{u}}(\Lambda)=\dim W^{\mathrm{cs}}(\Lambda)=1$ and $\dim W^{\mathrm{ss}}(\Lambda)=n-2$.

\subsection{A non-hyperbolic affine model with asymmetricity condition}\label{ss_affine_model}

In this subsection, we define a non-hyperbolic diffeomorphism $f_0$ which is similar to that given in \cite{KNS23}.
The open set $\mathcal{O}$ in the theorem is a small $C^r$-open neighborhood $\mathcal{O}(f_0)$ 
of $f_0$.

Let $\lambda_{\rm ss}, \lambda_{\rm cs0}, \lambda_{\rm cs1}$ and $\lambda_{\rm u}$ 
be real positive constants
with
\begin{subequations}
\begin{equation}\label{eqn_eigen_v}
\lambda_{\rm ss}<\lambda_{\rm cs0}<1/2<\lambda_{\rm cs1}<1<\lambda_{\rm cs0}+\lambda_{\rm cs1},\quad 2<\lambda_{\rm u}.
\index{lambda_ss@$\lambda_{\rm ss},\lambda_{\rm cs0},\lambda_{\rm cs1}$, 
$\lambda_{\rm cs0}, \lambda_{\rm cs1},\lambda_{\rm u}$}
\end{equation}
Moreover, we suppose that   
\begin{equation}\label{eqn_pdc}
\lambda_{\rm cs0}\lambda_{\rm cs1}\lambda_{\rm u}^{2}<1,
\end{equation}
which corresponds to the partially dissipative condition for $f_0$.
We fix a sufficiently small  positive number $\varepsilon_0$.
In particular, we may suppose that
\begin{equation}\label{eqn_lam1ev}
\lambda_{\mathrm{cs} 1}(1+\varepsilon_0)<1.
\end{equation}
\end{subequations}
Consider the 3-dimensional block 
 $\mathbb{B}\index{B@$\mathbb{B}$}=I_{\varepsilon_{0}}^{3}$ in $M$, where
$$
I_{\varepsilon_{0}}\index{Ie0@$I_{\varepsilon_0}$}=[-\varepsilon_{0},1+\varepsilon_{0}],
$$
and  the 
vertical sub-blocks of $\mathbb{B}$ defined as 
$$
\mathbb{V}_{0}= [-\varepsilon_{0}, \lambda_{\rm u}^{-1}+\varepsilon_0]\times I_{\varepsilon_{0}}^{2},\quad 
\mathbb{V}_{1}= [1-\lambda_{\rm u}^{-1}-\varepsilon_0,1+\varepsilon_{0}]\times I_{\varepsilon_{0}}^{2}. 
$$
\index{V01@$\mathbb{V}_0$, $\mathbb{V}_1$}
\begin{subequations}
Let $\index{f0@$f_0$}f_{0}$ be a 3-dimensional diffeomorphism such that $f_{0}|_{\mathbb{V}_{0}\cup \mathbb{V}_{1}}$ is defined as
\begin{equation}\label{eqn_f0V}
f_{0}(x,y,z)=
\begin{cases}
(\lambda_{\mathrm{u}}x,\lambda_{\mathrm{ss}}y,\zeta_0(z))&\text{if}\ (x,y,z)\in \mathbb{V}_{0},\\
(\lambda_{\mathrm{u}}(1-x),-\lambda_{\mathrm{ss}}y+1,\zeta_1(z)) &\text{if}\ (x,y,z)\in \mathbb{V}_{1},
\end{cases}
\end{equation} 
where 
$\zeta_0$ and $\zeta_1$ are the affine maps on $I_{\varepsilon_0}$ given by
\begin{equation}\label{eqn_def_zeta}
\zeta_0(z)=\lambda_{\mathrm{cs} 0}z\quad\text{and}\quad \zeta_1(z)=\lambda_{\mathrm{cs} 1}z+1-\lambda_{\mathrm{cs} 1}.
\end{equation}
\end{subequations}
See Figure \ref{f_3_1}.
\begin{figure}[hbtp]
\centering
\scalebox{0.6}{\includegraphics[clip]{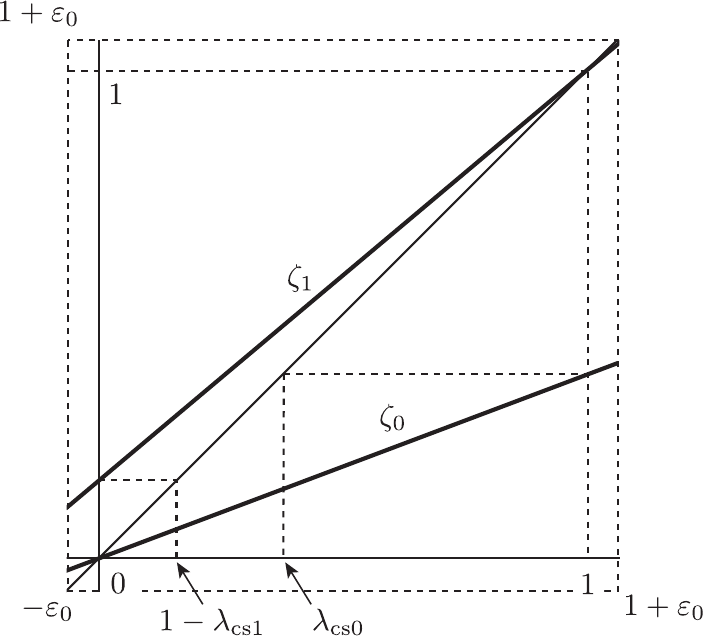}}
\caption{}
\label{f_3_1}
\end{figure}
From our setting, $f_{0}$ has the uniformly hyperbolic set  
\[\Lambda_{f_{0}}=\bigcap_{n\in \mathbb{Z}} f_{0}^n(\mathbb{V}_{0}\cup \mathbb{V}_{1})\]
\index{Lf0@$\Lambda_{f_{0}}$}
which belongs to the class of blender-horseshoes, see \cite{BD96,BD12} for details.

\begin{rmk}[Asymmetricity condition]\label{rmk2}
The inequalities 
\eqref{eqn_eigen_v} give asymmetric contractions for $f_{0}$ along the centre-stable direction for the blender-horseshoe.
This asymmetricity is unnecessary for general cases, but it is essential in our paper, 
which is used in the proof of Lemma \ref{lem7.3}.
\end{rmk}

In addition, we suppose another condition to obtain a non-hyperbolic situation.
Let $S_{1/2}\index{S12@$S_{1/2}$}$ be the $x=1/2$ section of $\mathbb{B}$ and $\mathbb{H}_{\varepsilon_0}\index{He0@$\mathbb{H}_{\varepsilon_0}$}$ the 
$\varepsilon_0$-neighborhood of $S_{1/2}$ in $\mathbb{B}$, that is, 
$\mathbb{H}_{\varepsilon_0}=[1/2-\varepsilon_0,1/2+\varepsilon_0]\times I_{\varepsilon_0}^2$.
For any $(x,y,z)\in \mathbb{H}_{\varepsilon_0}$, 
we suppose that   
\begin{equation}\label{eqn_tang}
f_{0}^{2}(x,y,z)=\left(-a_{1}\Bigl(x-\frac{1}{2}\Bigr)^{2}+a_{2}z+\mu,\ a_{3}\Bigl(y-\frac{1}{2}\Bigr)+\frac{1}{2},\ 
a_{4}\Bigl(x-\frac{1}{2}\Bigr)+\frac{1}{2} \right),
\end{equation}
where $a_{1},a_{2},a_{3},a_{4}$ are real constants 
 with 
\begin{equation}\label{eqn_a_1_4}
a_{1}>0,\ 
|a_{3}|<1-2\lambda_{\rm ss}\quad\text{and}\quad a_2a_3a_4<0.
\end{equation} 
The second condition assures that $f_{0}^{2}(\mathbb{H}_{\varepsilon_0})$ lies 
between $f_{0}(\mathbb{V}_{0})$ and $f_{0}(\mathbb{V}_{1})$. 
The third means that $f_0^2|_{\mathbb{H}_{\varepsilon_0}}$ is orientation preserving.
The constant $\mu$ is taken so that $f_0^2(\mathbb{H}_{\varepsilon_0})$ is contained in $(0,\lambda_{\mathrm{u}}^{-1})\times I_{\varepsilon_0}^2$.
See Figure \ref{f_3_2}.
\begin{figure}[hbtp]
\centering
\scalebox{0.7}{
\includegraphics[clip]{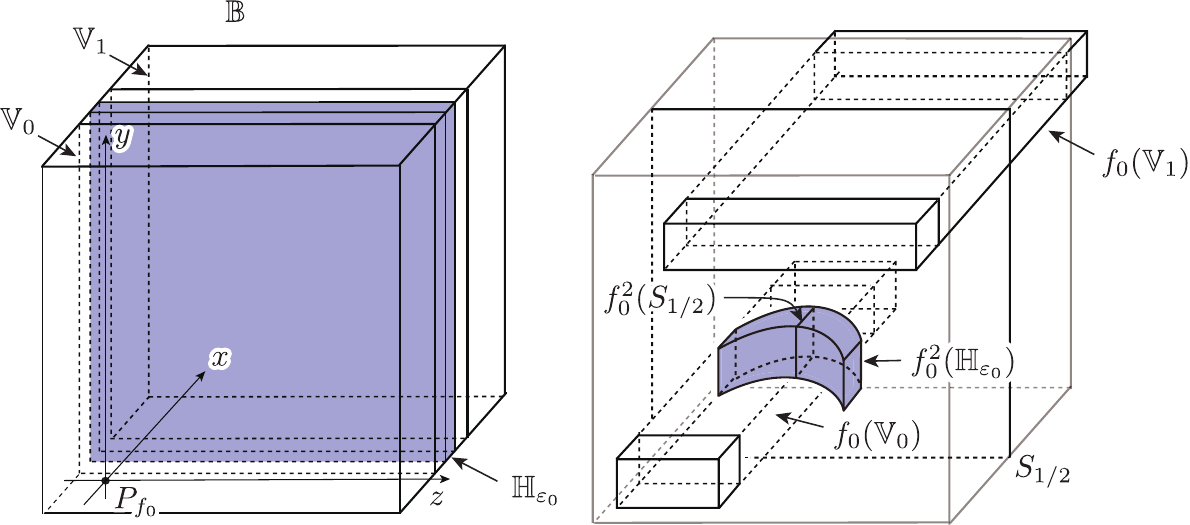}
}
\caption{ } 
\label{f_3_2}
\end{figure}
For example, in the case of $a_2>0$, the condition is equivalent to that $\mu$ satisfies 
$$a_1\varepsilon_0^2+a_2\varepsilon_0<\mu<\lambda_{\mathrm{u}}^{-1}-a_2(1+\varepsilon_0).$$
Hence there exists such a $\mu$ for any sufficiently small $\varepsilon_0>0$.

As in \cite{KNS23}, the diffeomorphism $f_{0}$ of \eqref{eqn_tang}  
gives to a $C^{1}$-robust homoclinic tangency associated with $\Lambda_{f_0}$. 
Such a blender-horseshoe with robust homoclinic tangency 
is called a \emph{wild blender-horseshoe}.

\subsection{Invariant cone-fields and stable and unstable foliations}\label{ss_cone_foliation}
One can choose the neighborhood $\index{Of0@$\mathcal{O}(f_0)$}\mathcal{O}(f_0)$ of $f_0$ in $\mathrm{Diff}^r(M)$ so that, for any $f\in \mathcal{O}(f_0)$, 
$f(\mathbb{V}_0\cup \mathbb{V}_1)\cap (\partial_y\mathbb{B}\cup \partial_z\mathbb{B})=\emptyset$, 
where $\partial_y\mathbb{B}=I_{\varepsilon_0}\times \{-\varepsilon_0,1+\varepsilon_0\}\times I_{\varepsilon_0}$ and 
$\partial_z\mathbb{B}=I_{\varepsilon_0}^2\times \{-\varepsilon_0,1+\varepsilon_0\}$.
For any $\varepsilon>0$ smaller than $\varepsilon_0$, we consider an open neighborhood $\mathcal{O}_{\varepsilon}$ of $f_0$ in $\mathrm{Diff}^r(M)$ 
such that the closure  $\overline{\mathcal{O}_{\varepsilon}}$ is contained in $\mathcal{O}(f_0)$, 
$\mathcal{O}_{\varepsilon}\subset \mathcal{O}_{\varepsilon'}$ if $\varepsilon<\varepsilon'<\varepsilon_0$ and $\bigcap_{0<\varepsilon<\varepsilon_0}\overline{\mathcal{O}_\varepsilon}=\{f_0\}$.
Here we note that $\mathcal{O}_{\varepsilon}$ is in general smaller than the $\varepsilon$-neighborhood of 
$f_0$ in $\mathrm{Diff}^r(M)$ with respect to the $C^r$-metric.
Real numbers $a_g$ depending on $g\in \mathcal{O}(f_0)$ are denoted by $O(\varepsilon)$ if there exists a constant $C>0$ independent of $\varepsilon$ 
and satisfying $|a_g|\leq C\varepsilon$ for any $g\in \mathcal{O}_{\varepsilon}$.
For functions  $a_{g,t}$ $(t\in T)$ defined on a compact subset $A$ of $\mathbb{B}$ with a compact parameter space $T$, 
$a_{g,t}=O(\varepsilon)$ means that $\max\{|a_{g,t}|\,;\, \boldsymbol{x}\in A,t\in T\}=O(\varepsilon)$.

If necessary replacing $\mathcal{O}(f_0)$ with $\mathcal{O}_{\varepsilon}$ for a sufficiently small $\varepsilon>0$, 
we may assume that any $f\in \mathcal{O}(f_0)$ is sufficiently $C^r$-close to $f_0$ in $\mathrm{Diff}^r(M)$.
However, it does not always mean that $f^n$ is close to $f_0^n$ for integers $n$ with large 
absolute value $|n|$.
To overcome the difficulty, we will employ the following $\mathrm{u}$, $\mathrm{ss}$, $\mathrm{cs}$-cone-fields on $\mathbb{B}$.
\begin{align*}
\boldsymbol{C}_{\varepsilon}^{\mathrm{u}}(\boldsymbol{x})&=
\left\{
\boldsymbol{v}=(v^{\mathrm{u}},v^{\mathrm{s}},v^{\mathrm{cs}})\in T_{\boldsymbol{x}}(\mathbb{B});\,
\sqrt{(v^{\mathrm{s}})^2+(v^{\mathrm{cs}})^2}\leq \varepsilon |v^{\mathrm{u}}|
\right\},\\
\boldsymbol{C}_{\varepsilon}^{\mathrm{ss}}(\boldsymbol{x})&=
\left\{
\boldsymbol{v}=(v^{\mathrm{u}},v^{\mathrm{s}},v^{\mathrm{cs}})\in T_{\boldsymbol{x}}(\mathbb{B});\,
\sqrt{(v^{\mathrm{u}})^2+(v^{\mathrm{cs}})^2}\leq\varepsilon |v^{\mathrm{s}}|
\right\},\\
\boldsymbol{C}_{\varepsilon}^{\mathrm{cs}}(\boldsymbol{x})&=
\left\{
\boldsymbol{v}=(v^{\mathrm{u}},v^{\mathrm{s}},v^{\mathrm{cs}})\in T_{\boldsymbol{x}}(\mathbb{B});\,
|v^{\mathrm{u}}|\leq\varepsilon\sqrt{(v^{\mathrm{s}})^2+(v^{\mathrm{cs}})^2}
\right\}
\index{Ceu@$\boldsymbol{C}_{\varepsilon}^{\mathrm{u}}(\boldsymbol{x})$, $\boldsymbol{C}_{\varepsilon}^{\mathrm{ss}}(\boldsymbol{x})$, $\boldsymbol{C}_{\varepsilon}^{\mathrm{cs}}(\boldsymbol{x})$}
\end{align*}
for $\boldsymbol{x}\in\mathbb{B}$.
We say that a $C^1$-surface $F$ in $\mathbb{B}$ is \emph{adaptable} to $\boldsymbol{C}_{\varepsilon}^{\mathrm{cs}}$ if, for any $\boldsymbol{x}\in F$, 
the tangent plane $T_{\boldsymbol{x}} F$ is contained in $\boldsymbol{C}_{\varepsilon}^{\mathrm{cs}}(\boldsymbol{x})$.
Similarly a $C^1$-arc $\alpha$ in $\mathbb{B}$ is \emph{adaptable} to $\boldsymbol{C}_{\varepsilon}^{\mathrm{u}}$ (resp.\ $\boldsymbol{C}_{\varepsilon}^{\mathrm{ss}}$) if, 
for any $\boldsymbol{x}\in \alpha$, $T_{\boldsymbol{x}}\alpha$ is contained in  $\boldsymbol{C}_{\varepsilon}^{\mathrm{u}}(\boldsymbol{x})$ (resp.\ in $\boldsymbol{C}_{\varepsilon}^{\mathrm{ss}}(\boldsymbol{x})$).

One can suppose that, for any $f\in\mathcal{O}(f_0)$, these cone-fields are $f$-\emph{invariant}.
This means that
\begin{subequations}
\begin{equation}\label{eqn_fCu}
Df(\boldsymbol{x})(\boldsymbol{C}_{\varepsilon}^{\mathrm{u}}(\boldsymbol{x}))\subset \boldsymbol{C}_{\varepsilon}^{\mathrm{u}}(f(\boldsymbol{x}))
\end{equation}
for any $\boldsymbol{x}\in\mathbb{B}\cap f^{-1}(\mathbb{B})$, and 
\begin{equation}\label{eqn_fcscs}
Df^{-1}(\boldsymbol{x})(\boldsymbol{C}_{\varepsilon}^{\mathrm{ss}}(\boldsymbol{x}))\subset \boldsymbol{C}_{\varepsilon}^{\mathrm{ss}}(f^{-1}(\boldsymbol{x})),\quad
Df^{-1}(\boldsymbol{x})(\boldsymbol{C}_{\varepsilon}^{\mathrm{cs}}(\boldsymbol{x}))\subset \boldsymbol{C}_{\varepsilon}^{\mathrm{cs}}(f^{-1}(\boldsymbol{x}))
\end{equation}
\end{subequations}
for any $\boldsymbol{x}\in\mathbb{B}\cap f(\mathbb{B})$.
We know that, for any $f\in\mathcal{O}(f_0)$, 
there exists the blender horseshoe $\Lambda_f$\index{Lf@$\Lambda_{f}$} for $f$ which is the continuation of $\Lambda_{f_0}$.
For any $\boldsymbol{x}\in \Lambda_f$, we define by $\lumfd(\boldsymbol{x})$ the component of 
$W^{\mathrm{u}}(\boldsymbol{x})\cap \mathbb{B}$ containing $\boldsymbol{x}$ and fix the local unstable manifold of 
$\Lambda_f$ by $\lsmfd(\Lambda_{f})=\bigcup_{\boldsymbol{x}\in \Lambda_f}\lumfd(\boldsymbol{x})
\index{W_lusmfd@$W_{\mathrm{loc}}^{\mathrm{u}}(\Lambda_f)$, $W_{\mathrm{loc}}^{\mathrm{s}}(\Lambda_f)$}$, 
and the local stable manifold $W_{\mathrm{loc}}^{\mathrm{s}}(\Lambda_f)$ is fixed similarly.
Then any components of $W_{\mathrm{loc}}^{\mathrm{u}}(\Lambda_f)$ and $W_{\mathrm{loc}}^{\mathrm{s}}(\Lambda_f)$ 
are proper one and two dimensional submanifolds of $\mathbb{B}$ respectively.

\medskip

By Turaev \cite[Theorem 1]{Tur24}, 
we have a  
$C^{1}$ stable foliation $\mathcal{F}^{\rm s}_{f}\index{Ffs@$\mathcal{F}^{\rm s}_{f}$}$ on $\mathbb{B}$ 
which is compatible with $W_{\mathrm{loc}}^{\rm s}(\Lambda_{f})$ and satisfying the following conditions.
\begin{enumerate}[({F}1)]
\makeatletter
\renewcommand{\p@enumi}{F}
\makeatother
\item
Each leaf of $\mathcal{F}_f^{\mathrm{s}}$ is a $C^r$-surface in $\mathbb{B}$.\label{F1}
\item
The restriction $\mathcal{F}_f^{\mathrm{s}}|_{\mathbb{H}_{\varepsilon_0}}$ consists of flat leaves parallel to the $yz$-plane.\label{F2}
\item
Any leaf of $\mathcal{F}_f^{\mathrm{s}}$ is adaptable to $\boldsymbol{C}_{\varepsilon}^{\mathrm{cs}}$.\label{F3}
\end{enumerate} 
By \eqref{F3}, $f^2(S_{1/2})$ meets leaves of $\mathcal{F}_f^{\mathrm{s}}|_{\mathbb{V}_{0,f}}$ $O(\varepsilon)$-almost orthogonally, that is, 
the intersection angle is $\pi/2+O(\varepsilon)$, where $\mathbb{V}_{i,f}$\index{V01f@$\mathbb{V}_{0,f}$, $\mathbb{V}_{1,f}$}  $(i=0,1)$ is the component of $\mathbb{B}\cap f^{-1}(\mathbb{B})$ contained in $\mathbb{V}_i$.
Moreover, by Proposition \ref{p_curvature_plane}, one can choose $\mathcal{F}_f^{\mathrm{s}}$ so that, 
for any leaf of $F$ of $\mathcal{F}_f^{\mathrm{s}}$ and any unit vector $\boldsymbol{u}$ tangent to $F$ 
at a point $\boldsymbol{x}$, the absolute value $|\kappa_{\boldsymbol{u}}(\boldsymbol{x})|$  of 
the normal curvature is $O(\varepsilon)$.

\subsection{U-bridges}\label{ss_u-bridge}
For any element $f$ of $\mathcal{O}(f_0)$, we may assume that 
then the continuation $\Lambda_f$ of $\Lambda_{f_0}$ is also a wild blender-horseshoe.
We fix a maximal segment in $\mathbb{B}$ parallel to the $x$-axis, which is  
naturally identified with $I_{\varepsilon_0}$.
Then $\varGamma_f^{\mathrm{u}}=\bigcap_{i=0}^\infty f^{-i}(\mathbb{B})\cap I_{\varepsilon_0}$ is a Cantor set in $I_{\varepsilon_0}$.
Let $B^{\mathrm{u}}(0)$ and $B^{\mathrm{u}}(1)$ be the smallest sub-intervals of $I_{\varepsilon_0}$ containing 
$\varGamma_f^{\mathrm{u}}\cap [-\varepsilon_0,1/2]$ and $\varGamma_f^{\mathrm{u}}\cap [1/2,1+\varepsilon_0]$ 
respectively.
Consider the $C^1$-projection
$$\pi^{\rm u}_{f}:\mathbb{B}\longrightarrow
I_{\varepsilon_0}
\index{piuf@$\pi^{\rm u}_{f}$}
$$
along leaves of $\mathcal{F}_f^{\mathrm{s}}$.
We set $\mathbb{B}^{\mathrm{u}}(0)=(\pi_f^{\mathrm{u}})^{-1}(B^{\mathrm{u}}(0))$ and  $\mathbb{B}^{\mathrm{u}}(1)=(\pi_f^{\mathrm{u}})^{-1}(B^{\mathrm{u}}(1))$.
Note that $\mathbb{B}^{\mathrm{u}}(i)$ is contained in $\mathbb{V}_{i,f}$ for $i=0,1$.
For any integer $n\geq 1$, let  
$\underline{w}^{(n)}$ be a binary code of $n$ entries, that is, $\underline{w}^{(n)}=w_{1}\dots w_{n}\in \{0,1\}^{n}$, 
and let 
\begin{equation}\label{eqn_BBu}
\mathbb{B}^{\rm u}(\underline{w}^{(n)})\index{Buwn@$\mathbb{B}^{\mathrm u}(\underline{w}^{(n)})$}
=\left\{\boldsymbol{x}\in \mathbb{B}\,;\, f^{i-1}(\boldsymbol{x})\in \mathbb{B}^{\mathrm{u}}(w_i),
i=1,\dots,n\right\},  
\end{equation}
which is  called the \emph{u-bridge block} with the code $\underline{w}^{(n)}$. 
If it is necessary to specify the diffeomorphism $f$ concerning the u-bridge block, 
we may write $\mathbb{B}_{f}^{\rm u}(\underline{w}^{(n)})$.
Observe that, for any $n$, the family  
$\bigl(\mathbb{B}^{\rm u}(\underline{w}^{(n)})\bigr)_{\underline{w}^{(n)}\in \{0,1\}^{n}}$ 
consists of $2^{n}$ mutually disjoint 3-dimensional blocks. 
Then we say that the sub-interval  
\begin{equation}\label{eqn_BuBu}
B^{\mathrm{u}}(\underline{w}^{(n)})\index{Buwn@$B^{\mathrm{u}}(\underline{w}^{(n)})$}=\mathbb{B}^{\mathrm{u}}(\underline{w}^{(n)})\cap I_{\varepsilon_0}
=\pi_f^{\mathrm{u}}(\mathbb{B}^{\mathrm{u}}(\underline{w}^{(n)}))
\end{equation}
of $I_{\varepsilon_0}$ is the \emph{u-bridge} associated with the code $\underline{w}^{(n)}$.
The length $n=|\underline{w}^{(n)}|$ of $\underline{w}^{(n)}$ is called the \emph{generation} of $B^{\rm u}(\underline{w}^{(n)})$.

Now we define the subfamilies  
$\bigl(B_{k}^{\rm u}\bigr)_{k\geq 1}$ and 
$\bigl(\widetilde B_{k}^{\rm u}\bigr)_{k\geq 0}$
of $\bigl(B^{\rm u}(\underline{w}^{(n)})\bigr)_{n\geq 0, \underline{w}^{(n)}\in\{0,1\}^{n}}$ 
for any $f\in \mathcal{O}$ as follows.
First we choose $\mu$ in \eqref{eqn_tang} so that 
\begin{equation}\label{eqn_B0u}
\widetilde B_0^{\rm u}=B^{\rm u}(\underline{\widetilde w}^{(n_0)})\subset 
\pi^{\rm u}_{f}\circ f^2(S_{1/2}).
\end{equation}
for any $f\in \mathcal{O}(f_0)$ and some binary code $\widetilde{\underline{w}}^{(n_0)}$ of finite length $n_0>0$.
See Figure \ref{f_3_3}.
\begin{figure}[hbtp]
\centering
\scalebox{0.6}{\includegraphics[clip]{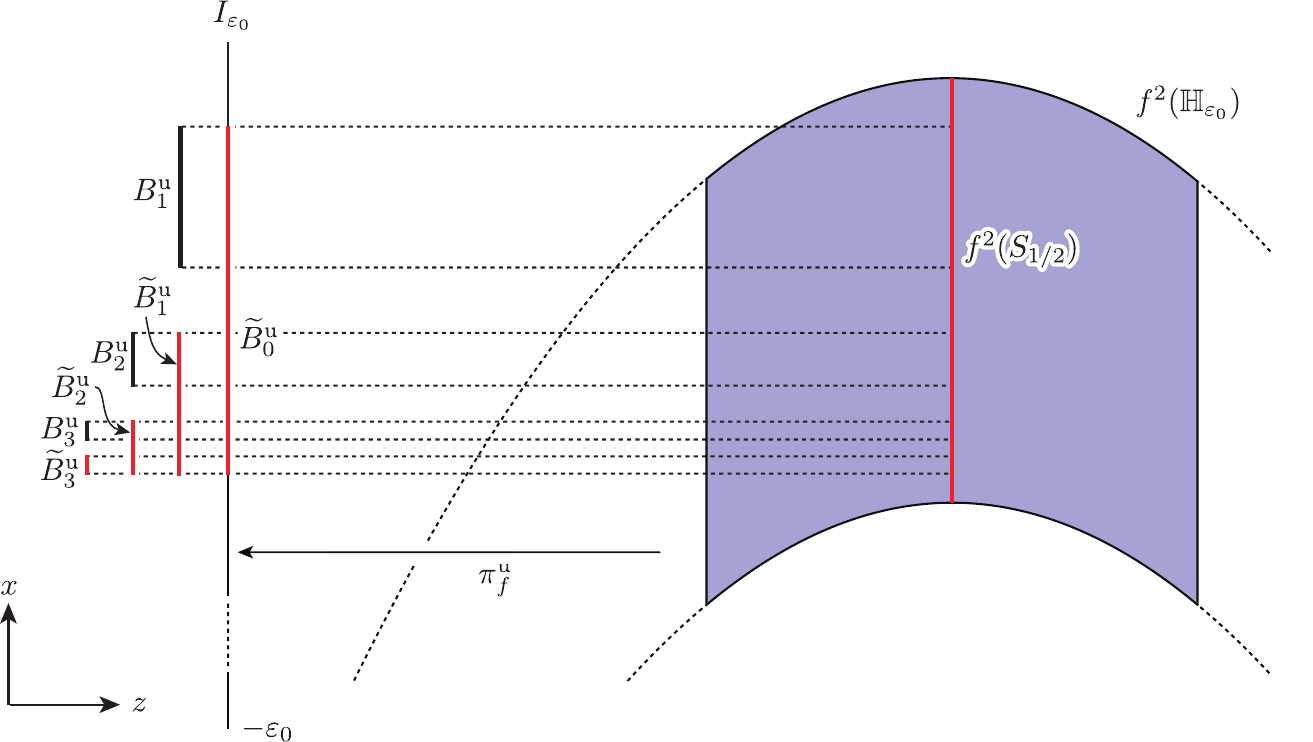}}
\caption{View from the top.}
\label{f_3_3}
\end{figure}
For each integer $k\geq 1$, 
we inductively define
the maximal sub-bridges
$B_k^{\rm u}$ and $\widetilde B_k^{\rm u}$\index{Bku@$\widetilde B_k^{\rm u}$, $B_k^{\rm u}$} of ${\widetilde B}_{k-1}^{\rm u}$ by
\begin{equation}\label{eqn_BBk}
\begin{split}
\widetilde B_k^{\rm u}&=B^{\rm u}(\underline{\widetilde w}^{(n_0+k-1)}\widetilde\alpha_k)
=B^{\rm u}(\underline{\widetilde w}^{(n_0+k)}),\\
B_k^{\rm u}
&=B^{\rm u}(\underline{\widetilde w}^{(n_0+k-1)}\alpha_k)
=B^{\rm u}(\underline{ w}^{(n_0+k)}),
\end{split}
\index{wn0k@$\underline{\widetilde w}^{(n_0+k)}$, $\underline{w}^{(n_0+k)}$}
\end{equation}
where 
$\widetilde\alpha_{k}\in \{0,1\}$, 
$\alpha_{k}=1-\widetilde\alpha_{k}$.
Here we choose $\alpha_k$ and $\widetilde\alpha_k$ so that $\widetilde B_k^{\mathrm{u}}$ 
lies in the component of $I_{\varepsilon_0}\setminus B_k^{\mathrm{u}}$ containing $-\varepsilon_0$.
We set as above $\widetilde{\mathbb{B}}_k^{\rm u}=\mathbb{B}^{\rm u}(\underline{\widetilde w}^{(n_0+k)})$.
\index{Bkub@$\widetilde{\mathbb{B}}_k^{\rm u}$, $\mathbb{B}_k^{\rm u}$}

\section{Conditions on diffeomorphisms near $f_0$}\label{S_condition_diffeo}

From \eqref{F1} in Subsection \ref{ss_cone_foliation}, 
we have the foliation $\mathcal{F}_f^{\mathrm{cs}}\index{Ffcs@$\mathcal{F}_f^{\mathrm{cs}}$}$ on $\mathbb{H}_{\varepsilon_0}$ 
induced from $\mathcal{F}_f^{\mathrm{s}}$ via $(f^2|_{\mathbb{H}_{\varepsilon_0}})^{-1}$, each leaf of which is a $C^r$-surface in $\mathbb{H}_{\varepsilon_0}$.
See Figure \ref{f_4_1} for the case of $f=f_0$.
\begin{figure}[hbtp]
\centering
\scalebox{0.6}{\includegraphics[clip]{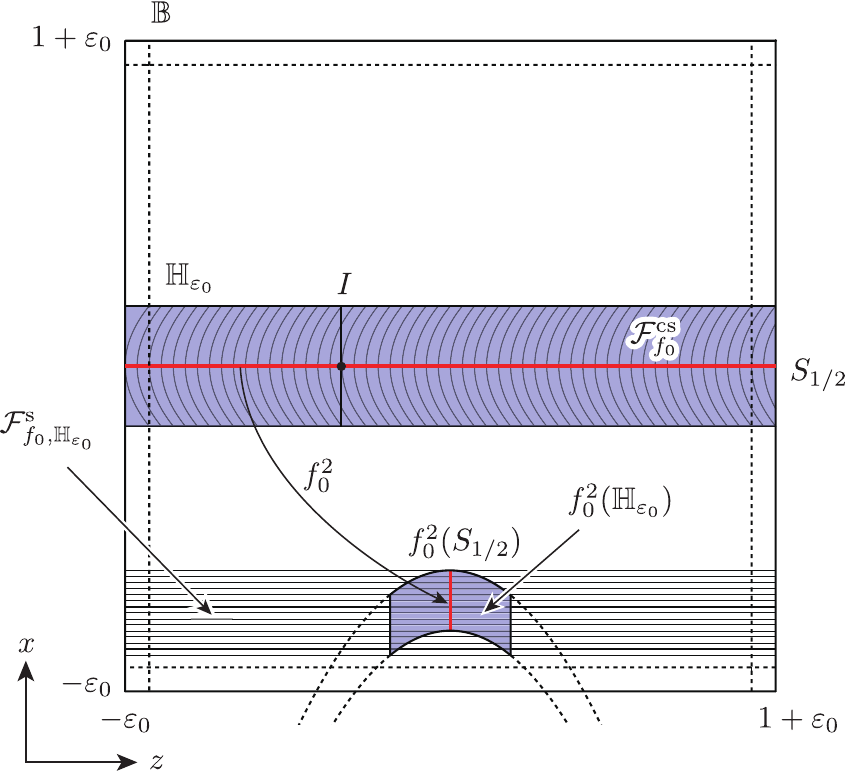}}
\caption{View from the top.
$\mathcal{F}_{f_0,\mathbb{H}_{\varepsilon_0}}^{\mathrm{s}}$ represents the sub-lamination of $\mathcal{F}_{f_0}^{\mathrm{s}}$ 
consisting of leaves meeting $f_0^2(\mathbb{H}_{\varepsilon_0})$ non-trivially.
}
\label{f_4_1}
\end{figure}
Then any maximal segment $I$ in $\mathbb{H}_{\varepsilon_0}$ parallel to the $x$-axis is tangent to a leaf $F$ of 
$\mathcal{F}_{f_0}^{\mathrm{cs}}$ at a point of $S_{1/2}$.
However, in the general case of $f\in \mathcal{O}(f_0)\setminus \{f_0\}$, we can not expect such a good situation.
So we introduce the notion of $\mathrm{cs}$-section instead of $S_{1/2}$.
To define $\mathrm{cs}$-sections, we need a 1-dimensional unstable foliation adaptable to the 
$f$-invariant cone-field $\boldsymbol{C}_{\varepsilon}^{\mathrm{u}}$ 
supported on a subset of $\mathbb{B}$ containing $\mathbb{H}_{\varepsilon_0}$.

\subsection{1-dimensional unstable foliations and cs-sections}\label{ss_1-dim}
For any binary code $\underline{w}^{(k)}=w_kw_{k-1}\dots w_2w_1$, let 
$\mathbb{H}_{\underline{w}^{(k)}}$ be the compact subset of $\mathbb{B}$ defined as
\begin{equation}\label{eqn_gg_gamma}
\mathbb{H}_{\underline{w}^{(k)}}\index{Hgk@$\mathbb{H}_{\underline{w}^{(k)}}$}=(f|_{\mathbb{V}_{w_1,f}}\circ f|_{\mathbb{V}_{w_2,f}}\circ\cdots \circ f|_{\mathbb{V}_{w_{k-1},f}}\circ f|_{\mathbb{V}_{w_k,f}})^{-1}(\mathbb{H}_{\varepsilon_0}),
\end{equation} 
and let 
$\mathbb{H}_{\,[k]}=\bigcup_{\underline{w}^{(k)}\in \{0,1\}^k}\mathbb{H}_{\underline{w}^{(k)}}$ and  
$\mathbb{H}_{\,[\infty]}=\bigcup_{k=0}^\infty \mathbb{H}_{\,[k]}$\index{Hkinf@$\mathbb{H}_{\,[k]}$, $\mathbb{H}_{\,[\infty]}$}, where $\mathbb{H}_{\,[0]}=\mathbb{H}_{\varepsilon_0}$.
Note that $\mathbb{H}_{\underline{w}^{(k)}}$ is contained in the u-bridge block $\mathbb{B}^{\mathrm{u}}(\underline{w}^{(k)})$ defined 
as \eqref{eqn_BBu} and called the \emph{u-flat block} of code $\underline{w}^{(k)}$.
Since $\mathbb{H}_{\varepsilon_0}$ is foliated by a sub-foliation of the $f$-invariant foliation $\mathcal{F}_f^{\mathrm{s}}$ by \eqref{F2}, 
$\mathbb{H}_{\,[\infty]}$ is also foliated by a sub-foliation of $\mathcal{F}_f^{\mathrm{s}}$, 
each leaf of which is adaptable to $\boldsymbol{C}_{\varepsilon}^{\mathrm{cs}}$.

For $a=x,y,z$, let $\pi_a:\mathbb{R}^3\longrightarrow \mathbb{R}$ be the orthogonal projection to the $a$-axis, that is, 
$\pi_x(x,y,z)=x$, $\pi_y(x,y,z)=y$, $\pi_z(x,y,z)=z$.
\index{pix@$\pi_x$, $\pi_y$, $\pi_z$}
For $k=0,1,2,\dots$, suppose that $\mathcal{L}_k$ is the 1-dimensional foliation on $\mathbb{H}_{\,[k]}$ each leaf of which is a straight segment 
in $\mathbb{H}_{\,[k]}$ parallel to the $x$-axis.
Let $\mathcal{N}(f(\mathbb{H}_{\,[k+1]}))$ be a small regular neighborhood of $f(\mathbb{H}_{\,[k+1]})$ in $\mathbb{H}_{\,[k]}$ 
such that $\mathbb{H}_{\,[k]}\setminus \mathcal{N}(f(\mathbb{H}_{\,[k+1]}))$ consists of leaves of 
$\mathcal{L}_{k}$, and let $\mathcal{N}_k$ be the closure of $\mathcal{N}(f(\mathbb{H}_{\,[k+1]}))\setminus 
 f(\mathbb{H}_{\,[k+1]})$ in $\mathbb{H}_{\,[k]}$.
The restriction $\mathcal{L}_k|_{f(\mathbb{H}_{\,[k+1]})}$ of the foliation $\mathcal{L}_k$ on $f(\mathbb{H}_{\,[k+1]})$ is not necessarily equal to the foliation 
$f(\mathcal{L}_{k+1})$ on $f(\mathbb{H}_{\,[k+1]})$ induced from $\mathcal{L}_{k+1}$ via $f|_{\mathbb{H}_{\,[k+1]}}$.
However, by \eqref{eqn_fCu}, any leaf of $f(\mathcal{L}_{k+1})$ is adaptable to $\boldsymbol{C}_{\varepsilon}^{\mathrm{u}}$.
Thus one can obtain a $C^r$-foliation $\mathcal{L}_{(k;k+1)}$ on $\mathbb{H}_{(k)}$ extending 
$\mathcal{L}_k|_{\mathbb{H}_{\,[k]}\setminus \mathcal{N}(f(\mathbb{H}_{\,[k+1]}))}\cup f(\mathcal{L}_{k+1})$ such that 
each leaf of $\mathcal{L}_{(k;k+1)}$ is also adaptable to $\boldsymbol{C}_{\varepsilon}^{\mathrm{u}}$.
See Figures \ref{f_4_2} and  \ref{f_4_3}\,(a) for the case of $k=0$.
\begin{figure}[hbtp]
\centering
\scalebox{0.6}{\includegraphics[clip]{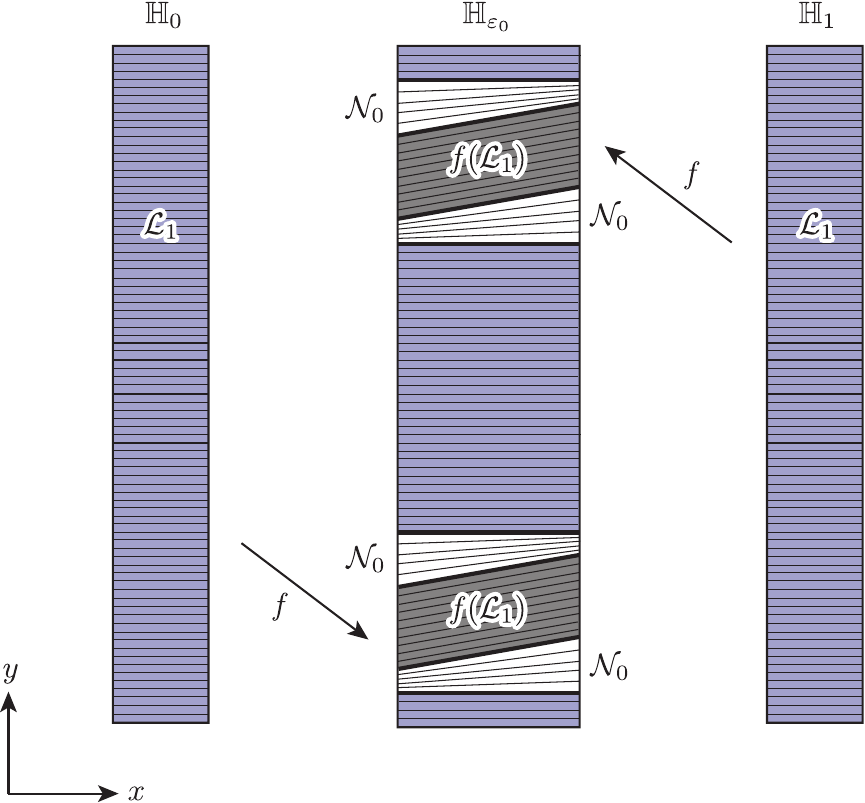}}
\caption{View from the side.
$\mathbb{H}_{\,[1]}=\mathbb{H}_0\cup \mathbb{H}_1$.}
\label{f_4_2}
\end{figure}
\begin{figure}[hbtp]
\centering
\scalebox{0.6}{\includegraphics[clip]{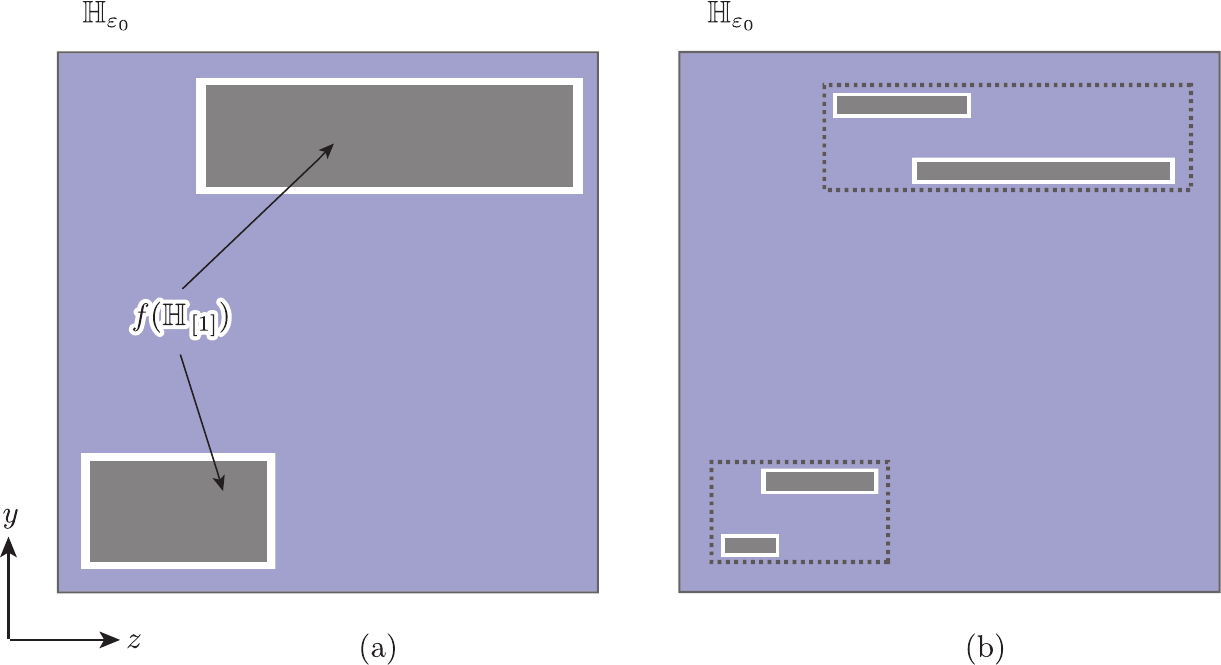}}
\caption{View from the front.
(a) The union of white frames represents $\mathcal{N}_0$.
(b) The union of gray rectangles represents $f^2(\mathbb{H}_{\,[2]})$.}
\label{f_4_3}
\end{figure}
Then $\mathcal{L}_{(k;k+1)}\cup\mathcal{L}_{k+1}$ is an $f$-invariant $C^r$-foliation on $\mathbb{H}_{\,[k]}\cup \mathbb{H}_{\,[k+1]}$.
By replacing $f(\mathcal{L}_{k+1})$ in $\mathcal{L}_{(k;k+1)}$ by $f(\mathcal{L}_{(k+1;k+2)})$, we have a foliation $\mathcal{L}_{(k;k+2)}$ 
on $\mathbb{H}_{\,[k]}$ so that $\mathcal{L}_{(k;k+2)}\cup \mathcal{L}_{(k+1;k+2)}\cup \mathcal{L}_{k+2}$ is an $f$-invariant foliation 
on $\mathbb{H}_{\,[k]}\cup \mathbb{H}_{\,[k+1]}\cup \mathbb{H}_{\,[k+2]}$.
See Figure \ref{f_4_3}\,(b)  for the case of $k=0$.
By applying the process repeatedly, we have a foliation $\mathcal{L}_{(k;\infty)}\index{Lkinf@$\mathcal{L}_{(k;\infty)}$}$ on $\mathbb{H}_{\,[k]}$ 
such that the union $\bigcup_{k=0}^\infty\mathcal{L}_{(k;\infty)}$ is an $f$-invariant foliation on $\mathbb{H}_{\,[\infty]}$ 
and each leaf $l$ of $\mathcal{L}_{(k;\infty)}$ is a $C^r$-arc adaptable to $\boldsymbol{C}_{\varepsilon}^{\mathrm{u}}$.
In particular, for any leaf $l$ of $\mathcal{L}_{(k;\infty)}$ contained in $f(\mathbb{H}_{\,[k+1]})$, 
$f^{-1}(l)$ is a leaf of $\mathcal{L}_{(k+1;\infty)}$.
Then we say that $\bigcup_{k=0}^\infty\mathcal{L}_{(k;\infty)}$ has the $f^{-1}$-\emph{invariance property}. 
This fact is used in Section \ref{S_BTC}.
From our construction, $\mathcal{L}_{(k;\infty)}$ is a $C^0$-foliation such that 
the restriction 
$\mathcal{L}_{(k;\infty)}|_{\mathbb{H}_{\,[k]}\setminus W_{\mathrm{loc}}^{\mathrm{u}}(\Lambda_f)}$ is a $C^r$-foliation on 
$\mathbb{H}_{\,[k]}\setminus W_{\mathrm{loc}}^{\mathrm{u}}(\Lambda_f)$.
However the authors do not know whether $\mathcal{L}_{(k;\infty)}$ is of $C^1$-class.
See Palis-Viana \cite[Example 3.1]{PV94} for a simple example of a $C^\infty$-diffeomorphism 
with a foliation of codimension 2 which is not $C^1$.

Since $Df(\boldsymbol{x})$ is sufficiently $C^{r-1}$-close to the constant diagonal matrix $Df_0(\boldsymbol{x})$ for any 
$\boldsymbol{x}\in \mathbb{V}_{0,f}\cup \mathbb{V}_{1,f}$, we may assume that the derivative  
of any entry of $Df(\boldsymbol{x})$ is an $O(\varepsilon)$-function, that is,
\begin{equation}\label{eqn_DDf}
\dfrac{\partial^2 (\pi_a\circ f)}{\partial x_j\partial x_k}
(\boldsymbol{x})=O(\varepsilon),
\end{equation}
where $a,x_j,x_k\in \{x,y,z\}$.
Hence one can choose the $C^r$-foliation $\mathcal{L}(k;k+1)$ so that, for any leaf 
$l_k$ of $\mathcal{L}(k;k+1)$ and any point $\boldsymbol{x}_k$ of $l_k$, 
the curvature $\kappa_{l_k}(\boldsymbol{x}_k)$ of $l_k$ at $\boldsymbol{x}_k$ is $O(\varepsilon)$.

From \eqref{eqn_tang}, we know that each leaf $F_0$ of $\mathcal{F}_{f_0}^{\mathrm{cs}}$ is 
a vertical parabolic cylinder parametrized as
\begin{equation}\label{eqn_F_cylinder}
\begin{split}
F_0=\biggl\{\left(a_4^{-1}t+\frac12,\ s,\ a_1a_2^{-1}a_4^{-2}t^2+c\right)\,;\,-|a_4|\varepsilon_0\leq &t\leq |a_4|\varepsilon_0,\\
&-\varepsilon_0\leq s\leq 1+\varepsilon_0\biggr\}
\end{split}
\end{equation}
for some constant $c$.
Since the restriction $f^2|_{\mathbb{H}_{\varepsilon_0}}$ is arbitrarily $C^r$ close to $f_0^2|_{\mathbb{H}_{\varepsilon_0}}$, 
any leaf $F$ of $\mathcal{F}_f^{\mathrm{cs}}$ also looks like a vertical parabolic cylinder.
In particular, we have the following lemma.

\begin{lem}\label{l_Fcs}
Any leaf $F$ of $\mathcal{F}_f^{\mathrm{cs}}$ has a non-singular $C^1$-vector field $X$ such that, 
for any $\boldsymbol{x}\in F$, $X(\boldsymbol{x})$ is contained in $\boldsymbol{C}_{\varepsilon}^{\mathrm{ss}}(\boldsymbol{x})$.
\end{lem}

Let $l$ be any leaf of $\mathcal{L}_{(0;\infty)}$.
By Propositions \ref{p_curvature_plane} and \ref{p_kappa_fn} in Appendix \ref{Ap_curvature}, 
$f^2(l)$ is quadratically tangent to a leaf of $\mathcal{F}_f^{\mathrm{s}}$.
Thus there exists a unique leaf of $\mathcal{F}_f^{\mathrm{cs}}$ quadratically tangent to $l$ at a single point.
We denote the leaf by $F^{\mathrm{cs}}(l)$\index{Fcsl@$F^{\mathrm{cs}}(l)$} and the tangent point by $\boldsymbol{x}(l)$.
See Figure \ref{f_4_4}\,(a).
\begin{figure}[hbtp]
\centering
\scalebox{0.6}{\includegraphics[clip]{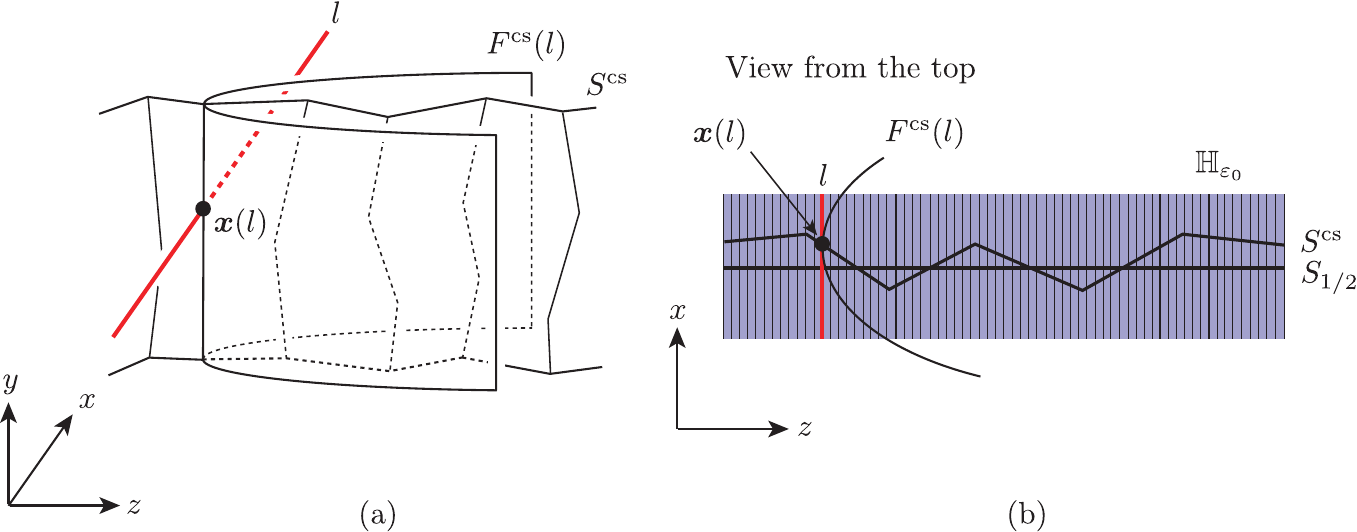}}
\caption{}
\label{f_4_4}
\end{figure}
Since $\mathcal{L}_{(0;\infty)}$ is a $C^0$-foliation on $\mathbb{H}_{\varepsilon_0}$, if $l_n\in \mathcal{L}_{(0;\infty)}$ converges to $l$, 
then $\boldsymbol{x}(l_n)$ converges to $\boldsymbol{x}(l)$.
Thus the subset  
$$S^{\mathrm{cs}}=\{\boldsymbol{x}(l)\,;\,l\in \mathcal{L}_{(0;\infty)}\}\index{Scs@$S^{\mathrm{cs}}$}$$ 
of $\mathbb{H}_{\varepsilon_0}$ is the `graph' of a continuous function on $S_{1/2}$ with respect to $\mathcal{L}_{(0;\infty)}$.
See Figure \ref{f_4_4}\,(b).
In particular, $S^{\mathrm{cs}}$ is a surface $C^0$-embedded in $\mathbb{H}_{\varepsilon_0}$ and 
$C^0$-converges to $S_{1/2}$ as $f\rightarrow f_0$.
We say that $S^{\mathrm{cs}}$ is the \emph{cs-section} of $\mathbb{H}_{\varepsilon_0}$ with respect to 
$\mathcal{L}_{(0;\infty)}$.

\subsection{Estimation of the norm of derivatives}\label{ss_norm_deriv}

\begin{subequations}
For the constants $\lambda_j$ with $j\in\{{\rm u,ss, cs0, cs1}\}$ given in Subsection \ref{ss_affine_model} and $\varepsilon>0$ as above, we write
\begin{equation}\label{abb1}
\underline{\lambda}_{j}=\lambda_{j}-\varepsilon,\ 
\bar{\lambda}_{j}=\lambda_{j}+\varepsilon.
\index{lambdabar@$\underline{\lambda}_j,\bar{\lambda}_j$}
\end{equation}
By \eqref{eqn_eigen_v}--\eqref{eqn_pdc}, 
one can suppose that the following inequalities hold.
\begin{equation}\label{eqn_eigen_v2}
0<\underline{\lambda}_{\rm ss}<\bar\lambda_{\rm cs0}<1/2<\underline\lambda_{\rm cs1}<\bar\lambda_{\rm cs1}<1<\underline\lambda_{\rm cs0}+\underline\lambda_{\rm cs1},\quad 2<\underline\lambda_{\rm u},
\end{equation}
\begin{equation}\label{pdc2}
\bar\lambda_{\rm cs0}\bar\lambda_{\rm cs1}\bar\lambda_{u}^{2}<1.
\end{equation}
\end{subequations}
The condition \eqref{pdc2} is used in the proof of Lemma \ref{lem7.3}.

By \eqref{eqn_f0V}, $Df(\boldsymbol{x})$ is arbitrarily $C^{r-1}$-close to the diagonal matrix  $Df_0(\boldsymbol{x})=\mathrm{diag}((-1)^i\lambda_{\mathrm{u}},(-1)^i\lambda_{\mathrm{ss}},\lambda_{\mathrm{cs} i})$
for $\boldsymbol{x}\in \mathbb{V}_{i,f}$ $(i=0,1)$.
Here we recall that $\mathbb{V}_{i,f}$ is the component of $\mathbb{B}\cap f^{-1}(\mathbb{B})$ containing $\mathbb{B}^{\mathrm{u}}(i)$.
\begin{subequations}
We may assume that   
\begin{equation}\label{eqn_lambda_maxS}
\begin{split}
&\max\bigl\{| D(\pi_z\circ f)(\boldsymbol{x})|\,;\, \boldsymbol{x}\in \mathbb{V}_{i,f}\bigr\}<\bar\lambda_{\mathrm{cs} i}-\frac{\varepsilon}2,\\
&\min\bigl\{m(D(\pi_z\circ f)(\boldsymbol{x}))\,;\, \boldsymbol{x}\in \mathbb{V}_{i,f}\bigr\}>\underline{\lambda}_{\mathrm{cs} i}+\frac{\varepsilon}2
\end{split}
\end{equation}
for $i=0,1$,
\begin{equation}\label{eqn_lambda_max}
\begin{split}
&\max\bigl\{| D(\pi_x\circ f)(\boldsymbol{x})|\,;\, \boldsymbol{x}\in \mathbb{V}_{0,f}\cup \mathbb{V}_{1,f}\bigr\}<\bar\lambda_{\mathrm{u}}-\frac{\varepsilon}2,\\
&\min\bigl\{m(D(\pi_x\circ f)(\boldsymbol{x}))\,;\, \boldsymbol{x}\in \mathbb{V}_{0,f}\cup \mathbb{V}_{1,f}\bigr\}>\underline{\lambda}_{\mathrm{u}}+\frac{\varepsilon}2,
\end{split}
\end{equation}
and
\begin{equation}\label{eqn_lambda_maxc}
\begin{split}
&\max\bigl\{| D(\pi_y\circ f)(\boldsymbol{x})|\,;\, \boldsymbol{x}\in \mathbb{V}_{0,f}\cup \mathbb{V}_{1,f}\bigr\}<\bar\lambda_{\mathrm{ss}}-\frac{\varepsilon}2,\\
&\min\bigl\{m(D(\pi_y\circ f)(\boldsymbol{x}))\,;\, \boldsymbol{x}\in \mathbb{V}_{0,f}\cup \mathbb{V}_{1,f}\bigr\}>\underline{\lambda}_{\mathrm{ss}}+\frac{\varepsilon}2,
\end{split}
\end{equation}
\end{subequations}
where we define, for any linear map $A:T_{\boldsymbol{x}}\mathbb{B}\longrightarrow \mathbb{R}$ $(\boldsymbol{x}\in \mathbb{V}_{0,f}\cup \mathbb{V}_{1,f})$, 
\begin{align*}
|A|&=\max\bigl\{|A(\boldsymbol{v})|\,;\,\boldsymbol{v}\in T_{\boldsymbol{x}}\mathbb{B}\text{ with }\|\boldsymbol{v}\|=1\bigr\},\\
m(A)&=\min\bigl\{|A(\boldsymbol{v})|\,;\,\boldsymbol{v}\in T_{\boldsymbol{x}}\mathbb{B}\text{ with }\|\boldsymbol{v}\|=1\bigr\}.
\end{align*}

By \eqref{eqn_lambda_max}, there exists a constant $0<C_0<1$ independent of $k$ such that 
\begin{equation}\label{eqn_C0GC0}
C_0\bar{\lambda}_{\mathrm{u}}^{-k}<|B^{\mathrm{u}}(\underline{w}^{(k)})|<C_0^{-1}\underline{\lambda}_{\mathrm{u}}^{-k}. 
\end{equation}
Hence, for any $C^1$-curve $l$ in $\mathbb{B}^{\mathrm{u}}(\underline{w}^{(k)})$ adaptable to the cone-field $\boldsymbol{C}_{\varepsilon}^{\mathrm{u}}$, 
we may assume that 
\begin{equation}\label{eqn_lC0lam}
\mathrm{length}(l)\leq C_0^{-1}\underline{\lambda}_{\mathrm{u}}^{-k}
\end{equation} if necessary replacing $C_0$ with a smaller positive number.

\section{Backtracking condition for cs-sections}\label{S_BTC}

This section provides geometric information near tangencies that will be perturbed in
Section \ref{S_perturb}.

\subsection{Forward sequence of cs-sections}\label{ss_forward}

In this subsection, we define a forward sequence of sc-sections, which is applied   
to construct a sequence from $S_{\widehat{\underline{w}}_k}^{\mathrm{cs}}$ to 
$S_{\underline{\gamma}^{(m_k)}}^{\mathrm{cs}}$ defined below as illustrated in Figure \ref{f_8_1}.

For any binary code $\underline{\gamma}^{(n)}=\gamma_n\dots \gamma_2\gamma_1\in \{0,1\}^n$ of finite length, we denote by 
$\zeta_{\underline\gamma}$ the composition $\zeta_{\gamma_1}\circ\zeta_{\gamma_2}\circ\cdots\circ 
\zeta_{\gamma_n}$.
Since $\lambda_{\mathrm{cs} 0}+\lambda_{\mathrm{cs} 1}>1$ by \eqref{eqn_eigen_v}, 
$1-\lambda_{\mathrm{cs} 1}+\eta<\lambda_{\mathrm{cs} 0}-\eta$ for any sufficiently small $\eta$.
Then we set 
$$I(\eta)=[1-\lambda_{\mathrm{cs} 1}+\eta,\lambda_{\mathrm{cs} 0}-\eta].$$
From the definition, $I(\eta)\subset I(\eta')$ if $\eta>\eta'$.

\begin{lem}\label{l_zeta_zeta}
There exists $\mu_0\in\mathbb{N}$ satisfying the following property.
For any 
$$z\in [-\varepsilon_0,1-\lambda_{\mathrm{cs} 1}+7\varepsilon]\cup[\lambda_{\mathrm{cs} 0}-7\varepsilon,1+\varepsilon_0]$$
there exists a binary code $\underline\iota$ with $|\underline\iota|\leq \mu_0$ such that 
$\zeta_{\underline\iota}(z)\in I(7\varepsilon)$ for any sufficiently small $\varepsilon>0$.
\end{lem}

Note that the code $\underline\iota$ depends on $z$ but is independent of $\varepsilon$.

\begin{proof}
From the definition \eqref{eqn_def_zeta} of $\zeta_0$ and $\zeta_1$, 
$$\zeta_1([0,1-\lambda_{\mathrm{cs} 1}+7\varepsilon])=[1-\lambda_{\mathrm{cs} 1}, 1-\lambda_{\mathrm{cs} 1}(\lambda_{\mathrm{cs} 1}-7\varepsilon)]
\subset [1-\lambda_{\mathrm{cs} 1},1-\lambda_{\mathrm{cs} 1}^3],$$
where $\varepsilon>0$ is taken so that $\lambda_{\mathrm{cs} 1}-7\varepsilon>\lambda_{\mathrm{cs} 1}^2$.
See Figure \ref{f_3_1}.
Fix $\varepsilon_1>0$ with $7\varepsilon_1<\lambda_{\mathrm{cs} 0}+\lambda_{\mathrm{cs} 1}-1-7\varepsilon_1$.
Let $\underline 0^{(\mu_1)}$ be the code $(00\dots 0)$ of length $\mu_1$.
One can take $\mu_1$ with 
$$\zeta_{\underline{0}^{(\mu_1)}}(1-\lambda_{\mathrm{cs} 1}^3)=\lambda_{\mathrm{cs} 0}^{\mu_1}(1-\lambda_{\mathrm{cs} 1}^3)<
\lambda_{\mathrm{cs} 1}^{-1}(\lambda_{\mathrm{cs} 0}+\lambda_{\mathrm{cs} 1}-1-7\varepsilon_1).$$
Take $0<\varepsilon_2\leq \varepsilon_1$ with 
$\zeta_{\underline{0}^{(\mu_1)}}(1-\lambda_{\mathrm{cs} 1})=\lambda_{\mathrm{cs} 0}^{\mu_1}(1-\lambda_{\mathrm{cs} 1})>7\lambda_{\mathrm{cs} 1}^{-1}\varepsilon_2$.
It follows that, for any $0<\varepsilon\leq \varepsilon_2$,
\begin{equation}\label{eqn_zeta_circ}
\begin{split}
\zeta_1\circ \zeta_{\underline{0}^{(\mu_1)}}\circ \zeta_1([0,& 1-\lambda_{\mathrm{cs} 1}+7\varepsilon])\\
&\subset \zeta_1\bigl(\,[7\lambda_{\mathrm{cs} 1}^{-1}\varepsilon,\lambda_{\mathrm{cs} 1}^{-1}(\lambda_{\mathrm{cs} 0}+\lambda_{\mathrm{cs} 1}-1-7\varepsilon)]\,\bigr)=I(7\varepsilon).
\end{split}
\end{equation}
On the other hand, by \eqref{eqn_lam1ev}, 
$$\zeta_1([-\varepsilon_0,0])=[1-\lambda_{\mathrm{cs} 1}(1+\varepsilon_0),1-\lambda_{\mathrm{cs} 1}]\subset [0,1-\lambda_{\mathrm{cs} 1}+7\varepsilon].$$
It follows from this fact together with \eqref{eqn_zeta_circ} that 
$$
\zeta_1\circ \zeta_{\underline{0}^{(\mu_1)}}\circ \zeta_{11}([-\varepsilon_0,0])
\subset I(7\varepsilon).$$
Now we fix $\mu_2$ with 
$$\zeta_{\underline{0}^{(\mu_2)}}(1+\varepsilon_0)=\lambda_{\mathrm{cs} 0}^{\mu_2}(1+\varepsilon_0)
<1-\lambda_{\mathrm{cs} 1}<1-\lambda_{\mathrm{cs} 1}+7\varepsilon.$$
Then, again by \eqref{eqn_zeta_circ},
$$\zeta_1\circ \zeta_{\underline{0}^{(\mu_1)}}\circ \zeta_1\circ \zeta_{\underline{0}^{(\mu_2)}}
([\lambda_{\mathrm{cs} 1}+7\varepsilon,1+\varepsilon_0])\subset I(7\varepsilon).$$
Thus $\mu_0=\mu_1+\mu_2+2$ is an integer satisfying the required condition.
\end{proof}

For any binary code $\underline{\gamma}^{(k)}=\gamma_k\gamma_{k-1}\dots \gamma_2\gamma_1$ of finite length, the 
surface $S_{\underline{\gamma}^{(k)}}^{\mathrm{cs}}\index{Scsgk@$S_{\underline{\gamma}^{(k)}}^{\mathrm{cs}}$}$ defined by 
$$S_{\underline{\gamma}^{(k)}}^{\mathrm{cs}}=(f|_{\mathbb{V}_{\gamma_1,f}}\circ f|_{\mathbb{V}_{\gamma_2,f}}\circ\dots \circ f|_{\mathbb{V}_{\gamma_{k-1},f}}\circ f|_{\mathbb{V}_{\gamma_k,f}})^{-1}(S^{\mathrm{cs}})$$
is called the \emph{cs-section} of $\mathbb{H}_{\underline{\gamma}^{(k)}}$.

\begin{lem}\label{l_fSw}
Suppose that $\mu_0$ is the positive integer given in Lemma \ref{l_zeta_zeta}.
Let $\underline{w}\underline{\gamma}$ be any binary code of finite length.
If $|\underline{w}|$ is sufficiently large, then there exists a binary code $\underline\iota$ of length at most $\mu_0$ 
(possibly $\underline\iota=\emptyset$) 
such that $\pi_z\circ f^{|\underline{w}|+|\underline\iota|}(S_{\underline{w}\underline\iota\underline{\gamma}}^{\mathrm{cs}})$ 
is contained in $I(4\varepsilon)$ 
for any $f\in\mathrm{Diff}^r(M)$ sufficiently $C^r$-close to $f_0$.
\end{lem}
\begin{proof}
Fix an element $\boldsymbol{x}_{\underline{w}\underline{\gamma}}$ in $S_{\underline{w}\underline{\gamma}}^{\mathrm{cs}}$ and suppose first that 
$\pi_z\circ f^{|\underline{w}|}(\boldsymbol{x}_{\underline{w}\underline{\gamma}})$ is contained in $I(7\varepsilon)$.
By \eqref{eqn_lambda_maxS}, $|\pi_z\circ f^{|\underline{w}|}(S_{\underline{w}\underline{\gamma}}^{\mathrm{cs}})|\leq C(\bar\lambda_{\mathrm{cs} 1})^{|\underline{w}|}$ holds for some constant $C>0$ independent of $|\underline{w}|$ or $|\underline{\gamma}|$.
Since $I(4\varepsilon)\setminus I(7\varepsilon)$ 
consists of two intervals of length $3\varepsilon$, 
$\pi_z\circ f^{|\underline{w}|}(S_{\underline{w}\underline{\gamma}}^{\mathrm{cs}})$ is contained in $I(4\varepsilon)$ if 
$|\underline{w}|$ is sufficiently large.

Next we consider the case that $\pi_z\circ f^{|\underline{w}|}(\boldsymbol{x}_{\underline{w}\underline{\gamma}})$ is not an element of $I(7\varepsilon)$.
Then $\pi_z\circ f^{|\underline{w}|}(\boldsymbol{x}_{\underline{w}\underline{\gamma}})$ is contained in 
$[-\varepsilon_0,1-\lambda_{\mathrm{cs} 1}+7\varepsilon]\cup [\lambda_{\mathrm{cs} 0}-7\varepsilon,1+\varepsilon_0]$.
Suppose that $\pi_{yz}:\mathbb{B}\longrightarrow I_{\varepsilon_0}^2$\index{piyz@$\pi_{yz}$} is the orthogonal projection defined as 
$\pi_{yz}(x,y,z)=(y,z)$.
By Lemma \ref{l_zeta_zeta}, 
there exits a binary code $\underline\iota$ of length at most $\mu_0$ such that 
$\pi_z\circ f^{|\underline w|+|\underline\iota|}(\boldsymbol{x}_{\underline w\underline\gamma})$ is contained in $I(6\varepsilon)$ if 
$f$ is sufficiently $C^r$-close to $f_0$.
For the proof, we need to show that $\pi_z(f^{|\underline{w}|}(\boldsymbol{x}_{\underline{w}\underline{\iota}\underline{\gamma}}))$ is arbitrarily close to 
$\pi_z(f^{|\underline{w}|}(\boldsymbol{x}_{\underline{w}\underline{\gamma}}))$ even in the case that $|\underline{w}|$ is large.
We use here the $f$-invariant unstable cone-field $\boldsymbol{C}_{\varepsilon}^{\mathrm{u}}$.
Consider the straight segment $l$ in $\mathbb{B}$ passing through $\boldsymbol{x}_{\underline{w}\underline{\gamma}}$ and 
$\boldsymbol{x}_{\underline{w}\underline{\iota}\underline{\gamma}}$.
Since $l$ is parallel to the $x$-axis, $T_{\boldsymbol{x}}(l)$ is contained in $\boldsymbol{C}_{\varepsilon}^{\mathrm{u}}(\boldsymbol{x})$ for 
any $\boldsymbol{x}\in l$.
Let $l'$ be the component of $f^{|\underline{w}|}(l)\cap \mathbb{B}$ with $l'\supset \{f^{|\underline{w}|}(\boldsymbol{x}_{\underline{w}\underline{\iota}\underline{\gamma}}),
f^{|\underline{w}|}(\boldsymbol{x}_{\underline{w}\underline{\gamma}})\}$.
Since $\boldsymbol{C}_{\varepsilon}^{\mathrm{u}}$ is $f$-invariant, $T_{\boldsymbol{x}'}(l')$ is contained in $\boldsymbol{C}_{\varepsilon}^{\mathrm{u}}(\boldsymbol{x}')$ for any 
$\boldsymbol{x}'\in l'$.
This implies that $\pi_z(f^{|\underline{w}|+|\underline{\iota}|}(\boldsymbol{x}_{\underline{w}\underline{\iota}\underline{\gamma}}))$ is arbitrarily close to 
$\pi_z(f^{|\underline{w}|+|\underline\iota|}(\boldsymbol{x}_{\underline{w}\underline{\gamma}}))$.
See Figure \ref{f_5_1}.
\begin{figure}[hbtp]
\centering
\scalebox{0.6}{\includegraphics[clip]{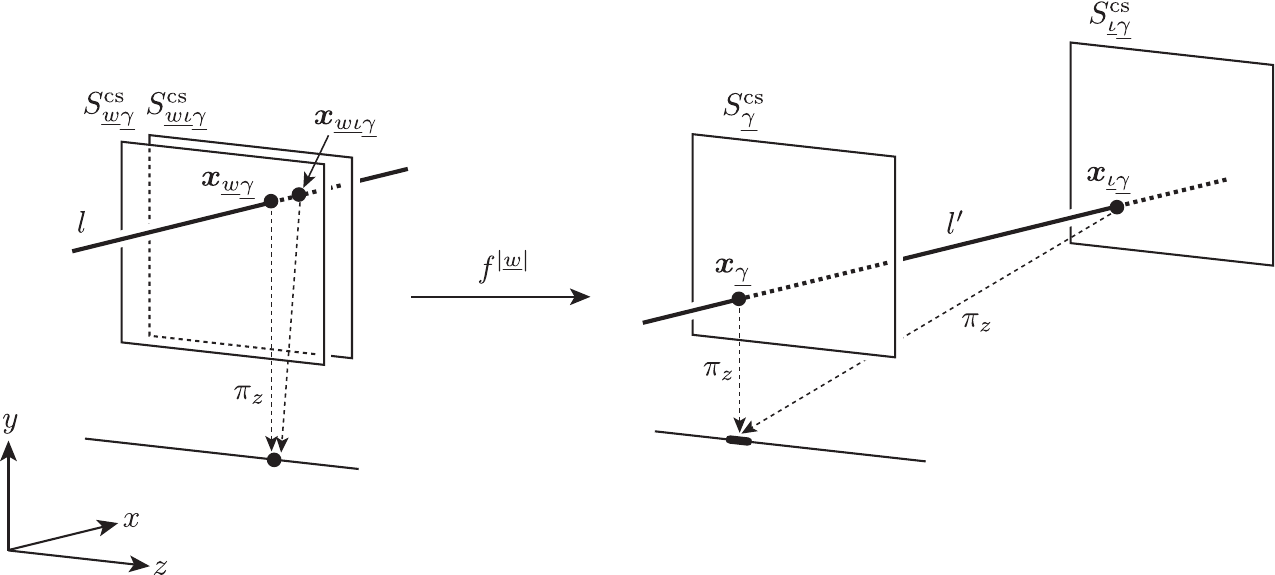}}
\caption{}
\label{f_5_1}
\end{figure}
Thus one can suppose that $\pi_z(f^{|\underline{w}|+|\underline{\iota}|}(\boldsymbol{x}_{\underline{w}\underline{\iota}\underline{\gamma}}))$ is 
contained in $I(5\varepsilon)$ and hence $\pi_z(f^{|\underline{w}|+|\underline{\iota}|}(S_{\underline{w}\underline{\iota}\underline{\gamma}}^{\mathrm{cs}}))$ 
is in $I(4\varepsilon)$ if $|\underline w|$ is sufficiently large.
\end{proof}

\subsection{Backward sequence of sub-surfaces of $S^{\mathrm{cs}}$}\label{ss_BTC}

This subsection is a preparation for the construction of a 
backward sequence from a certain sub-surface $\widehat{\Sigma}_{k+1}^{\mathrm{cs}}$ of $S^{\mathrm{cs}}$ to 
$\Sigma_{\underline{\gamma}^{(m_k)}}^{\mathrm{cs}}$ as illustrated in Figure \ref{f_8_1}.

For any binary code $\underline{\gamma}$ of finite length, we say that a compact connected sub-surface $\Sigma$ 
of $S_{\underline{\gamma}}^{\mathrm{cs}}$ satisfies the \emph{backtracking condition} if $\pi_z(\Sigma)$ is contained 
in $[\varepsilon,\lambda_{\mathrm{cs} 0}-\varepsilon]$ or $[1-\lambda_{\mathrm{cs} 1}+\varepsilon,1-\varepsilon]$.
Recall that $\mathcal{F}_f^{\mathrm{cs}}$ is the foliation on $\mathbb{H}_{\varepsilon_0}$ induced from $\mathcal{F}_f^{\mathrm{s}}$ via 
$(f|_{\mathbb{H}_{\varepsilon_0}})^{-2}$ defined in Section \ref{S_condition_diffeo}.
Let $\mathbb{U}^{\mathrm{cs}}\index{Ucs@$\mathbb{U}^{\mathrm{cs}}$}$ be the closure of the middle component of $\mathbb{H}_{\varepsilon_0}\setminus (F^{\mathrm{cs}-}\cup F^{\mathrm{cs}+})$ for 
some leaves $F^{\mathrm{cs}-}$, $F^{\mathrm{cs}+}$ of $\mathcal{F}_f^{\mathrm{cs}}$ with $F^{\mathrm{cs}-}\neq F^{\mathrm{cs}+}$, where 
$F^{\mathrm{cs}-}$ is assumed to be closer to the vertical plane $z=-\varepsilon_0$ compared with $F^{\mathrm{cs}+}$.
We set $F^{\mathrm{cs}-}\cup F^{\mathrm{cs}+}=\partial_z\mathbb{U}^{\mathrm{cs}}\index{Ucs_Pz@$\partial_z\mathbb{U}^{\mathrm{cs}}$}$ and call $F^{\mathrm{cs}-}$ and $F^{\mathrm{cs}+}$ respectively the 
\emph{left} and \emph{right components} of $\partial_z\mathbb{U}^{\mathrm{cs}}$.
See Figure \ref{f_5_2}.
\begin{figure}[hbtp]
\centering
\scalebox{0.6}{\includegraphics[clip]{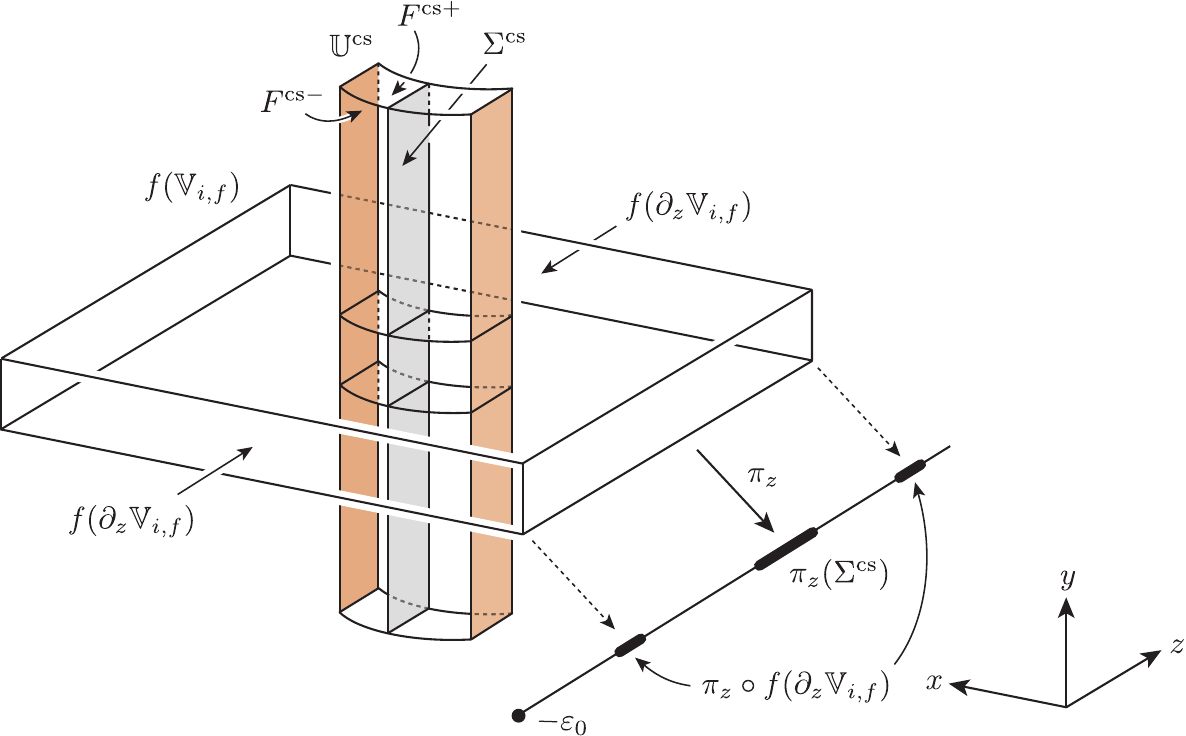}}
\caption{The case of $a_2>0$.
On the other hand, when $a_2<0$, $F_{\mathrm{cs}+}$ is convex and $F^{\mathrm{cs}-}$ is concave.}
\label{f_5_2}
\end{figure}
Suppose that the \emph{cs-section} $\Sigma^{\mathrm{cs}}=\mathbb{U}^{\mathrm{cs}}\cap S^{\mathrm{cs}}\index{Sgcs@$\Sigma^{\mathrm{cs}}$}$ of $\mathbb{U}^{\mathrm{cs}}$ satisfies the backtracking condition.
Then, at least one of $i=0,1$, $\pi_z(\Sigma^{\mathrm{cs}})\subset \pi_z(f(\mathbb{V}_{i,f}))$ and 
$\pi_z(\Sigma^{\mathrm{cs}})\cap \pi_z(f(\partial_z\mathbb{V}_{i,f}))=\emptyset$, where 
$\partial_z \mathbb{V}_{i,f}=\partial\mathbb{V}_{i,f}\cap (I_{\varepsilon_0}^2\times \{-\varepsilon_0,1+\varepsilon_0\})$.
We denote the `$i$' by $\gamma_1$.
Then one can obtain the \emph{cs-curved block} $\mathbb{U}_{\gamma_1}^{\mathrm{cs}}=(f|_{\mathbb{V}_{\gamma_1,f}})^{-1}(\mathbb{U}^{\mathrm{cs}})$ in $\mathbb{H}_{\gamma_1}$ with the 
section $\Sigma_{\gamma_1}^{\mathrm{cs}}=\mathbb{U}_{\gamma_1}^{\mathrm{cs}}\cap S_{\gamma_1}^{\mathrm{cs}}$.
If $\Sigma_{\gamma_1}^{\mathrm{cs}}$ also satisfies the backtracking condition, then we have the cs-curved block 
$\mathbb{U}_{\gamma_2\gamma_1}^{\mathrm{cs}}=(f|_{\mathbb{V}_{\gamma_2,f}})^{-1}(\mathbb{U}_{\gamma_1}^{\mathrm{cs}})$ in $\mathbb{H}_{\gamma_2\gamma_1}$ with the $\mathrm{cs}$-section $\Sigma_{\gamma_2\gamma_1}^{\mathrm{cs}}$ 
similarly.
We repeat the process as much as possible so that $\Sigma_{\underline{\gamma}^{(j)}}^{\mathrm{cs}}\index{Sggcs@$\Sigma_{\underline{\gamma}^{(j)}}^{\mathrm{cs}}$}$ satisfies the 
backtracking condition for $j=1,\dots,m-1$ and $\Sigma_{\underline{\gamma}^{(m)}}^{\mathrm{cs}}$ does not, where 
$\underline{\gamma}^{(j)}=\gamma_j\gamma_{j-1}\dots \gamma_2\gamma_1$.
We say that $\Sigma_{\underline{\gamma}^{(m)}}^{\mathrm{cs}}$ is a \emph{back-end section} based at $\Sigma^{\mathrm{cs}}$.
Then $\mathbb{U}_{\underline{\gamma}^{(j)}}^{\mathrm{cs}}\index{Ucsg@$\mathbb{U}_{\underline{\gamma}^{(j)}}^{\mathrm{cs}}$}$ $(j=1,\dots,m)$ is 
the cs-curved block in $\mathbb{H}_{\underline{\gamma}^{(j)}}$ defined inductively from $\mathbb{U}_{\gamma_1}^{\mathrm{cs}}$.

Let $F_{\underline{\gamma}^{(j)}}^{\mathrm{cs}-}$\index{Fgcs-@$F_{\underline{\gamma}^{(j)}}^{\mathrm{cs}-}$, $F_{\underline{\gamma}^{(j)}}^{\mathrm{cs}+}$} and $F_{\underline{\gamma}^{(j)}}^{\mathrm{cs}+}$ be the left and right components of 
$\partial_z \mathbb{U}_{\underline{\gamma}^{(j)}}^{\mathrm{cs}}$ respectively.
For any $t$ with $-\varepsilon_0\leq t\leq 1+\varepsilon_0$, let $P_t$ be the horizontal plane $y=t$ in $\mathbb{B}$.
Note that $F_{\underline{\gamma}^{(j)}}^{\mathrm{cs} \ast}\cap P_t$ $(\ast=\pm)$ is an almost parabolic curve in $P_t$, 
that is, it is represented as the graph of a $C^r$-function
$$z=a_{t;\ast}(x-b_{t;\ast})^2(1+O(x-b_{t;\ast}))+c_{t;\ast}$$
on $x$, 
where $a_{t;\ast}(\neq 0)$, $b_{t;\ast}$, $c_{t;\ast}$ are $C^r$-functions of $t$.
By \eqref{eqn_F_cylinder}, $a_{t;\ast}$ and $a_2$ have the same sign.
By Lemma \ref{l_Fcs}, there exists a non-singular $C^1$-vector field $X_\ast$ on $F_{\underline{\gamma}^{(j)}}^{\mathrm{cs} \ast}$ for 
$\ast=\pm$ with $X_\ast(\boldsymbol{x})\in \boldsymbol{C}_{\varepsilon}^{\mathrm{ss}}(\boldsymbol{x})$.
It follows that
$$a_{t;\ast}=a_{1/2;\ast}+O(\varepsilon),\quad b_{t;\ast}=b_{1/2;\ast}+O(\varepsilon)\quad\text{and}\quad c_{t;\ast}=c_{1/2;\ast}+O(\varepsilon)$$
for any $t\in [-\varepsilon_0,1+\varepsilon_0]$.
By the $f^{-1}$-invariance property on $\bigcup_{j=0}^\infty \mathcal{L}_{(j;\infty)}$, 
for any $\boldsymbol{x}_\ast\in F_{\underline{\gamma}^{(j)}}^{\mathrm{cs} \ast}\cap \Sigma_{\underline{\gamma}^{(j)}}^{\mathrm{cs}}$, 
there exists a leaf $l$ of $\mathcal{L}_{(j;\infty)}$ which is tangent to $F_{\underline{\gamma}^{(j)}}^{\mathrm{cs} \ast}$ at $\boldsymbol{x}_\ast$.
Since $l$ is adaptable to $\boldsymbol{C}_{\varepsilon}^{\mathrm{u}}$, $|\pi_z(\boldsymbol{x}_\ast)-c_{t;\ast}|=O(\varepsilon)$ if $\pi_y(\boldsymbol{x}_\ast)=t$ and hence 
$|\pi_z(\boldsymbol{x}_\ast)-c_{1/2;\ast}|=O(\varepsilon)$.
If necessary reconstructing $\mathcal{O}(f_0)$, one can suppose that 
\begin{equation}\label{eqn_pi_z_xq}
|\pi_z(\boldsymbol{x}_\ast)-c_{1/2;\ast}|<\varepsilon
\end{equation}
for any $\boldsymbol{x}_\ast \in F_{\underline{\gamma}^{(j)}}^{\mathrm{cs} \ast}\cap \Sigma_{\underline{\gamma}^{(j)}}^{\mathrm{cs}}$ if $f\in \mathcal{O}(f_0)$.
See Figure \ref{f_5_3}.
\begin{figure}[hbtp]
\centering
\scalebox{0.6}{\includegraphics[clip]{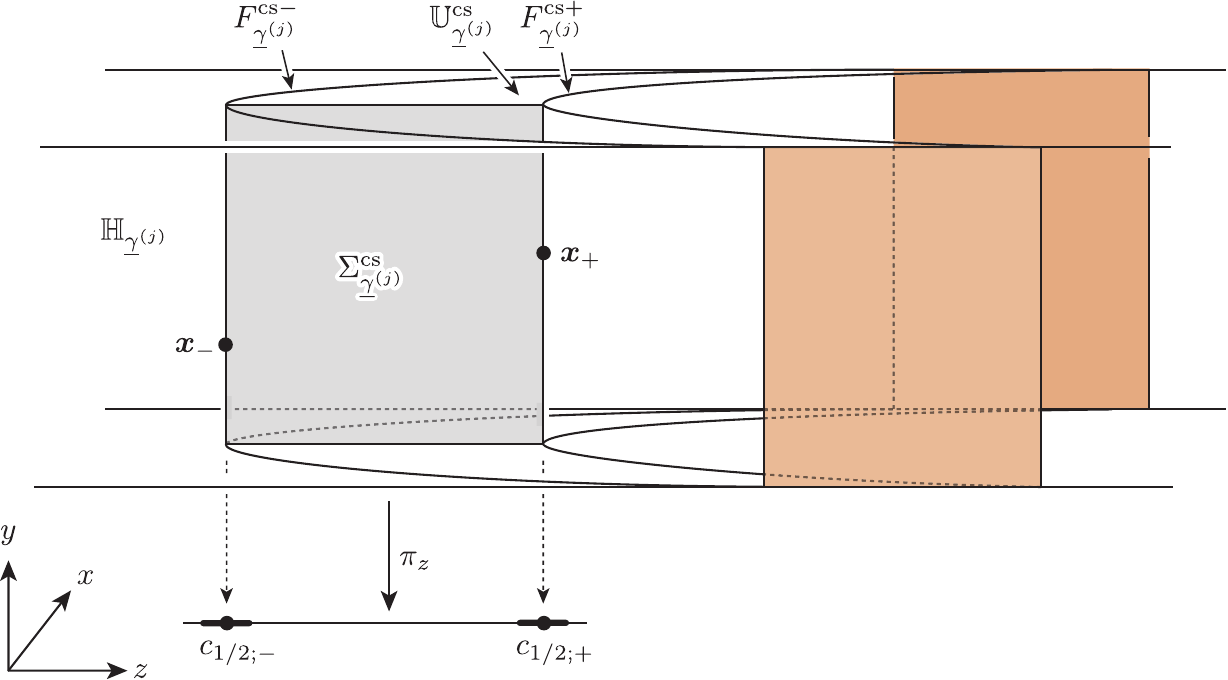}}
\caption{The case of $a_2>0$.}
\label{f_5_3}
\end{figure}
It follows that 
$$\pi_z(\partial_z\Sigma_{\underline{\gamma}^{(j)}}^{\mathrm{cs}})\subset [c_{1/2;-}-\varepsilon,c_{1/2;-}+\varepsilon]\cup 
[c_{1/2;+}-\varepsilon,c_{1/2;+}+\varepsilon],$$
where $\partial_z\Sigma_{\underline{\gamma}^{(j)}}^{\mathrm{cs}}=(F_{\underline{\gamma}^{(j)}}^{\mathrm{cs} -}\cup F_{\underline{\gamma}^{(j)}}^{\mathrm{cs} +})\cap \Sigma_{\underline{\gamma}^{(j)}}^{\mathrm{cs}}\index{Sggcs_Pz@$\partial_z\Sigma_{\underline{\gamma}^{(j)}}^{\mathrm{cs}}$}$.

\begin{lem}\label{l_S_not_BT}
Under the assumptions as above, the $\pi_z$-image $\pi_z(\Sigma_{\underline{\gamma}^{(m)}}^{\mathrm{cs}})$ of 
the back-end section $\Sigma_{\underline{\gamma}^{(m)}}^{\mathrm{cs}}$ contains $I(3\varepsilon)$.
\end{lem}
\begin{proof}
Since $\Sigma_{\underline{\gamma}^{(m-1)}}^{\mathrm{cs}}$ satisfies the backtracking condition, 
at least one of $[\varepsilon,\lambda_{\mathrm{cs} 0}-\varepsilon]$ and $[1-\lambda_{\mathrm{cs} 1}+\varepsilon,1-\varepsilon]$ contains 
$\pi_z(\Sigma_{\underline{\gamma}^{(m-1)}}^{\mathrm{cs}})$.
We set $\pi_z(\Sigma_{\underline{\gamma}^{(m-1)}}^{\mathrm{cs}})=[a,b]$ and $\pi_z(\Sigma_{\underline{\gamma}^{(m)}}^{\mathrm{cs}})=[a',b']$.

First we consider the case of $[a,b]\subset [1-\lambda_{\mathrm{cs} 1}+\varepsilon,1-\varepsilon]$.
If $\zeta_1^{-1}(a)>1-\lambda_{\mathrm{s} 1}+2\varepsilon$, then $a'>1-\lambda_{\mathrm{s} 1}+\varepsilon$ and hence 
$\pi_z(\Sigma_{\underline{\gamma}^{(m)}}^{\mathrm{cs}})\subset [1-\lambda_{\mathrm{cs} 1}+\varepsilon,1-\varepsilon]$.
See Figure \ref{f_5_4}.
\begin{figure}[hbtp]
\centering
\scalebox{0.6}{\includegraphics[clip]{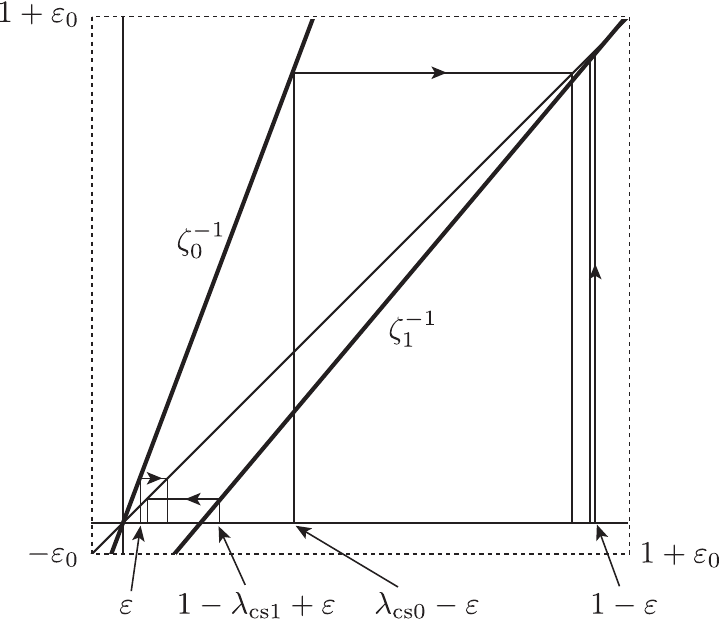}}
\caption{}
\label{f_5_4}
\end{figure}
If $\zeta_1^{-1}(b)<\lambda_{\mathrm{cs} 0}-2\varepsilon$, then $b'<\lambda_{\mathrm{cs} 0}-\varepsilon$ 
and hence $\pi_z(\Sigma_{\underline{\gamma}^{(m)}}^{\mathrm{cs}})\subset [\varepsilon,\lambda_{\mathrm{cs} 0}-\varepsilon]$.
In either case, it contradicts that $\pi_z(\Sigma_{\underline{\gamma}^{(m)}}^{\mathrm{cs}})$ is a back-end section.
Thus we have $[\zeta_1^{-1}(a),\zeta_1^{-1}(b)]\supset I(2\varepsilon)$.

Next we consider the case of $[a,b]\subset [\varepsilon,\lambda_{\mathrm{cs} 0}-\varepsilon]$.
If $\zeta_0^{-1}(a)>1-\lambda_{\mathrm{cs} 1}+2\varepsilon$, then 
$\pi_z(\Sigma_{\underline{\gamma}^{(m)}}^{\mathrm{cs}})\subset [1-\lambda_{\mathrm{cs} 1}+\varepsilon,1-\varepsilon]$.
If $\zeta_0^{-1}(b)<\lambda_{\mathrm{cs} 0}-2\varepsilon$, then 
$\pi_z(\Sigma_{\underline{\gamma}^{(m)}}^{\mathrm{cs}})\subset [\varepsilon,\lambda_{\mathrm{cs} 0}-\varepsilon]$.
In either case, we have again a contradiction, and hence 
$[\zeta_0^{-1}(a),\zeta_0^{-1}(b)]\supset I(2\varepsilon)$.

It follows from the two cases as above that $\pi_z(\Sigma_{\underline{\gamma}^{(m)}}^{\mathrm{cs}})$ contains $I(3\varepsilon)$ 
if $f$ is sufficiently $C^r$-closed to $f_0$.
\end{proof}

\section{Variation of tangent spaces of stable leaves}\label{S_variation}

This section provides geometric considerations to show Lemma \ref{l_xkinS} in Section \ref{S_perturb}.

In the case of dimension $>2$, we do not know whether the tangent plane 
$T_{\boldsymbol{x}}F^{\mathrm{cs}}(\boldsymbol{x})$ $C^1$-varies in contrast to the  
2-dimensional case, where $F^{\mathrm{cs}}(\boldsymbol{x})$ is the leaf of $\mathcal{F}_f^{\mathrm{cs}}$ containing $\boldsymbol{x}\in \mathbb{H}_{\varepsilon_0}$.
However Proposition \ref{c_Df2v} implies that the face angle $\omega$ between 
the tangent spaces of $F^{\mathrm{cs}}(\boldsymbol{x}_1)$ and $F^{\mathrm{cs}}(\boldsymbol{x}_2)$ is bounded by 
$C\|\boldsymbol{x}_1-\boldsymbol{x}_2\|$ for some constant $C>0$.
This fact is used to prove \eqref{eqn_Ak-E} in Section \ref{S_perturb}.
See Figure \ref{f_8_3} for the angle $\omega_k$ between the tangent space of 
$F^{\mathrm{cs}}(f^{\widehat{n}_k}(\widehat{\boldsymbol{x}}_k))$ and a line $l(f^{\widehat n_k}(\widehat{\boldsymbol{x}}_k))$ tangent to $F^{\mathrm{cs}}(\widehat{\boldsymbol{y}}_{k+1})$.
Our argument in this section is based on the fact that $f$ is sufficiently $C^2$-close to the affine model $f_0$ and 
hence in particular it satisfies \eqref{eqn_DDf}.

For $\boldsymbol{x}\in \mathbb{B}$, let $F^{\mathrm{s}}(\boldsymbol{x})$\index{Fsx@$F^{\mathrm{s}}(\boldsymbol{x})$} be the leaf of $\mathcal{F}_f^{\mathrm{s}}$ containing $\boldsymbol{x}$.
Consider the vectors $\boldsymbol{u}_0(\boldsymbol{x})$ and $\boldsymbol{u}_1(\boldsymbol{x})$ tangent to $F^{\mathrm{s}}(\boldsymbol{x})$ at $\boldsymbol{x}$ 
such that the $(y,z)$ entries of which are $(1,0)$ and $(0,1)$ respectively.
Since $F^{\mathrm{s}}(\boldsymbol{x})$ is adaptable to $\boldsymbol{C}_{\varepsilon}^{\mathrm{cs}}$,  
\begin{equation}\label{eqn_u0u1}
\boldsymbol{u}_0(\boldsymbol{x})=(O(\varepsilon),1,0)^T\quad\text{and}\quad\boldsymbol{u}_1(\boldsymbol{x})=(O(\varepsilon),0,1)^T,
\end{equation}
where $\boldsymbol{v}^T$ denotes the column vector obtained by transposing the row vector $\boldsymbol{v}$.
For any $\boldsymbol{x},\boldsymbol{x}'\in \mathbb{B}$, we naturally identify $T_{\boldsymbol{x}}\mathbb{B}$ and $T_{\boldsymbol{x}'}\mathbb{B}$ with $\mathbb{R}^3$.
So, for any $\boldsymbol{v}\in T_{\boldsymbol{x}} \mathbb{B}$ and $\boldsymbol{v}'\in T_{\boldsymbol{x}'} \mathbb{B}$, the sum $\boldsymbol{v}+ \boldsymbol{v}'$ is well 
defined.
In other words, $\boldsymbol{v}+\boldsymbol{v}'$ means $\boldsymbol{v}+ \tau_{(\boldsymbol{x}-\boldsymbol{x}')}\boldsymbol{v}'$ for 
the parallel transformation $\tau_{(\boldsymbol{x}-\boldsymbol{x}')}:T_{\boldsymbol{x}'}\mathbb{B}\longrightarrow T_{\boldsymbol{x}}\mathbb{B}$.

\begin{lem}\label{l_omega_k}
For any binary code $\underline{\gamma}^{(n)}$ of length $n$, let $\ell_n$ 
be a $C^1$-curve in $\mathbb{B}^{\mathrm{u}}(\underline{\gamma}^{(n)})$ adaptable to 
$\boldsymbol{C}_{\varepsilon}^{\mathrm{u}}$ and 
$\boldsymbol{x}_n^+$, $\boldsymbol{x}_n^-$ mutually distinct points of $\ell_n$.
Then 
$$\|\boldsymbol{u}_i(\boldsymbol{x}_n^+)-\boldsymbol{u}_i(\boldsymbol{x}_n^-)\|\leq \underline{\lambda}_{\mathrm{u}}^{-n}$$ 
holds for $i=0,1$.
\end{lem}
\begin{proof}
We set $\boldsymbol{x}_{n-j}^\pm=f^j(\boldsymbol{x}_n^\pm)$ and $\ell_{n-j}=f^j(\ell_n)$ for $j=1,\dots,n$.
Then $\boldsymbol{x}_{n-j}^+$ and $\boldsymbol{x}_{n-j}^-$ are points of $\mathbb{B}^{\mathrm{u}}(\underline{\gamma}^{(n-1)})$ contained in 
$\ell_{n-j}$, where $\underline{\gamma}^{(n-j)}$ is the code consisting of 
the latter $n-j$ entries of $\underline{\gamma}^{(n)}$.
Since $\boldsymbol{C}_{\varepsilon}^{\mathrm{u}}$ is a $f$-invariant cone-field, $\ell_{n-j}$ is adaptable to $\boldsymbol{C}_{\varepsilon}^{\mathrm{u}}$.
We prove inductively 
\begin{equation}\label{eqn_induction}
\|\boldsymbol{u}_i(\boldsymbol{x}_k^+)-\boldsymbol{u}_i(\boldsymbol{x}_k^-)\|\leq \underline{\lambda}_{\mathrm{u}}^{-k}
\end{equation}
for $k=0,1,\dots,n$.
Since $\mathcal{F}_f^{\mathrm{s}}$ is adaptable to $\boldsymbol{C}_{\varepsilon}^{\mathrm{cs}}$, \eqref{eqn_induction} holds 
for $k=0$.
Here we suppose that $1<m\leq n$ and \eqref{eqn_induction} holds for $k=0,1,\dots,m-1$ and 
set $\boldsymbol{u}_{i,k}^\pm=\boldsymbol{u}_i(\boldsymbol{x}_k^\pm)$.
The diagonal entries of $D(f^{-1})(\boldsymbol{x}_{m-1}^\pm)$ are $\lambda_{\mathrm{u}}^{-1}+O(\varepsilon)$, $\lambda_{\mathrm{ss}}^{-1}+O(\varepsilon)$ and 
$\lambda_{\mathrm{cs} j}^{-1}+O(\varepsilon)$ in order if $\boldsymbol{x}_m^\pm\in \mathbb{V}_{j,f}$ and any non-diagonal entry is $O(\varepsilon)$.
Hence, by \eqref{eqn_u0u1}, 
\begin{align*}
\widehat{\boldsymbol{u}}_{0,m}^\pm &:=D(f^{-1})(\boldsymbol{x}_{m-1}^\pm)\boldsymbol{u}_{0,m-1}^\pm=\bigl(O(\varepsilon),\lambda_{\mathrm{ss}}^{-1}+O(\varepsilon),O(\varepsilon)\bigr),\\
\widehat{\boldsymbol{u}}_{1,m}^\pm &:=D(f^{-1})(\boldsymbol{x}_{m-1}^\pm)\boldsymbol{u}_{1,m-1}^\pm=\bigl(O(\varepsilon),O(\varepsilon),\lambda_{\mathrm{cs} j}^{-1}+O(\varepsilon)\bigr).
\end{align*}
This shows that
\begin{equation}\label{eqn_|u|}
\begin{split}
\|\widehat{\boldsymbol{u}}_{0,m}^\pm\|&=\lambda_{\mathrm{ss}}^{-1}+O(\varepsilon)=\lambda_{\mathrm{ss}}^{-1}(1+O(\varepsilon)),\\
\|\widehat{\boldsymbol{u}}_{1,m}^\pm\|&=\lambda_{\mathrm{cs} j}^{-1}+O(\varepsilon)=\lambda_{\mathrm{cs} j}^{-1}(1+O(\varepsilon)).
\end{split}
\end{equation}
Since we assumed that \eqref{eqn_induction} holds for $k=m-1$, 
$\boldsymbol{u}_{1,m-1}^+-\boldsymbol{u}_{1,m-1}^-$ is represented as $(a_{m-1},0,0)^T$ for some $a_{m-1}$ with $|a_{m-1}|\leq \underline{\lambda}_{\mathrm{u}}^{-(m-1)}$.
Thus we have
$$D(f^{-1})(\boldsymbol{x}_{m-1}^+)(\boldsymbol{u}_{1,m-1}^+-\boldsymbol{u}_{1,m}^-)=
\bigl((\lambda_{\mathrm{u}}^{-1}+O(\varepsilon))a_{m-1}, O(\varepsilon)a_{m-1},O(\varepsilon)a_{m-1}\bigr).$$
It follows that 
$$\|D(f^{-1})(\boldsymbol{x}_{m-1}^+
)(\boldsymbol{u}_{1,m-1}^+-\boldsymbol{u}_{1,m-1}^-)\|\leq (\lambda_{\mathrm{u}}^{-1}+O(\varepsilon))\underline{\lambda}_{\mathrm{u}}^{-(m-1)}.$$
Since the derivative of any entry of $D(f^{-1})(\boldsymbol{x})$ with $\boldsymbol{x}\in \mathbb{B}\cap f^{-1}(\mathbb{B})$ is an $O(\varepsilon)$-function as 
\eqref{eqn_DDf} for $Df(\boldsymbol{x})$, by \eqref{eqn_lC0lam} 
\begin{align*}
\|(D(f^{-1})(\boldsymbol{x}_{m-1}^+)&-D(f^{-1})(\boldsymbol{x}_{m-1}^-))\boldsymbol{u}_{1,m-1}^-\|\leq 
O(\varepsilon)\|\boldsymbol{x}_{m-1}^+-\boldsymbol{x}_{m-1}^-\|\,\|\boldsymbol{u}_{1,m-1}^-\|\\
&\leq O(\varepsilon)\underline{\lambda}_{\mathrm{u}}^{-(m-1)}(1+O(\varepsilon))=O(\varepsilon)\underline{\lambda}_{\mathrm{u}}^{-(m-1)}.
\end{align*}
This shows that
\begin{subequations}
\begin{equation}\label{eqn_u1u1m}
\begin{split}
\|\widehat{\boldsymbol{u}}_{1,m}^+ -\widehat{\boldsymbol{u}}_{1,m}^-\|&=\|D(f^{-1})(\boldsymbol{x}_{m-1}^+)\boldsymbol{u}_{1,m-1}^+
-D(f^{-1})(\boldsymbol{x}_{m-1}^-)\boldsymbol{u}_{1,m-1}^-\|\\
&\leq \|D(f^{-1})(\boldsymbol{x}_{m-1}^+)(\boldsymbol{u}_{1,m-1}^+-\boldsymbol{u}_{1,m-1}^-)\|\\
&\hspace{40pt}+
\|(D(f^{-1})(\boldsymbol{x}_{m-1}^+)-D(f^{-1})(\boldsymbol{x}_{m-1}^-))\boldsymbol{u}_{1,m-1}^-\|\\
&\leq (\lambda_{\mathrm{u}}^{-1}+O(\varepsilon))\underline{\lambda}_{\mathrm{u}}^{-(m-1)}.
\end{split}
\end{equation}
Similarly one can show that 
\begin{equation}\label{eqn_u0u0m}
\|\widehat{\boldsymbol{u}}_{0,m}^+ -\widehat{\boldsymbol{u}}_{0,m}^-\|\leq (\lambda_{\mathrm{u}}^{-1}+O(\varepsilon))\underline{\lambda}_{\mathrm{u}}^{-(m-1)}.
\end{equation}
\end{subequations}

Let $A_1^\pm$ be the points of $\mathbb{R}^3$ with $\overrightarrow{OA_1^+}=\widehat{\boldsymbol{u}}_{1,m}^+$ and $\overrightarrow{OA_1^-}=\widehat{\boldsymbol{u}}_{1,m}^-$ and 
$P_0$ the $xz$-plane in $\mathbb{R}^3$.
We denote by $l^\pm$ the lines in $\mathbb{R}^3$ passing through $A_1^\pm$ and parallel to $\widehat{\boldsymbol{u}}_{0,m}^-$ 
and set $C=l^+\cap P_0$ and $A_2^-=l^-\cap P_0$.
Suppose that $B$ is a point of $\mathbb{R}^3$ such that either $\overrightarrow{A_1^+B}=\widehat{\boldsymbol{u}}_{0,m}^+$ or $\overrightarrow{A_1^+B}=-\widehat{\boldsymbol{u}}_{0,m}^+$ 
and the straight segment $\overline{A_1^+B}$ connecting $A_1^+$ with $B$ meets $P_0$ non-trivially.
The intersection point is denoted by $A_2^+$.
In the case of $A_1^+\in P_0$, $A_2^+=A_1^+=C$.
Let $B'$ be the point in $l^+$ which lies in the side same as $B$ with respect to $P_0$ and such that the 
length of $\overline{A_1^+B'}$ is $\|\widehat{\boldsymbol{u}}_{0,m}^-\|$. 
Since $\widehat{\boldsymbol{u}}_{0,m}^\pm,\widehat{\boldsymbol{u}}_{1,m}^\pm \in T_{\boldsymbol{x}_m^\pm}F^{\mathrm{s}}(\boldsymbol{x}_m^\pm)$, we have 
$\overrightarrow{OA_2^+}\in T_{\boldsymbol{x}_m^+}F^{\mathrm{s}}(\boldsymbol{x}_m^+)\cap P_0$ and 
$\overrightarrow{OA_2^-}\in T_{\boldsymbol{x}_m^-}F^{\mathrm{s}}(\boldsymbol{x}_m^-)\cap P_0$.
Let $A_3^\pm$ be the intersection points of $\overline{OA_2^\pm}$ and the line $z=1$ in $P_0$.
See Figure \ref{f_6_1}.
\begin{figure}[hbtp]
\centering
\scalebox{0.6}{\includegraphics[clip]{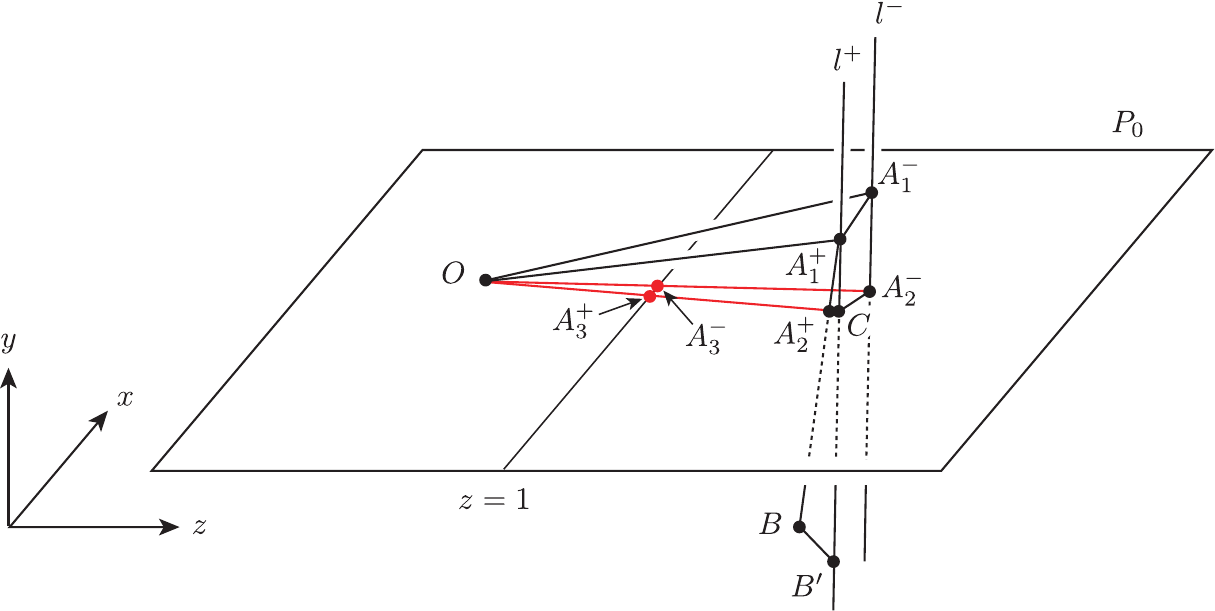}}
\caption{}
\label{f_6_1}
\end{figure}
From the construction, we know that $\boldsymbol{u}_{1,m}^+=\overrightarrow{OA_3^+}$ and $\boldsymbol{u}_{1,m}^-=\overrightarrow{OA_3^-}$.
By \eqref{eqn_u1u1m}, $\|A_1^+-A_1^-\|\leq (\lambda_{\mathrm{u}}^{-1}+O(\varepsilon))\underline{\lambda}_{\mathrm{u}}^{-(m-1)}$.
By \eqref{eqn_u0u0m}, both $l^+$ and $l^-$ meet $P_0$ $O(\varepsilon)$-\emph{almost orthogonally}.
It follows that $\overline{CA_2^-}$ meets $l^+$ and $l^-$ $O(\varepsilon)$-almost orthogonally and 
hence $\|C-A_2^-\| \leq  (\lambda_{\mathrm{u}}^{-1}+O(\varepsilon))\underline{\lambda}_{\mathrm{u}}^{-(m-1)}$.
By \eqref{eqn_u0u0m}, $\|B-B'\|\leq  (\lambda_{\mathrm{u}}^{-1}+O(\varepsilon))\underline{\lambda}_{\mathrm{u}}^{-(m-1)}$.
Since $\|A_1^+-A_2^+\|=O(\varepsilon)\|A_1^+-B\|$ and $\|A_1^+-C\|=O(\varepsilon)\|A_1^+-B'\|$, 
we have $\|A_2^+-C\|\leq O(\varepsilon)(\lambda_{\mathrm{u}}^{-1}+O(\varepsilon))\underline{\lambda}_{\mathrm{u}}^{-(m-1)}$ and hence 
$\|A_2^+-A_2^-\|\leq (\lambda_{\mathrm{u}}^{-1}+O(\varepsilon))\underline{\lambda}_{\mathrm{u}}^{-(m-1)}$.
Then, by \eqref{eqn_|u|}, $\|A_3^\pm\|\leq \lambda_{\mathrm{cs} j}(1+O(\varepsilon))\|A_2^\pm\|$.
It follows that
$$\|\boldsymbol{u}_{1,m}^+-\boldsymbol{u}_{1,m}^-\|=\|A_3^+-A_3^-\|\leq 
\lambda_{\mathrm{cs} j}(1+O(\varepsilon))(\lambda_{\mathrm{u}}^{-1}+O(\varepsilon))\underline{\lambda}_{\mathrm{u}}^{-(m-1)}\leq \underline{\lambda}_{\mathrm{u}}^{-m}.$$
This completes the induction.
The proof of $\|\boldsymbol{u}_{0,m}^+-\boldsymbol{u}_{0,m}^-\|\leq \underline{\lambda}_{\mathrm{u}}^{-m}$ is done quite similarly.
\end{proof}

The following proposition is used in the proof of Lemma \ref{lem-3-1}.
See also Remark \ref{r_Df2v} for the role.

\begin{prop}\label{c_Df2v}
Under the notations as in Lemma \ref{l_omega_k}, 
suppose that $\boldsymbol{x}_n^+,\boldsymbol{x}_n^-\in f^2(\mathbb{H}_{\varepsilon_0})\cap \mathbb{B}^{\mathrm{u}}(\underline{\gamma}^{(n)})$.
Then there exists a constant $C_1>0$ independent of $n$ and satisfying 
$$\|D(f^{-2})(\boldsymbol{x}_n^+)\boldsymbol{u}_i(\boldsymbol{x}_n^+)-D(f^{-2})(\boldsymbol{x}_n^-)\boldsymbol{u}_i(\boldsymbol{x}_n^-)\|\leq C_1\underline{\lambda}_{\mathrm{u}}^{-n}$$
for $i=0,1$.
\end{prop}
\begin{proof}
As in the proof of Lemma \ref{l_omega_k}, we set $\boldsymbol{u}_i(\boldsymbol{x}_n^+)=\boldsymbol{u}_{i,n}^+$ and $\boldsymbol{u}_i(\boldsymbol{x}_n^-)=\boldsymbol{u}_{i,n}^-$.
By the mean value theorem together with \eqref{eqn_lC0lam}, 
\begin{align*}
\|D(f^{-2})&(\boldsymbol{x}_n^+)\boldsymbol{u}_{i,n}^+-D(f^{-2})(\boldsymbol{x}_n^-)\boldsymbol{u}_{i,n}^-\|\\
&\leq 
\| D(f^{-2})(\boldsymbol{x}_n^+)\|\,\|\boldsymbol{u}_{i,n}^+-\boldsymbol{u}_{i,n}^-\|+
\|D(f^{-2})(\boldsymbol{x}_n^+)-D(f^{-2})(\boldsymbol{x}_n^-)\|\,\|\boldsymbol{u}_{i,n}^-\|\\
&\leq C_{10}\|\boldsymbol{u}_{i,n}^+-\boldsymbol{u}_{i,n}^-\|+C_{11}\|\boldsymbol{x}_n^+-\boldsymbol{x}_n^-\|(1+O(\varepsilon))\\
&\leq (C_{10}+C_{11}C_0^{-1}(1+O(\varepsilon)))\underline{\lambda}_{\mathrm{u}}^{-n},
\end{align*}
where $C_{10}=\max\bigl\{|D(f^{-2})(\boldsymbol{x})|\,;\,\boldsymbol{x}\in f^2(\mathbb{H}_{\varepsilon_0})\bigr\}$ and 
$$C_{11}=\max\left\{ \left|\frac{\partial^2(\pi_a\circ f^{-2})}{\partial x_j\partial x_k}(\boldsymbol{x})\right|\,;\, 
a,x_j,x_k\in \{x,y,z\},\boldsymbol{x}\in f^2(\mathbb{H}_{\varepsilon_0})\right\}.$$
Hence the required inequality is obtained by setting $C_1=C_{10}+2C_{11}C_0^{-1}$.
\end{proof}

\section{Backward sequences of cs-curved blocks}\label{S_BWS}

In this section, we specify the cs-section $\widehat\Sigma_k^{\mathrm{cs}}$ associated with $B_k^{\mathrm{u}}$ 
and show that, if $\Sigma_{\underline{\gamma}^{(m_k)}}^{\mathrm{cs}}$ is a back-end section based at 
$\widehat\Sigma_k^{\mathrm{cs}}$, then the length $m_k$ is $O(k)$.
See Figure \ref{f_8_1} for the situation.
The code $\underline{\gamma}^{(m_k)}$ obtained here is a part of 
the code $\widehat{\underline{w}}_k$ defined in Lemma \ref{lem-3-1}.

Recall that $B_k^{\mathrm{u}}=B^{\mathrm{u}}(\underline{w}^{(n_0+k)})$ is the $\mathrm{u}$-bridge of \eqref{eqn_BBk}.
For a fixed integer $L\geq 4$, consider any sequence of sub-bridges 
$B^{\mathrm{u}}(\underline{w}^{(n_0+Lk)})$ of $B_k^{\mathrm{u}}$ 
such that 
$\underline{w}^{(n_0+Lk)}=\underline{w}^{(n_0+k)}\underline{\nu}^{(Lk-k)}$\index{wn0Lk@$\underline{w}^{(n_0+Lk)}$}  for binary codes $\underline{\nu}^{(Lk-k)}$ of length $Lk-k$.
In Lemma \ref{l_psi_n}, $L$ will be taken so that 
$L>3r$.
By \eqref{eqn_C0GC0},
\begin{equation}\label{eqn_CBuC}
C_0\bar{\lambda}_{\mathrm{u}}^{-(n_0+Lk)}<|B^{\mathrm{u}}(\underline{w}^{(n_0+Lk)})|<C_0^{-1}\underline{\lambda}_{\mathrm{u}}^{-(n_0+Lk)}.
\end{equation}
Here we consider the cs-curved block 
$\mathbb{U}_k^{\mathrm{cs}}=f^{-2}(\mathbb{B}^{\mathrm{u}}(\underline{w}^{(n_0+Lk)}))\cap \mathbb{H}_{\varepsilon_0}$\index{Ucsgk@$\mathbb{U}_k^{\mathrm{cs}}$} 
and the cs-section  
$$\widehat\Sigma_{k}^{\mathrm{cs}}=\mathbb{U}_{k}^{\mathrm{cs}}\cap S^{\mathrm{cs}},
\index{Sgkcsh@$\widehat{\Sigma}_{k}^{\mathrm{cs}}$}$$
of $\mathbb{U}_{k}^{\mathrm{cs}}$.
The \emph{width} of $\mathbb{U}_k^{\mathrm{cs}}$ is defined as
$$\mathrm{width}(\mathbb{U}_{k}^{\mathrm{cs}})=\min\{\|\boldsymbol{x}_- -\boldsymbol{x}_+\|\,;\,
\boldsymbol{x}_-\in F^{\mathrm{cs} -}, \boldsymbol{x}_+\in F^{\mathrm{cs}+}\},$$ 
where $F^{\mathrm{cs} -}$ and $F^{\mathrm{cs}+}$ are the left and right components of $\partial_z\mathbb{U}_k^{\mathrm{cs}}$ respectively.
See Figure \ref{f_5_2} again.
From the definitions of $\mathbb{U}_{k}^{\mathrm{cs}}$ together with \eqref{eqn_CBuC}, 
there exists a constant $0<C_2<1$ independent of $k$ and such that
\begin{equation}\label{eqn_piB}
C_0C_2\bar{\lambda}_{\mathrm{u}}^{\,-(n_0+Lk)}
\leq \mathrm{width}(\mathbb{U}_k^{\mathrm{cs}})\leq (C_0C_2)^{-1}\underline{\lambda}_{\mathrm{u}}^{\,-(n_0+Lk)}.
\end{equation}

We use the notations given in Section \ref{S_BTC} by letting $\mathbb{U}_{k+1}^{\mathrm{cs}}=\mathbb{U}^{\mathrm{cs}}$ and 
$\widehat\Sigma_{k+1}^{\mathrm{cs}}=\Sigma^{\mathrm{cs}}$.
Suppose that $\underline{\gamma}^{(m_k)}=\gamma_{m_k}\gamma_{m_k-1}\dots \gamma_2\gamma_1$\index{gmk@$\underline{\gamma}^{(m_k)}$} is a binary code such that 
$\Sigma_{\underline{\gamma}^{(m_k)}}^{\mathrm{cs}}$ is a back-end section based at $\widehat{\Sigma}_{k+1}^{\mathrm{cs}}$.
Strictly $\mathbb{U}_{\underline{\gamma}^{(m_k)}}^{\mathrm{cs}}=(\mathbb{U}_{k+1}^{\mathrm{cs}})_{\underline{\gamma}^{(m_k)}}$ and 
$\Sigma_{\underline{\gamma}^{(m_k)}}^{\mathrm{cs}}=(\widehat{\Sigma}_{k+1}^{\mathrm{cs}})_{\underline{\gamma}^{(m_k)}}$.

\begin{lem}\label{l_transf}
There exist positive integers $N_1,N_2$ independent of $k$ or $\underline{\gamma}^{(m_k)}$ such that 
$m_k\leq N_0+N_1k$.
\end{lem}
\begin{proof}
We consider the case of $a_2>0$.
Fix a point $\boldsymbol{x}_{m_k}^+$ of $\Sigma_{\underline{\gamma}^{(m_k)}}^{\mathrm{cs}}\cap F_{\underline{\gamma}^{(m_k)}}^{\mathrm{cs}+}$ and define the points 
$\boldsymbol{x}_j^+\in \Sigma_{\underline{\gamma}^{(j)}}^{\mathrm{cs}}\cap F_{\underline{\gamma}^{(j)}}^{\mathrm{cs}+}$ by
$\boldsymbol{x}_{j}^+=f^{m_k-j}(\boldsymbol{x}_{m_k}^+)$ for $j=m_k-1,m_k-2,\dots,1,0$.
Let $F^{\mathrm{s}}(\boldsymbol{x}_j^+)$ be the leaf of $\mathcal{F}_f^{\mathrm{s}}$ containing $\boldsymbol{x}_j^+$ and 
$P_j$ the plane in $\mathbb{B}$ with $P_j\ni \boldsymbol{x}_j^+$ and parallel to the $xz$-plane.
We denote by $\sigma_j$ an arc in $F^{\mathrm{s}}(\boldsymbol{x}_j^+)\cap P_j$ connecting $\boldsymbol{x}_j^+$ with a point of 
$F_{\underline{\gamma}^{(j)}}^{\mathrm{cs}-}\cap P_j$.
See Figure \ref{f_7_1}.
\begin{figure}[hbtp]
\centering
\scalebox{0.6}{\includegraphics[clip]{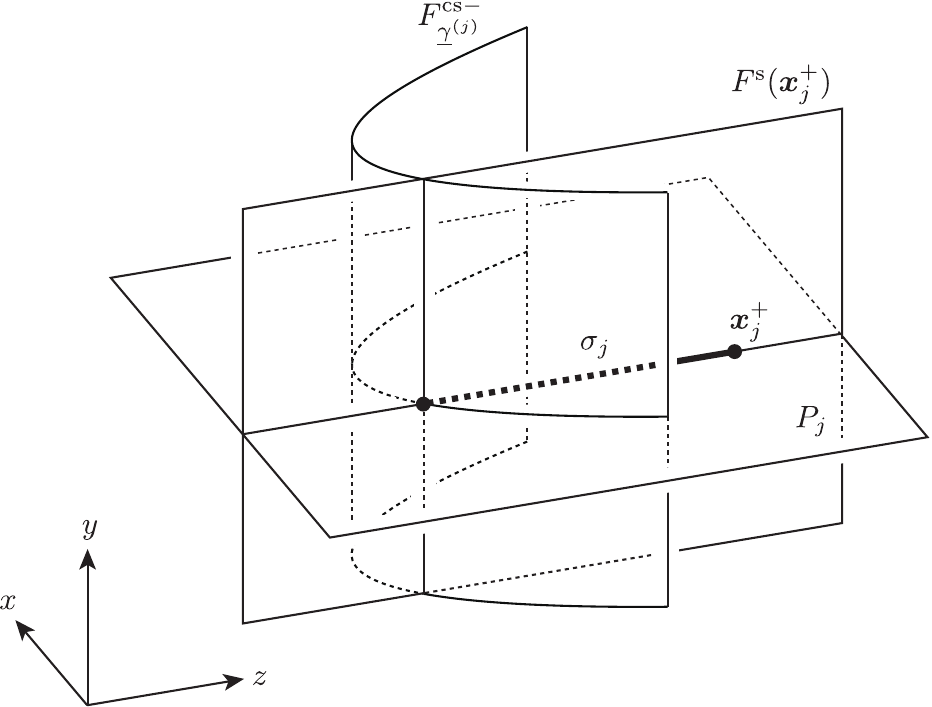}}
\caption{}
\label{f_7_1}
\end{figure}
Since $F^{\mathrm{s}}(\boldsymbol{x}_j^+)$ is adaptable to the cone-field $\boldsymbol{C}_{\varepsilon}^{\mathrm{cs}}$, we have 
$$\mathrm{length}(\sigma_j)=|\pi_z(\sigma_j)|(1+O(\varepsilon))\quad\text{and}\quad |\pi_x(\sigma_j)|=O(\varepsilon).$$
Since by \eqref{eqn_f0V} $Df(\boldsymbol{x})$ is arbitrarily $C^{r-1}$-close to the diagonal matrix $Df_0(\boldsymbol{x})=\mathrm{diag}((-1)^i\lambda_{\mathrm{u}},(-1)^i\lambda_{\mathrm{ss}},\lambda_{\mathrm{cs} i})$ for $\boldsymbol{x}\in \mathbb{V}_{i,f}$, 
we may assume that an non-diagonal entry of $Df(\boldsymbol{x})$ has $O(\varepsilon)$-value.
It follows from this fact together with \eqref{eqn_lambda_maxS} and \eqref{eqn_lambda_max} that
\begin{equation}\label{eqn_pi_f_sg}
\begin{split}
|\pi_z(f^{-1}(\sigma_j))|&\geq \left(\left(\bar{\lambda}_{\mathrm{cs} 1}-\frac{\varepsilon}2\right)^{-1}+O(\varepsilon)\right)|\pi_z(\sigma_j)|
\geq \left(\bar{\lambda}_{\mathrm{cs} 1}^{-1}+O(\varepsilon)\right)|\pi_z(\sigma_j)|\\
&\geq \bar\lambda_{\mathrm{cs} 1}^{-1/2}|\pi_z(\sigma_j)|,
\end{split}
\end{equation}\label{eqn_len_f_sg}
\begin{equation}
\mathrm{length}(\pi_x(f^{-1}(\sigma_j)))\leq \left(\underline{\lambda}_{\mathrm{u}}+\frac{\varepsilon}2\right)^{-1}O(\varepsilon)=O(\varepsilon),
\end{equation}
where we use the fact that $\bar\lambda_{\mathrm{cs} 0}^{-1}>\bar\lambda_{\mathrm{cs} 1}^{-1}$.
Since $\mathcal{F}_f^{\mathrm{s}}$ is $f$-invariant, $f^{-1}(\sigma_j)$ and $\sigma_{j+1}$ are contained in the same 
leaf $F^{\mathrm{s}}(\boldsymbol{x}_{j+1}^+)$ of $\mathcal{F}_f^{\mathrm{s}}$.
Since $F^{\mathrm{s}}(\boldsymbol{x}_{j+1}^+)$ is adaptable to $\boldsymbol{C}_{\varepsilon}^{\mathrm{cs}}$, 
by Lemma \ref{l_Fcs} $F_{\underline{\gamma}^{(j+1)}}^{\mathrm{cs}-}\cap F^{\mathrm{s}}(\boldsymbol{x}_{j+1}^+)$ is an 
$O(\varepsilon)$-almost vertical arc which contains end points of $f^{-1}(\sigma_j)$ and $\sigma_{j+1}$ other than $\boldsymbol{x}_{j+1}^+$.
This implies that  
$$|\pi_z(\sigma_{j+1})|=|\pi_z(f^{-1}(\sigma_{j}))|(1+O(\varepsilon)).$$
Hence, by \eqref{eqn_pi_f_sg}, $|\pi_z(\sigma_{j+1})|\geq \bar{\lambda}_{\mathrm{cs} 1}^{-1/3}|\pi_z(\sigma_j)|$.
This shows that
\begin{equation}\label{eqn_pi_sg_m}
|\pi_z(\sigma_{m_k})|\geq \bar{\lambda}_{\mathrm{cs} 1}^{-m_k/3}|\pi_z(\sigma_0)|.
\end{equation}
Let $\boldsymbol{x}_{m_k}^-$ be a point of $F_{\underline{\gamma}^{(m_k)}}^{\mathrm{cs}-}\cap \Sigma_{\underline{\gamma}^{(m_k)}}^{\mathrm{cs}}\cap P_{m_k}$ 
and $\boldsymbol{x}_{m_k}^0$ the end point of $\sigma_{m_k}$ other than $\boldsymbol{x}_{m_k}^+$.
By \eqref{eqn_pi_z_xq},
$\pi_z(\boldsymbol{x}_{m_k}^-)< \pi_z(\boldsymbol{x}_{m_k}^0)+\varepsilon$.
See Figure \ref{f_7_2}.
\begin{figure}[hbtp]
\centering
\scalebox{0.6}{\includegraphics[clip]{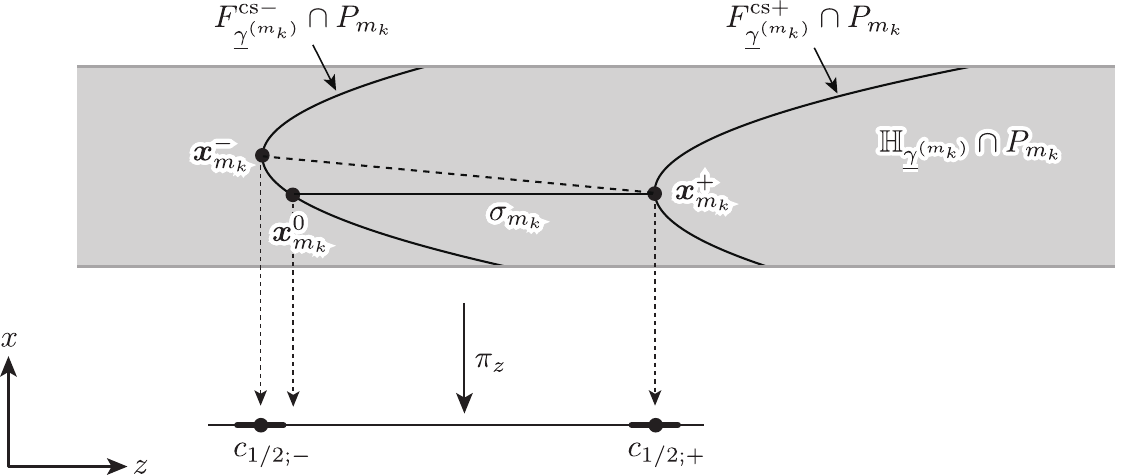}}
\caption{View from the top.}
\label{f_7_2}
\end{figure}
By \eqref{eqn_pi_sg_m}, there exists a constant $C_3>0$ independent of $\underline{\gamma}^{(m_k)}$ and satisfying
$$
|\pi_z(\Sigma_{\underline{\gamma}^{(m_k)}}^{\mathrm{cs}})|\geq C_3\bar{\lambda}_{\mathrm{cs} 1}^{-m_k/3}
|\pi_z(\sigma_0)|.
$$
Since $\sigma_0$ is contained in $F^{\mathrm{s}}({\boldsymbol{x}}_0^+)\cap P_0$ and $F^{\mathrm{s}}({\boldsymbol{x}}_0^+)$ is a plane parallel to the $yz$-plane 
by the condition \eqref{F2} on $\mathcal{F}_f^{\mathrm{s}}$ given in Subsection \ref{ss_cone_foliation}, 
$\sigma_0$ is a straight segment parallel to the $z$-axis.
This shows that 
$$|\pi_z(\sigma_0)|=\mathrm{length}(\sigma_0)\geq \mathrm{width}(\mathbb{U}_{k+1}^{\mathrm{cs}}).$$
Hence we have
$$C_3\bar{\lambda}_{\mathrm{cs} 1}^{-m_k/3}\mathrm{width}(\mathbb{U}_{k+1}^{\mathrm{cs}})\leq |\pi_z(\Sigma_{\underline{\gamma}^{(m_k)}}^{\mathrm{cs}})|
<1+2\varepsilon_0<2.$$
By this fact together with \eqref{eqn_piB} that 
$\bar{\lambda}_{\mathrm{cs} 1}^{-m_k/3}\leq 2(C_0C_2C_3)^{-1}\bar{\lambda}_{\mathrm{u}}^{(n_0+Lk)}$.
It follows that
$$m_k\leq \frac{3\log\bigl(2(C_0C_2C_3)^{-1}\bar{\lambda}_{\mathrm{u}}^{n_0}\bigr)}{\log \bar{\lambda}_{\mathrm{cs} 1}^{-1}}
+\frac{3\log \bar{\lambda}_{\mathrm{u}}^{L}}{\log \bar{\lambda}_{\mathrm{cs} 1}^{-1}}k.$$
Let $N_0$ and $N_1$ be the smallest positive integers with 
$$N_0\geq \frac{3\log\bigl(2(C_0C_2C_3)^{-1}\bar{\lambda}_{\mathrm{u}}^{n_0}\bigr)}{\log \bar{\lambda}_{\mathrm{cs} 1}^{-1}}
\quad\text{and}\quad 
N_1\geq \frac{3\log \bar{\lambda}_{\mathrm{u}}^{L}}{\log \bar{\lambda}_{\mathrm{cs} 1}^{-1}}.$$
Then $m_k\leq N_0+N_1k$.
This completes the proof in the case of $a_2>0$.

When $a_2<0$, one can prove the lemma quite similarly by considering a 
point $\boldsymbol{x}_{m_k}^-$ of $\Sigma_{\underline{\gamma}^{(m_k)}}^{\mathrm{cs}}\cap F_{\underline{\gamma}^{(m_k)}}^{\mathrm{cs}-}$ 
instead of $\boldsymbol{x}_{m_k}^+$.
\end{proof}

\section{$C^r$-perturbations of $f$}\label{S_perturb}
In Subsection \ref{ss_binary_free}, we define the binary code $\widehat{\underline{w}}_k$ the main part $\underline{u}_k$ 
of which can be chosen freely and the front and back complements are used to connect 
$\widehat{\underline{w}}_k$ with $\widehat{\underline{w}}_{k-1}$ and 
$\widehat{\underline{w}}_{k+1}$ respectively.
Furthermore we present an $f$-pseudo-orbit
\begin{equation}\label{eqn_pseudo_orb}
(\dots, \widehat{\boldsymbol{x}}_k, f(\widehat{\boldsymbol{x}}_k),\dots,f^{\widehat{n}_k}(\widehat{\boldsymbol{x}}_k), 
f(\widehat{\boldsymbol{y}}_{k+1}),\widehat{\boldsymbol{x}}_{k+1},\dots)
\end{equation}
as illustrated in Figure \ref{f_8_1}, where $\widehat{\boldsymbol{y}}_{k+1}$ is 
a point of $f^{-2}(S_{\widehat{\underline{w}}_{k+1}}^{\mathrm{cs}})\cap \mathbb{H}_{\varepsilon_0}$ 
which is $O(\underline{\lambda}_{\mathrm{u}}^{-(n_0+L(k+1))})$--close to $f^{\widehat{n}_k}(\widehat{\boldsymbol{x}}_k)$. 
\begin{figure}[hbtp]
\centering
\scalebox{0.6}{\includegraphics[clip]{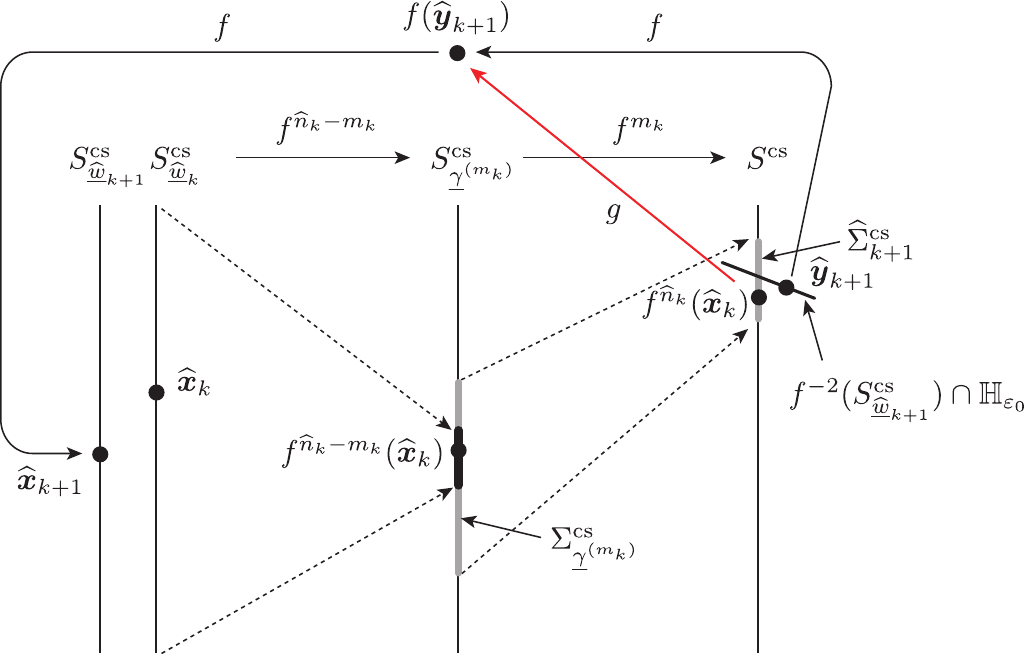}}
\caption{}
\label{f_8_1}
\end{figure}
In Subsection \ref{ss_bump}, 
we define a diffeomorphism $g$ by a $C^r$-perturbation of $f$ supported in a small neighborhood of 
$f^{\widehat{n}_k}(\widehat{\boldsymbol{x}}_k)$ such that 
$g(f^{\widehat{n}_k}(\widehat{\boldsymbol{x}}_k))$ coincides with $f(\widehat{\boldsymbol{y}}_{k+1})$.
In particular, the sequence \eqref{eqn_pseudo_orb} is an actual orbit of $g$.

\subsection{Binary codes with free parts and mutually disjoint cubes}\label{ss_binary_free}
Recall that 
$S_{\underline{\gamma}^{(k)}}^{\mathrm{cs}}$
and $\widehat{\Sigma}_{k}^{\mathrm{cs}}$
are the cs-sections of 
$\mathbb{H}_{\underline{\gamma}^{(k)}}$
and $\mathbb{U}_{k}^{\mathrm{cs}}$
defined in Sections \ref{S_BTC} and \ref{S_BWS} respectively.

\begin{lem}\label{lem-3-1}
Let $f$ be any element of $\mathcal{O}(f_0)$ and 
$\underline{w}^{(n_0+Lk)}$ the binary code given Section \ref{S_BWS}.
For any binary code $\underline{u}_k$ with arbitrary finite length, 
there exists a binary code $\underline{\widehat w}_k$ 
satisfying the following {\rm (1)} and {\rm (2)}.
\begin{enumerate}[\rm (1)]
\item\label{lem-3-1(1)}
$\underline{\widehat w}_k$\index{wkhk@$\underline{\widehat w}_k$} 
is represented as $\underline{w}^{(n_0+Lk)}\underline{u}_k\underline{\iota}_k\underline{\gamma}^{(m_k)}$, 
where 
$\underline{\iota}_k$ and $\underline{\gamma}^{(m_k)}$ are binary codes given as follows. 
\begin{itemize}
\item
The length of $\underline{\iota}_k$ is at most $\mu_0$ (possibly $\underline{\iota}_k=\emptyset$), where 
$\mu_0$ is the constant given in Lemma \ref{l_zeta_zeta}, 
\item
$\underline{\gamma}^{(m_k)}=\gamma_{m_k}\gamma_{m_k-1}\dots \gamma_2\gamma_1$ 
for some $0<m_k\leq N_0+N_1k$, where $N_0$ and $N_1$ are the positive integers given in Lemma \ref{l_transf}.
\end{itemize}
\item \label{lem-3-1(2)}
$f^{|\widehat{\underline{w}}_k|}(S_{\widehat{\underline{w}}_k}^{\mathrm{cs}})$ is contained in $\widehat{\Sigma}_{k+1}^{\mathrm{cs}}$.
\end{enumerate}
\end{lem}

\begin{proof}
By Lemma \ref{l_fSw}, there exists a binary cord $\underline{\iota}_k$ of length at most $\mu_0$ such that 
$\pi_z(f^{n_0+Lk+|\underline{u}_k|+|\underline{\iota}_k|}(S_{\widehat{\underline w}_k}^{\mathrm{cs}}))$ is contained in $I(4\varepsilon)$, 
where 
$\widehat{\underline{w}}_k=\underline{w}^{(n_0+Lk)}\underline{u}_k\underline{\iota}_k\underline{\gamma}^{(m_k)}$.
On the other hand, by Lemma \ref{l_S_not_BT}, for a back-end section $\Sigma_{\underline{\gamma}^{(m_k)}}^{\mathrm{cs}}$ of $\widehat{\Sigma}_{k+1}^{\mathrm{cs}}$, $\pi_z(\Sigma_{\underline{\gamma}^{(m_k)}}^{\mathrm{cs}})$ contains $I(3\varepsilon)$.
It follows that $f^{n_0+Lk+|\underline{u}_k|+|\underline{\iota}_k|}(S_{\widehat{\underline{w}}^{(m_k)}}^{\mathrm{cs}})\subset \Sigma_{\underline{\gamma}^{(m_k)}}^{\mathrm{cs}}$ and hence 
$f^{|\widehat{\underline{w}}_k|}(S_{\widehat{\underline{w}}_k}^{\mathrm{cs}})$ is contained in $\widehat{\Sigma}_{k+1}^{\mathrm{cs}}$.
This shows the assertion (2).
\end{proof}

\begin{rmk}\label{r_Lemma8.1} 
(1) 
The freedom on the choice of the sub-codes $\underline{u}_k$ in Lemma \ref{lem-3-1} is one of the essential ideas of this paper, which is 
a generalization of \cite[page 4015, Lemma 1.4]{KNS23}\footnote{Note that \cite{KNS23} contains two Lemma 1.4 
due to some editorial mistake.}. 
Such an idea of incorporating a free choice of sub-codes comes from \cite{CV01}.
We will see in subsequent sections that 
this is a mechanism to realize the pluripotency of wandering domains.

\noindent(2) 
The length $|\underline{\iota}_k|$ of $\underline{\iota}_k$ depends on the choice of $\underline{u}_k$, which is 
crucial in the process of determining $\alpha_k$ in Subsection \ref{ss_quad_majority}.
\end{rmk}

We set $|\widehat{\underline{w}}_k|=\widehat n_k$\index{nkh@$\widehat n_k$} for short.
From the definition of the binary code $\widehat{\underline{w}}_k$ in Lemma \ref{lem-3-1}, 
\begin{equation}\label{eqn_hatn0}
\widehat n_k=n_0+Lk+|\underline{u}_k|+|\underline{\iota}_k|+m_k.
\end{equation}

\begin{lem}\label{l_xkinS}
There exists a sequence $(\widehat{\boldsymbol{x}}_k)_{k\geq 1}\index{x_k@$\widehat{\boldsymbol{x}}_k$}$ with $\widehat{\boldsymbol{x}}_k\in S_{\widehat{\underline{w}}_k}^{\mathrm{cs}}$ and satisfying $f^{-2}(\widehat{\boldsymbol{x}}_{k+1})\in \mathbb{U}_{k+1}^{\mathrm{cs}}$ and 
$\|f^{\widehat n_k}(\widehat{\boldsymbol{x}}_k)- f^{-2}(\widehat{\boldsymbol{x}}_{k+1})\|=O(\underline{\lambda}_{\mathrm{u}}^{-(n_0+L(k+1))})$.
\end{lem}
\begin{proof}
Let $\widehat{\boldsymbol{x}}_1$ be any element  of $\Sigma_{\widehat{\underline{w}}_1}^{\mathrm{cs}}\setminus W_{\mathrm{loc}}^{\mathrm{u}}(\Lambda_f)$ and suppose that 
$\widehat{\boldsymbol{x}}_1,\dots,\widehat{\boldsymbol{x}}_k$ are already determined.
By Lemma \ref{lem-3-1}\,\eqref{lem-3-1(2)}, $f^{\widehat{n}_k}(\widehat{\boldsymbol{x}}_k)$ is an element of 
$\widehat{\Sigma}_{k+1}^{\mathrm{cs}}=\mathbb{U}_{k+1}^{\mathrm{cs}}\cap S^{\mathrm{cs}}$.
We denote by $\mathcal{M}_{k+1}$ the 1-dimensional foliation on $\mathbb{B}^{\mathrm{u}}(\underline{w}^{(n_0+L(k+1))})$ consisting of 
maximal segments in $\mathbb{B}^{\mathrm{u}}(\underline{w}^{(n_0+L(k+1))})$ parallel to the $x$-axis.
Since $f^{-2}(\mathbb{B}^{\mathrm{u}}(\underline{w}^{(n_0+L(k+1))}))\supset \mathbb{U}_{k+1}^{\mathrm{cs}}$, 
there exists a unique leaf $l$ of $\mathcal{M}_{k+1}$ such that $f^{-2}(l)$ passes through $f^{\widehat n_k}(\widehat{\boldsymbol{x}}_k)$.
Let $\widehat{\boldsymbol{x}}_{k+1}$ be the intersection point of $l$ with $S_{\widehat{\underline{w}}_{k+1}}^{\mathrm{cs}}$.
See Figure \ref{f_8_2}.
\begin{figure}[hbtp]
\centering
\scalebox{0.6}{\includegraphics[clip]{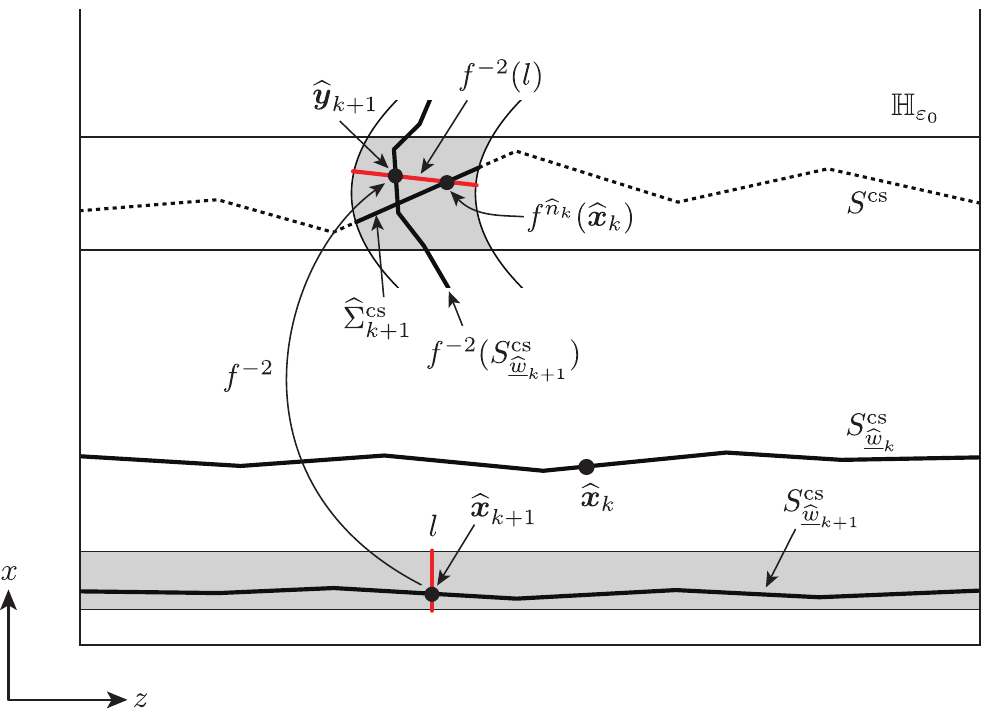}}
\caption{View from the top.
The lower shaded region represents $\mathbb{B}^{\mathrm{u}}(\underline{w}^{(n_0+L(k+1))})$ and the upper does $\mathbb{U}_{k+1}^{\mathrm{cs}}$.}
\label{f_8_2}
\end{figure}
By \eqref{eqn_lC0lam}, $\mathrm{length}(l)=O(\underline{\lambda}_{\mathrm{u}}^{-(n_0+L(k+1))})$.
We set $f^{-2}(\widehat{\boldsymbol{x}}_{k+1})=\widehat{\boldsymbol{y}}_{k+1}$ for short.
Since both $f^{\widehat n_k+2}(\widehat{\boldsymbol{x}}_{k})$ and 
$f^{2}(\widehat{\boldsymbol{y}}_{k+1})$ are contained in $l$, 
by applying the mean value theorem to $f^{-2}|_{\mathbb{B}^{\mathrm{u}}(\underline{w}^{(n_0+L(k+1))})}$ 
we have
\begin{equation}\label{eqn_fxfxD}
\begin{split}
\|f^{\widehat n_k}(\widehat{\boldsymbol{x}}_k)- \widehat{\boldsymbol{y}}_{k+1}\|
&\leq \| Df^{-2}\|\,\|f^{\widehat n_k+2}(\widehat{\boldsymbol{x}}_k)- \widehat{\boldsymbol{x}}_{k+1}\|\\
&\leq \| Df^{-2}\|\,\mathrm{length}(l)=
O(\underline{\lambda}_{\mathrm{u}}^{-(n_0+L(k+1))}).
\end{split}
\end{equation}
This completes the proof.
\end{proof}

By Lemma \ref{l_xkinS}, 
the vector $\boldsymbol{u}_k=\widehat{\boldsymbol{y}}_{k+1}-f^{\widehat{n}_k}(\widehat{\boldsymbol{x}}_k)$ 
 satisfies 
\begin{equation}\label{eqn_|u_k|}
\|\boldsymbol{u}_k\|=O(\underline{\lambda}_{\mathrm{u}}^{-(n_0+L(k+1))}).
\end{equation}
We denote by $F^{\mathrm{cs}}(\boldsymbol{x})$\index{Fcsx@$F^{\mathrm{cs}}(\boldsymbol{x})$} the leaf of $\mathcal{F}_f^{\mathrm{cs}}$ containing $\boldsymbol{x}\in \mathbb{H}_{\varepsilon_0}$.
Let $A_k$ be an orthogonal matrix  of order 3 with determinant $+1$ and 
$$A_k(T_{f^{\widehat{n}_k}(\widehat{\boldsymbol{x}}_k)}F^{\mathrm{cs}}(f^{\widehat{n}_k}(\widehat{\boldsymbol{x}}_k)))=T_{\widehat{\boldsymbol{y}}_{k+1}}F^{\mathrm{cs}}(\widehat{\boldsymbol{y}}_{k+1}).$$
Since the segment $l$ in the proof of Lemma \ref{l_xkinS} is parallel to the $x$-axis, 
$l$ is adaptable to $\boldsymbol{C}_\varepsilon^{\mathrm{u}}$.
So we may apply Proposition \ref{c_Df2v} to $\widehat{\boldsymbol{x}}_{k+1}$ and 
$f^{\widehat n_k+2}(\boldsymbol{x}_k)$.
Hence, by \eqref{eqn_|u_k|},  
one can choose $A_k$ so that 
\begin{equation}\label{eqn_Ak-E}
\|A_k-E\|_{C^r}=\|A_k-E\|_{C^0}=O(\underline{\lambda}_{\mathrm{u}}^{-(n_0+L(k+1))}),
\end{equation}
where $E$ is the unit matrix of order 3.
Here the former equality holds due to the linearity of $A_k$.
Let $\alpha_k:\mathbb{R}^3\longrightarrow \mathbb{R}^3$ be the isometry defined by
$$\alpha_k(\boldsymbol{x})=A_k(\boldsymbol{x}-f^{\widehat{n}_k}(\widehat{\boldsymbol{x}}_k))+f^{\widehat{n}_k}(\widehat{\boldsymbol{x}}_k)+\boldsymbol{u}_k.$$
Then $\alpha_k(l(f^{\widehat{n}_k}(\widehat{\boldsymbol{x}}_k)))+\boldsymbol{u}_k$ is a $C^r$-arc tangent to $F^{\mathrm{cs}}(\widehat{\boldsymbol{y}}_{k+1})$ at 
$\widehat{\boldsymbol{y}}_{k+1}$.
See Figure \ref{f_8_3}, 
where $\omega_k$ denotes the angle between $l(f^{\widehat{n_k}}(\widehat{\boldsymbol{x}}_k))+\boldsymbol{u}_k$ and $T_{\widehat{\boldsymbol{y}}_{k+1}}F^{\mathrm{cs}}(\widehat{\boldsymbol{y}}_{k+1})$ at $\widehat{\boldsymbol{y}}_{k+1}$.
\begin{figure}[hbtp]
\centering
\scalebox{0.6}{\includegraphics[clip]{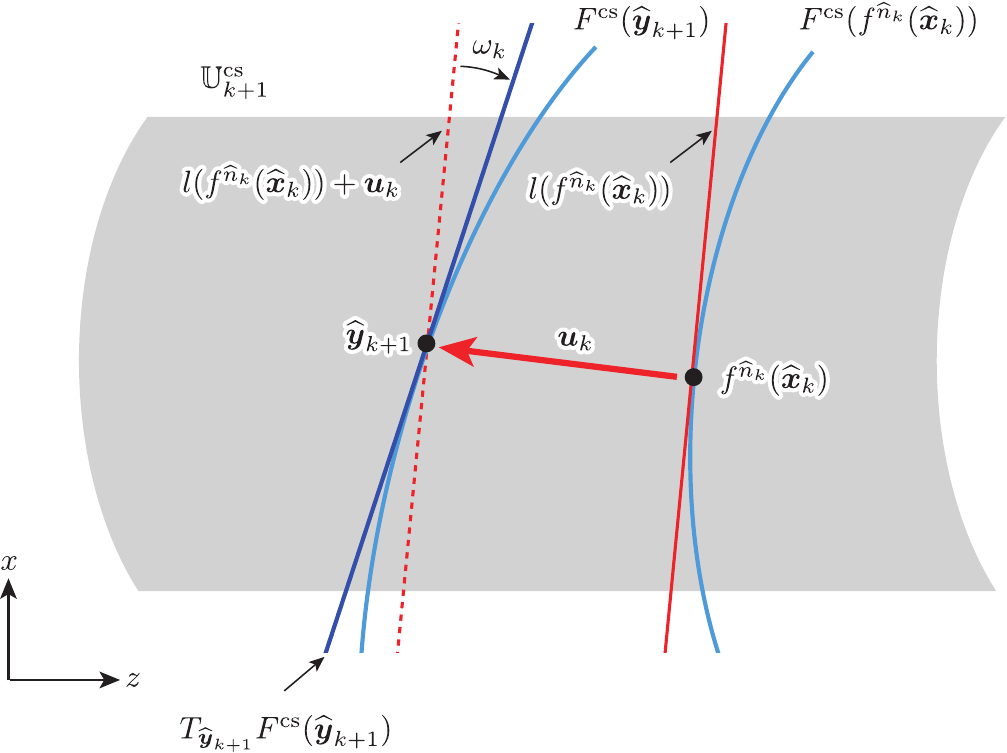}}
\caption{View from the top.}
\label{f_8_3}
\end{figure}
Our situation is similar to that in \cite[Section 7]{KS17}.
Compare the figure here with Figure 7.7 in \cite{KS17}.
Since any isometry on $\mathbb{R}^3$ preserving curvature, 
the tangency of $\alpha_k(l(f^{\widehat{n}_k}(\widehat{\boldsymbol{x}}_k)))$ and $F^{\mathrm{cs}}(\widehat{\boldsymbol{y}}_{k+1})$ at $\widehat{\boldsymbol{y}}_{k+1}$ is quadratic.
Since $\alpha_k(\boldsymbol{x})-\boldsymbol{x}=(A_k-E)(\boldsymbol{x}-f^{\widehat{n}_k}(\widehat{\boldsymbol{x}}_k))+\boldsymbol{u}_k$ on the compact set $\mathbb{B}$, 
we have
\begin{equation}\label{eqn_xiId}
\|(\alpha_k-\mathrm{Id}_{\mathbb{R}^3})|_{\mathbb{B}}\|_{C^r}=O(\underline{\lambda}_{\mathrm{u}}^{-(n_0+L(k+1))}).
\end{equation}

\begin{rmk}\label{r_Df2v}
Suppose that $g$ is a 2-dimensional $C^3$-diffeomorphism with a basic set $\Lambda$ and $\mathcal{F}_g^{\mathrm{s}}$ is 
a stable foliation of $g$ compatible with a locally stable manifold of $\Lambda$.
Then leaves of $\mathcal{F}_g^{\mathrm{s}}$ vary $C^1$ with respect to any transverse direction, for example 
see \cite[Appendix 1, Theorem 8]{PT93} or \cite[Lemma 4.1]{KKY92}.
In \cite{KS17}, the fact is used to get the situation corresponding to our \eqref{eqn_xiId}.
On the other hand, in the 3-dimensional case, we may not expect such a good property of stable foliations.
So we used Proposition \ref{c_Df2v} instead of it.
\end{rmk}

Let $\mathbb{G}^{\mathrm{u}}(\underline{\widetilde w}_k^{(n_0+k)})$ be the closure of the component of 
$\mathbb{B}\setminus (\widetilde{\mathbb{B}}_{k+1}^{\mathrm{u}}\cup \mathbb{B}_{k+1}^{\mathrm{u}})$ such that 
$\mathbb{B}\setminus \mathbb{G}^{\mathrm{u}}(\underline{\widetilde w}_k^{(n_0+k)})$ consists of two components 
and let $\mathbb{C}_k^{\mathrm{cs}}=f^{-2}(\mathbb{G}^{\mathrm{u}}(\underline{\widetilde w}_k^{(n_0+k)}))\cap \mathbb{H}_{\varepsilon_0}$.
By \eqref{abb1}, $\bar\lambda_{\mathrm{u}}=\underline{\lambda}_{\mathrm{u}}+2\varepsilon$ and hence 
$\bar\lambda_{\mathrm{u}}<\underline{\lambda}_{\mathrm{u}}^2$ for any sufficiently small $\varepsilon>0$.
Applying \eqref{eqn_piB} to $\mathbb{C}_k^{\mathrm{cs}}$ instead of $\mathbb{U}_k^{\mathrm{cs}}$, we have
$$
\mathrm{width}(\mathbb{C}_k^{\mathrm{cs}})> C\bar{\lambda}_{\mathrm{u}}^{-(n_0+k-1)}
>C\underline{\lambda}_{\mathrm{u}}^{-2(n_0+k)}
$$
for some constant $C>0$ independent of $k$.
Since $f^{\widehat n_k}(\widehat{\boldsymbol{x}}_k)\in \mathbb{U}_{k+1}^{\mathrm{cs}}\subset f^{-2}(\mathbb{B}_{k+1}^{\mathrm{u}})$ and 
$f^{\widehat n_{k+1}}(\widehat{\boldsymbol{x}}_{k+1})\in \mathbb{U}_{k+2}^{\mathrm{cs}}\subset f^{-2}(\mathbb{B}_{k+2}^{\mathrm{u}})\subset f^{-2}(\widetilde{\mathbb{B}}_{k+1}^{\mathrm{u}})$, 
the segment $\alpha$ in $\mathbb{H}_{\varepsilon_0}$ connecting $f^{\widehat n_k}(\widehat{\boldsymbol{x}}_k)$ with $f^{\widehat n_{k+1}}(\widehat{\boldsymbol{x}}_{k+1})$ 
goes across $\mathbb{C}_k^{\mathrm{cs}}$, see Figure \ref{f_8_4} and also Figure \ref{f_3_3} 
for the placements of $\mathbb{B}_{k+1}^{\mathrm{u}}$, $\widetilde{\mathbb{B}}_{k+1}^{\mathrm{u}}$ and 
$\mathbb{B}_{k+2}^{\mathrm{u}}$.
\begin{figure}[hbtp]
\centering
\scalebox{0.6}{\includegraphics[clip]{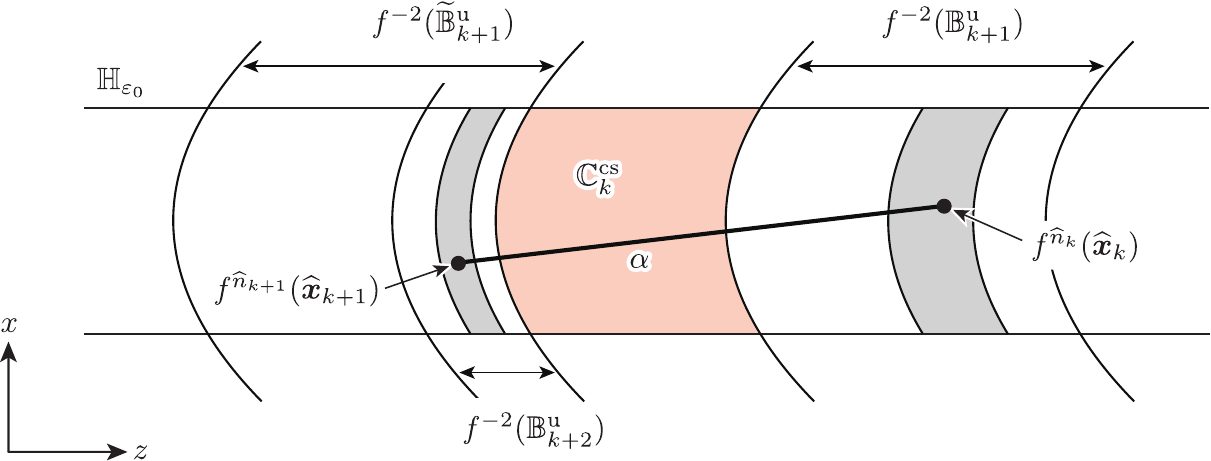}}
\caption{View from the top.
The left-side shaded region represents $\mathbb{U}_{k+2}^{\mathrm{cs}}$ and the right-side does $\mathbb{U}_{k+1}^{\mathrm{cs}}$.}
\label{f_8_4}
\end{figure}
This shows that 
\begin{equation}\label{eqn_wBwB}
\|f^{\widehat n_k}(\widehat{\boldsymbol{x}}_k)-f^{\widehat n_{k+1}}(\widehat{\boldsymbol{x}}_{k+1})\|>C\underline{\lambda}_{\mathrm{u}}^{-2(n_0+k)}.
\end{equation}
Recall that $L$ is the integer given in Section \ref{S_BWS} with $L\geq 4$.
By \eqref{eqn_|u_k|}, there exists a positive integer $k_0$ such that, for any $k\geq k_0$, 
\begin{equation}\label{eqn_root3}
C\underline{\lambda}_{\mathrm{u}}^{-2(n_0+k)}\geq 2\sqrt3 \underline{\lambda}_{\mathrm{u}}^{-(n_0+3k)}>\underline{\lambda}_{\mathrm{u}}^{-(n_0+3k)}\geq 
3\|\boldsymbol{u}_k\|.
\end{equation}
Here `$\sqrt3$' means that the radius of the circumscribed sphere of a cube of edge length $2d$ is $\sqrt3 d$.
We set $f^{\widehat{n}_k}(\widehat{\boldsymbol{x}}_k)=(\widehat x_k,\widehat y_k,\widehat z_k)$, $d_k=\underline{\lambda}_{\mathrm{u}}^{-(n_0+3k)}$
and consider the cube in $\mathbb{B}$ of edge length $2d_k$ defined as
$$\mathbb{D}_k\index{Dk@$\mathbb{D}_k$}=\left[\widehat x_k-d_k,\widehat x_k+d_k\right]\times \left[\widehat y_k-d_k,\widehat y_k+d_k\right]\times \left[\widehat z_k-d_k,\widehat z_k+d_k\right].$$
By \eqref{eqn_wBwB} and 
\eqref{eqn_root3}, we know that $\mathbb{D}_k$ $(k\geq k_0)$ are mutually disjoint.

\subsection{Bump functions for perturbations}\label{ss_bump}

We here prepare bump functions for our perturbations.
Let $\beta$ be  
a non-negative, non-decreasing $C^{r}$ function on $\mathbb{R}$
such that $\beta (x)=0$ if $x\leq -1$ while $\beta(x)=1$ if $x\geq 0$. 
Using it, we define the bump function as   
$$
\beta _{c, J}(x)=\beta \biggl(\frac{x-a}{c|J|}\biggr)+\beta \biggl(-\frac{x-a'}{c|J|}\biggr)-1,
$$
where $c$ is a positive constant and $J=[a, a']$ for any $a, a'\in \mathbb{R}$ with $a<a'$.
Note that $\beta_{c,J}$ is a non-negative function with $\beta_{c,J}(x)=1$ if $x\in J$ and 
the support of which is contained in the $c|J|$-neighborhood of $J$ in $\mathbb{R}$.
From the definition of $\beta_{c,J}$,  
\begin{equation}\label{eqn_|bump|}
\left\|
\beta_{c, J}
\right\|_{C^{r}}\leq 
\frac{1}{(c|J|)^{r}} 
\|
\beta 
\|_{C^{r}}
\end{equation}
if $c|J|\leq 1$.

Let $\boldsymbol{\beta}_k:\mathbb{B}\longrightarrow\mathbb{R}$ be the map defined as
\begin{align*}
\boldsymbol{\beta}_k(\boldsymbol{x})=\beta_{1/2,[\widehat x_k-d_k/2,\widehat x_k+d_k/2]}(x)\,&\beta_{1/2,[\widehat y_k-d_k/2,\widehat y_k+d_k/2]}(y)\,\\
&\hspace{40pt}\times \beta_{1/2,[\widehat z_k-d_k/2,\widehat z_k+d_k/2]}(z)
\end{align*}
for $\boldsymbol{x}=(x,y,z)$, which is the bump function supported on the cube $\mathbb{D}_k$ given in the previous subsection.
Since $d_k=\underline{\lambda}_{\mathrm{u}}^{-(n_0+3k)}$, we have 
by \eqref{eqn_|bump|}   
\begin{equation}\label{eqn_bold_beta}
\|\boldsymbol{\beta}_k\|_{C^r}\leq O\biggl(\biggl(\frac2{d_k}\biggr)^r\biggr)=O(\underline{\lambda}_{\mathrm{u}}^{3rk}).
\end{equation}
For any integers $n,a$ with $1\leq n<a$,
we define the sequence of $C^r$-perturbation maps $\psi_{n,a}:M\longrightarrow M$ supported on the disjoint union 
$\bigcup_{k=n}^a\mathbb{D}_k\subset \mathbb{B}$ by
$$\psi_{n,a}(\boldsymbol{x})=\boldsymbol{x}+\sum_{k=n}^a\boldsymbol{\beta}_k(\boldsymbol{x})(\alpha_k(\boldsymbol{x})-\boldsymbol{x})$$
for $\boldsymbol{x}\in \bigcup_{k=1}^a\mathbb{D}_k$.
By \eqref{eqn_|u_k|} and \eqref{eqn_bold_beta},   
\begin{equation}\label{eqn_|beta_u|}
\|\boldsymbol{\beta}_k\|_{C^r}\|\boldsymbol{u}_k\|\leq O(\underline{\lambda}_{\mathrm{u}}^{(9r-L)k}).
\end{equation}

\begin{lem}\label{l_psi_n}
The sequence $\{\psi_{n,a}\}_{a=1}^\infty$ $C^r$-converges as $a\rightarrow \infty$ to the $C^r$-map 
$\psi_n:M\longrightarrow M$ with 
$$\psi_n(\boldsymbol{x})=\boldsymbol{x}+\sum_{k=n}^\infty\boldsymbol{\beta}_k(\boldsymbol{x})(\alpha_k(\boldsymbol{x})-\boldsymbol{x})$$
for $\boldsymbol{x}\in \bigcup_{k=n}^\infty \mathbb{D}_k$ 
if $L>3r$ and $n\geq k_0$.
Moreover $\psi_n$ are $C^r$-diffeomorphisms on $M$ for all sufficiently large $n$ which $C^r$-converges to 
the identity as $n\rightarrow \infty$.
\end{lem}
\begin{proof}
By \eqref{eqn_xiId} and \eqref{eqn_|beta_u|},
\begin{equation}\label{eqn_|psi_psi|}
\|\psi_{n,a}-\psi_{n,b}\|_{C^r}\leq 
O\left(\sum_{k=a+1}^\infty\underline{\lambda}_{\mathrm{u}}^{(3r-L)k}\right)
=O\left((1-\underline{\lambda}_{\mathrm{u}}^{3r-L})^{-1}\underline{\lambda}_{\mathrm{u}}^{(3r-L)(a+1)}\right)
\end{equation}
for any integers $a,b$ with $n\leq a<b$.
This shows that $\{\psi_{n,a}\}_{a=n}^\infty$ is a Cauchy sequence in the space $(\mathrm{Map}^r(M),\|\cdot\|_{C^r})$ 
of $C^r$-maps on $M$, which is a complete metric space.
Thus $\psi_{n,a}$ $C^r$-converges to the $C^r$-mas $\psi_n$ as $a\rightarrow \infty$.
Furthermore, by \eqref{eqn_|psi_psi|}, we know that $\psi_{n}$ $C^r$-converges to the identity as $n\rightarrow\infty$.
Since the identity is a diffeomorphism on $M$, $\psi_n$ is also a diffeomorphism for all sufficiently large $n$.
\end{proof}

This lemma shows that the composition 
\begin{equation}\label{eqn_gfpsi}
g=f\circ \psi_n:M\longrightarrow M
\end{equation}
is a $C^r$-diffeomorphism arbitrarily $C^r$-close to $f$ and hence contained 
in $\mathcal{O}(f_0)$ if $n$ is sufficiently large. 
From the definition of $g$, we know that $\mathcal{F}_g^{\mathrm{s}}=\mathcal{F}_f^{\mathrm{s}}$ and $g^{{\widehat n}_k+2}(\widehat{ \boldsymbol{x}}_k)=\widehat{\boldsymbol{x}}_{k+1}$ if $k\geq n$.
In particular, $(\widehat{\boldsymbol{x}}_k)_{k\geq n}$ is a subsequence of the $g$-orbit $\mathrm{Orb}_g(\widehat{\boldsymbol{x}}_n)$ emanating from $\widehat{\boldsymbol{x}}_n$.

\section{Construction of contracting wandering domains}\label{S_CWD}

\subsection{Quadratic and majority conditions}\label{ss_quad_majority} 
Let $\underline{\widehat w}_{k}=\underline{w}^{(n_0+Lk)}\underline{u}_{k}\underline{\iota}_k\underline{\gamma}^{(m_k)}$ 
be the binary code presented in Lemma \ref{lem-3-1}.
Recall that the length $\widehat{n}_k$ of $\underline{\widehat w}_{k}$ is given by \eqref{eqn_hatn0} and 
$$
|\underline{w}^{(n_0+Lk)}|=n_{0}+Lk=O(k),\ 
|\underline{\iota}_k\underline{\gamma}^{(m_k)}|=|\underline{\iota}_k|+m_{k}=O(k).
$$
As described in Lemma \ref{lem-3-1}-(1), 
the sub-code $\underline{u}_{k}$ of 
$\widehat{\underline{w}}_k$ can be chosen freely.
So we may assume the extra condition, called the \emph{quadratic condition}, that 
the length of $\underline{u}_{k}$ is just  
\begin{equation}\label{eqn_k2}
 |\underline{u}_{k}|=k^{2}.
\end{equation}
Then 
\begin{equation}\label{eqn_|nkhat|}
\widehat n_k=|\widehat{\underline{w}}_k|=n_0+Lk+k^2+|\underline{\iota}_k|+m_k=k^2+O(k).
\end{equation}
This implies that $\widehat n_k$ 
increases subexponentially as $k\rightarrow \infty$.
More precisely, we have the following lemma.

\begin{lem}\label{subexp}
For any $\eta>0$, there is an integer $k_{1}\geq k_0$ such that, for any integer $k\geq k_{0}$, 
\[
\widehat n_{k}<\widehat n_{k+1}<(1+\eta) \widehat n_{k}. 
\]
\end{lem}
\begin{proof}
From the definition, $\widehat n_{k}<\widehat n_{k+1}$. 
Moreover, it follows from \eqref{eqn_|nkhat|} that  
\[
\frac{\widehat n_{k+1}}{\widehat n_{k}}
=\frac{(k+1)^{2}+O(k+1)}{k^{2}+O(k)}\rightarrow 1\ \text{as $k\rightarrow+\infty$},
\]
and the claim is correct.
\end{proof}

Suppose that $\underline{v}=(v_j)_{j\in\mathbb{Z}}$ is any element of $\{0,1\}^{\mathbb{Z}}$ 
with the majority condition in Definition \ref{m-cnd}, that is, $\liminf\limits_{n\rightarrow \infty}p_n(\underline{v})\geq 
\dfrac12$ holds for the sequence $p_n(\underline{v})$ of \eqref{eqn_pnv}.
So, for any $\eta>0$, there exists an integer $n_*\in\mathbb{N}$ with
\begin{equation}\label{eqn_pn}
p_n(\underline v)>\dfrac12-\frac{\eta}2
\end{equation}
if $n\geq n_*$.

We set $\beta_k=k^2$ for $k\in \mathbb{N}$ and will determine a sequence $(\alpha_k)_{k\geq 1}$ inductively.
Let $\alpha_1=n_0+L$ and suppose that
$\alpha_j$ for $j=1,\dots,k-1$ is already determined.
Fix the free code $\underline u_k$ as 
\begin{equation}\label{eqn_ukvalpha}
\underline u_k=(v_{\alpha_k+1}v_{\alpha_k+2}\dots v_{\alpha_k+\beta_k}),
\end{equation}
which determines $|\underline{\iota}_k|$ by Lemma \ref{lem-3-1} and hence $\widehat n_k$ by \eqref{eqn_|nkhat|}.
Then one can define $\alpha_{k+1}$ as 
\begin{equation}\label{eqn_alpha(k+1)}
\alpha_{k+1}=\sum_{i=1}^k(\widehat n_i+2)+n_0+L(k+1).
\end{equation}
Since 
$$
\alpha_{k+1}-(\alpha_k+\beta_k)=|\underline{\iota}_k|+m_k+2+n_0+L(k+1)=O(k),
$$
the sequences $(\alpha_k)_{k\in\mathbb{N}}$ and $(\beta_k)_{k\in\mathbb{N}}$ 
satisfy (DEI) in Definition \ref{describable}.
See Figure \ref{f_9_1}.
\begin{figure}[hbtp]
\centering
\scalebox{0.6}{\includegraphics[clip]{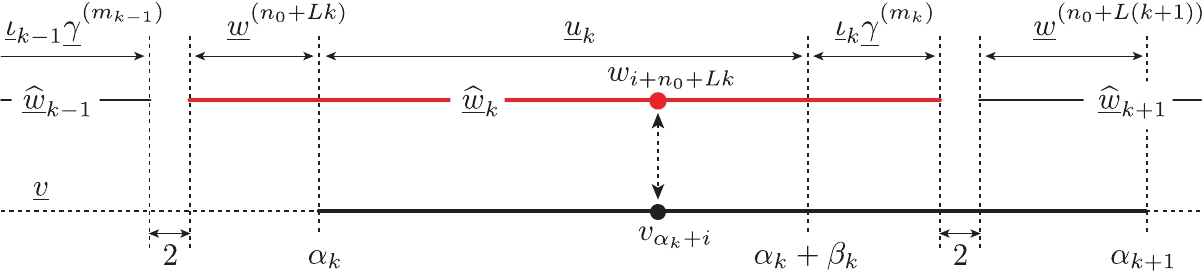}}
\caption{}
\label{f_9_1}
\end{figure}
By \eqref{eqn_|nkhat|} and \eqref{eqn_alpha(k+1)}, 
$$\alpha_{k+1}=\sum_{i=1}^k(i^2+O(i))+O(k+1)=\frac{k(k+1)(2k+1)}6+O(k^2)=\frac13 k^3+O(k^2).$$
If we write $\alpha_{k+1}=n$, then 
\begin{equation}\label{eqn_3n}
(3n)^{2/3}=(k^3+O(k^2))^{2/3}=k^2\left(1+O(k^{-1})\right)^{2/3}.
\end{equation}
This implies that 
\begin{equation}\label{eqn_alpha_alpha}
\begin{split}
\left|(n-(3n)^{2/3})-\alpha_k\right|&=
\left|(\alpha_{k+1}-\alpha_k)-(3n)^{2/3}\right|\\
&=
k^2\left|1+O(k^{-1})-\left(1+O(k^{-1})\right)^{2/3}\right|=O(k).
\end{split}
\end{equation}
We denote the total numbers of $0$ and $1$ entries in the code $\underline{\widehat w}_{k}$ by 
$\widehat n_{k(0)}$ and $\widehat n_{k(1)}$, respectively.
If $\widehat{\underline{w}}_k=(w_1w_2\dots w_{\widehat n_k})$, then 
by \eqref{eqn_ukvalpha}
\begin{equation}\label{eqn_wiva}
w_{i+n_0+Lk}=v_{\alpha_k+i}\quad(i=1,\dots,k^2).
\end{equation}
Since 
\begin{align*}
\#\bigl\{i\,;\,1\leq i\leq \widehat n_k,w_i=0\bigr\}&=\#\bigl\{i\,;\,n_0+Lk+1\leq i\leq n_0+Lk+k^2, 
w_i=0\bigr\}+O(k),\\
\#\bigl\{j\,;\,\alpha_k+1\leq j\leq \alpha_{k+1},v_j=0\bigr\}&=
\#\bigl\{j\,;\,\alpha_k+1\leq j\leq \alpha_{k}+\beta_k,v_j=0\bigr\}+O(k),
\end{align*} 
we have 
$$\left|\,\#\bigl\{i\,;\,1\leq i\leq \widehat n_k,w_i=0\bigr\}-\#\bigl\{j\,;\,\alpha_k+1\leq j\leq \alpha_{k+1},v_j=0\bigr\}\right|=O(k).$$
It follows from this fact together with \eqref{eqn_pn} and \eqref{eqn_alpha_alpha} that  
\begin{align*}
\frac{\widehat n_{k(0)}}{\widehat n_k}&=
\frac{\#\bigl\{i\,;\,1\leq i\leq \widehat n_k,w_i=0\bigr\}}{\widehat n_k}\\
&=
\frac{\#\bigl\{j\,;\,\alpha_k+1\leq j\leq \alpha_{k+1},v_j=0\bigr\}+O(k)}{k^2+O(k)}\\
&=\frac{\#\{j\,;\, n-(3n)^{2/3}< j\leq n, v_j=0\}+O(k)}{k^2+O(k)}\\
&=p_n(\underline v)\frac{(3n)^{2/3}}{k^2+O(k)}+\frac{O(k)}{k^2+O(k)}\\
&=p_n(\underline v)\frac{\left(1+O(k^{-1})\right)^{2/3}}{1+O(k^{-1})}+O(k^{-1})
>\frac12-\eta
\end{align*}
for any sufficiently large $k$.
Since 
$\widehat n_{k}=\widehat n_{k(0)}+\widehat n_{k(1)}$, 
the preceding inequality implies  
\begin{equation}\label{eqn_majority_nk}
\widehat n_{k(1)}<(1+\eta_0)\widehat n_{k(0)}, 
\end{equation}
where $\eta_0=\dfrac{4\eta}{1-\eta}$.

\begin{rmk}\label{r_majority}
The inequality \eqref{eqn_majority_nk} is a sort of majority condition on $\widehat{\underline{w}}_k$,  
which corresponds to the original majority condition $\widehat n_{k(1)}\leq \widehat n_{k(0)}$ in \cite[(4.1b)]{KNS23}.
The inequality \eqref{eqn_majority_nk} is indispensable to show Lemma \ref{lem7.3}, which 
is a key to Theorem \ref{thmWD}. 
To define the condition in \cite{KNS23}, we need some constants associated with $f_0$, 
which are used to determine $\widehat n_k$.
On the other hand, the majority condition in Definition \ref{m-cnd} requires only data of the binary code 
$\underline v$ and independent of the choice of $f_0$.
\end{rmk}

\subsection{Settings for wandering domains}\label{ss_setting_WD}
Suppose that the binary code $\widehat{\underline{w}}_k$ satisfies the conditions given in the previous subsection.
For each integer $k\geq k_1$, we introduce the following notations:
\begin{equation}\label{eqn_xi_zeta}
\xi_{k}\index{xik@$\xi_k$}=\xi_{k,\sigma}=\sigma\biggl(\bar{\lambda}_{\rm u}^{\sum_{i=0}^{\infty} \tfrac{\widehat n_{k+i}}{2^{i}}} \biggr)^{-1}
\quad\text{and}\quad
\rho_{k}\index{rhok@$\rho_k$}=\rho_{k,\sigma}=\sigma^{-1}\xi_k^{\frac12},
\end{equation}
where $\sigma$ is a positive constant independent of $k$ and will be fixed 
in the proof of Theorem \ref{thmWD}.
Then we have 
\begin{equation}\label{eqn_xi_al_xi}
\xi_{k+1}=\sigma^{-1}\bar\lambda_{\mathrm{u}}^{2\widehat n_k}\xi_k^2.
\end{equation}

\begin{lem}\label{lem7.3}
$$
\bar\lambda_{\rm cs0}^{\widehat n_{k(0)}}\bar\lambda_{\rm cs1}^{\widehat n_{k(1)}}\rho_{k}=o(\xi_{k+1}).
$$
\end{lem}
\begin{proof}
By Lemma \ref{subexp},  for any $\eta$ with $0<\eta<1$, there exists a positive integer $k_1$ such that 
$\widehat n_{k+i}<(1+\eta)^i\widehat n_{k}$ if $k\geq k_1$ and $i\geq 0$.
Then we have 
$$
\frac{3}{2}\sum_{i=0}^{\infty}\frac{\widehat n_{k+i}}{2^{i}}
\leq 
\frac{3\widehat n_{k}}{2}\sum_{i=0}^{\infty}
\biggl(\frac{1+\eta}{2}\biggr)^{i}
=\frac{3\widehat n_{k}}{1-\eta}=(3+\eta_{1})\widehat n_{k},
$$
where $\eta_{1}=3\eta/(1-\eta)$.
Then, by \eqref{eqn_xi_zeta} and \eqref{eqn_xi_al_xi}, 
\begin{equation}\label{eqn_lambda_lambda}
\begin{split}
\frac{ \bar\lambda_{\rm cs0}^{\widehat n_{k(0)}}\bar\lambda_{\rm cs1}^{\widehat n_{k(1)}}\rho_{k}}
{\xi_{k+1}}
&=
\sigma^{-\frac{3}{2}} \bar\lambda_{\rm cs0}^{\widehat n_{k(0)}}\bar\lambda_{\rm cs1}^{\widehat n_{k(1)}} \bar{\lambda}_{\rm u}^{-2 \widehat n_{k}}
\biggl(\bar{\lambda}_{\rm u}^
{\sum_{i=0}^{\infty}\tfrac{\widehat n_{k+i}}{2^{i}}}
\biggr)^{\frac{3}{2}}\\
&\leq \sigma^{-\frac{3}{2}} \bar\lambda_{\rm cs0}^{\widehat n_{k(0)}}\bar\lambda_{\rm cs1}^{\widehat n_{k(1)}} \bar{\lambda}_{\rm u}^{(1+\eta_1)\widehat n_{k}}.
\end{split}
\end{equation}
Since 
$\bar\lambda_{\rm cs0}\bar\lambda_{\rm cs1}\bar\lambda_{\mathrm{u}}^{2}<1$ 
by \eqref{pdc2}, we have 
$$\bar\lambda_{\rm cs0}\bar\lambda_{\rm cs1}^{(1+\eta_0)}\bar\lambda_{\mathrm{u}}^{(2+\eta_0)(1+\eta_1)}<1$$
if $\eta>0$ is sufficiently small.
On the other hand, since $\bar\lambda_{\rm cs1}\bar\lambda_{\mathrm{u}}>1$ by \eqref{eqn_eigen_v2}, 
the majority condition \eqref{eqn_majority_nk} implies  
$$(\bar\lambda_{\rm cs1}\bar\lambda_{\mathrm{u}}^{(1+\eta_{1})})^{\widehat n_{k(1)}}
< (\bar\lambda_{\rm cs1}\bar\lambda_{\mathrm{u}}^{(1+\eta_{1})})^{(1+\eta_0)\widehat n_{k(0)}}.
$$
Then, by \eqref{eqn_lambda_lambda},   
\begin{align*}
\frac{\bar\lambda_{\rm cs0}^{\widehat n_{k(0)}}\bar\lambda_{\rm cs1}^{\widehat n_{k(1)}}\rho_{k}}{\xi_{k+1}}
&
\leq 
\sigma^{-\frac32} 
\bar\lambda_{\rm cs0}^{\widehat n_{k(0)}}
\bar\lambda_{\rm cs1}^{\widehat n_{k(1)}}
\bar\lambda_{\mathrm{u}}^{(1+\eta_{1})(\widehat n_{k(0)}+\widehat n_{k(1)})}\\
&= 
\sigma^{-\frac32}
(\bar\lambda_{\rm cs0}\bar\lambda_{\mathrm{u}}^{(1+\eta_{1})})^{\widehat n_{k(0)}}(\bar\lambda_{\rm cs1}\bar\lambda_{\mathrm{u}}^{(1+\eta_{1})})^{\widehat n_{k(1)}}\\
&<
\sigma^{-\frac32}
(\bar\lambda_{\rm cs0}\bar\lambda_{\mathrm{u}}^{(1+\eta_{1})})^{\widehat n_{k(0)}}(\bar\lambda_{\rm cs1}\bar\lambda_{\mathrm{u}}^{(1+\eta_{1})})^{(1+\eta_0)\widehat n_{k(0)}}\\
&\leq 
\sigma^{-\frac32}
\bigl(\bar\lambda_{\rm cs0}\bar\lambda_{\rm cs1}^{(1+\eta_0)}\bar\lambda_{\mathrm{u}}^{(2+\eta_0)(1+\eta_{1})}\bigr)^{\widehat n_{k(0)}}\rightarrow 0\quad\text{as $k\rightarrow \infty$}.
\end{align*}
This completes the proof.
\end{proof}

Let $(\widehat{\boldsymbol{x}}_k)_{k\geq 1}$ be the sequence given in Lemma \ref{l_xkinS} and $l(\widehat{\boldsymbol{x}}_k)$ the leaf of $\mathcal{L}_{(\widehat n_k;\infty)}$ containing $\widehat{\boldsymbol{x}}_k$.
Then $l(\widehat{\boldsymbol{x}}_k)$ is an arc in $\mathbb{H}_{\widehat{\underline{w}}_k}$ divided by $\widehat{\boldsymbol{x}}_k$ into sub-arcs $l_0(\widehat{\boldsymbol{x}}_k)$, $l_1(\widehat{\boldsymbol{x}}_k)$.
Since $\widehat{\boldsymbol{x}}_k$ is contained in $S_{\widehat{\underline{w}}_k}^{\mathrm{cs}}$, by 
\eqref{eqn_lambda_max}
there exists a constant $C_0>0$ independent of $k$ such that 
$$\delta_k= \min\bigl\{\mathrm{length}(l_0(\widehat{\boldsymbol{x}}_k)), \mathrm{length}(l_1(\widehat{\boldsymbol{x}}_k))\bigr\}
\geq C_0 \bar\lambda_{\mathrm{u}}^{-\widehat n_k}.$$
Since $\widehat n_{k+i}\geq \widehat n_k$ for any $i\geq 0$, \eqref{eqn_xi_zeta} implies  $\xi_{k}<\sigma\biggl(\bar{\lambda}_{\rm u}^{\sum_{i=0}^{\infty} \tfrac{\widehat n_k}{2^{i}}} \biggr)^{-1}=\sigma\bar\lambda_{\mathrm{u}}^{-2\widehat n_k}$.
So one can assume that 
$\delta_k>\xi_k$ for any $k\geq k_1$.
Let $J_k$\index{Jk@$J_k$} be the sub-arc of $l(\widehat{\boldsymbol{x}}_k)$ with $\widehat{\boldsymbol{x}}_k$ as its center and of length $\xi_k$.
Recall that, for any $\boldsymbol{x}\in J_k$, $F^{\mathrm{s}}(\boldsymbol{x})$ is the leaf of $\mathcal{F}_f^{\mathrm{s}}$ containing $\boldsymbol{x}$.
Let $U_{\rho_k}(\boldsymbol{x})$ be the disk in $F^{\mathrm{s}}(\boldsymbol{x})$ centered at $\boldsymbol{x}$ and of radius $\rho_k$.
Then the union $D_k=D_{k,\sigma}=\bigcup_{\boldsymbol{x}\in J_k}U_{\rho_k}(\boldsymbol{x})$\index{Dk@$D_k$} is a subset of $\mathbb{B}$ in shape of a thin solid cylinder.
See Figure \ref{f_9_2}.
\begin{figure}[hbtp]
\centering
\scalebox{0.6}{\includegraphics[clip]{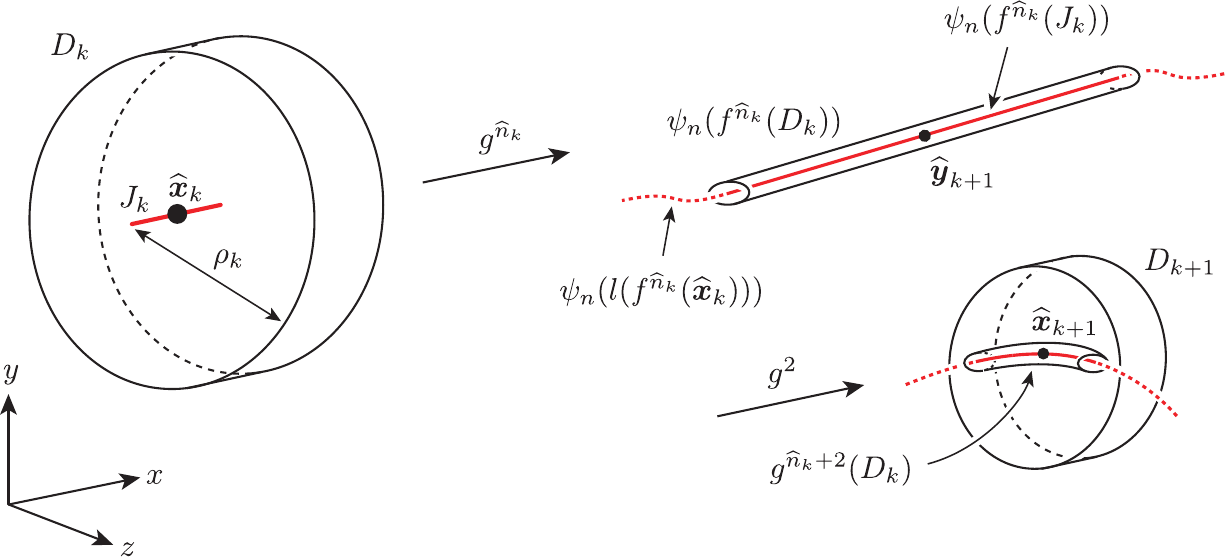}}
\caption{The leaf $l(f^{\widehat{n}_k}(\widehat{\boldsymbol{x}}_k))$ of $\mathcal{L}_{(0;\infty)}$ is slightly distorted by the perturbation $\psi_n$.}
\label{f_9_2}
\end{figure}

The following is the main result of this section.

\begin{thm}\label{thmWD}
Suppose that $g$ is the diffeomorphism of \eqref{eqn_gfpsi}.
Then there exists $\sigma>0$ and an integer $k_2\geq k_1$ such that, 
for every integer $k\geq k_2$, the interior $\mathrm{Int}D_k$ of 
$D_k=D_{k,\sigma}$ is a 
contracting wandering domain for $g$ 
satisfying
\[
g^{\widehat n_{k}+2}(D_{k})\subset \mathrm{Int}D_{k+1}.
\]
\end{thm}
\begin{proof}
By \eqref{eqn_lambda_maxS} and \eqref{eqn_lambda_maxc}, for any leaf $F$ of $\mathcal{F}_f^{\mathrm{s}}|_{\mathbb{V}_{i,g}}$ $(i=0,1)$ and 
any $\boldsymbol{v}\in T_{\boldsymbol{x}}F$ with $\boldsymbol{x}\in F$,
$$\|Df(\boldsymbol{x})\boldsymbol{v}\|\leq \bar\lambda_{\mathrm{cs} i}\|\boldsymbol{v}\|.$$
By Lemma \ref{lem7.3}, $\mathrm{diam}(g^{\widehat n_k}(U_{\rho_k}(\boldsymbol{x})))=o(\xi_{k+1})$ and hence
\begin{equation}\label{eqn_diam_g}
\mathrm{diam}(g^{\widehat n_k+2}(U_{\rho_k}(\boldsymbol{x})))=o(\xi_{k+1}).
\end{equation}
By \eqref{eqn_lambda_max}, 
$\mathrm{length}(g^{\widehat n_k+2}(J_k))<C_1\bar\lambda_{\mathrm{u}}^{\widehat n_k}\xi_k$ for some constant $C_1>0$.
Since $g^{\widehat n_k}(J_k)$ is quadratically tangent to a leaf of $\mathcal{F}_f^{\mathrm{cs}}$ at $g^{\widehat n_k}(\widehat{ \boldsymbol{x}}_k)$, 
$g^{\widehat n_k+2}(J_k)$ is so to $F^{\mathrm{s}}(\widehat{\boldsymbol{x}}_{k+1})$ at $\widehat{\boldsymbol{x}}_{k+1}$.
By this fact 
together with \eqref{eqn_xi_al_xi}, there exists a constant $C_2>0$ independent of $k$ such that
$$|\pi_f^{\mathrm{u}}(g^{\widehat n_k+2}(J_k))|\leq C_1^2C_2\bar\lambda_{\mathrm{u}}^{2\widehat n_k}\xi_k^2=\sigma C_1^2C_2\xi_{k+1}.$$
In fact, Propositions \ref{p_curvature_plane} and \ref{p_kappa_fn} 
in Appendix \ref{Ap_curvature} imply that $C_2$ depends only on the constants $a_1$, $a_4$ given in \eqref{eqn_tang}.
Hence one can choose $\sigma>0$ sufficiently small so that  
\begin{equation}\label{eqn_pi_gJ}
|\pi_f^{\mathrm{u}}(g^{\widehat n_k+2}(J_k))|<\frac{\xi_{k+1}}3
\end{equation}
holds.
It follows from \eqref{eqn_diam_g} and \eqref{eqn_pi_gJ} that
$\pi_f^{\mathrm{u}}(g^{\widehat n_k+2}(D_k))\subset \mathrm{Int}\, \pi_f^{\mathrm{u}}(D_{k+1})$.
Again, by using the fact that $\mathrm{length}(g^{\widehat n_k+2}(J_k))<C_1\bar\lambda_{\mathrm{u}}^{\widehat n_k}\xi_k$, 
we have a constant $C_3>0$ independent of $k$ such that
$\mathrm{diam}(\pi_{yz}(g^{\widehat n_k+2}(J_k)))\leq C_3\bar\lambda_{\mathrm{u}}^{\widehat n_k}\xi_k$.
Since $\rho_{k+1}=\sigma^{-1}\xi_{k+1}^{1/2}=\sigma^{-3/2}\bar\lambda_{\mathrm{u}}^{\widehat n_k}\xi_k$, 
we may assume that $\mathrm{diam}(\pi_{yz}(g^{\widehat n_k+2}(J_k)))<\rho_{k+1}/3$ 
if necessary replacing $\sigma$ by a smaller positive number.
See Figure \ref{f_9_3}.
\begin{figure}[hbtp]
\centering
\scalebox{0.6}{\includegraphics[clip]{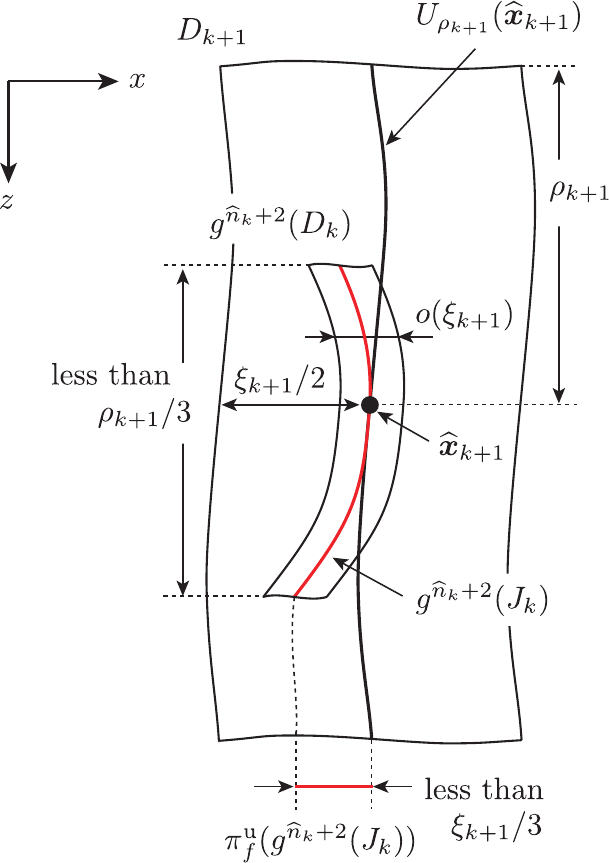}}
\caption{View from the top.}
\label{f_9_3}
\end{figure}
Hence, $\pi_{yz}(g^{\widehat n_k+2}(D_k))$ is contained in $\mathrm{Int}\pi_{yz}(D_{k+1})$.
This implies $g^{\widehat n_k+2}(D_k)\subset \mathrm{Int}D_{k+1}$ and 
completes the proof.
\end{proof}

\section{Proof of Theorem \ref{mainthm}} \label{prThmB}
By using arguments in the previous sections, we will prove Theorem  \ref{mainthm}.

\begin{proof}[Proof of Theorem \ref{mainthm}]
Recall that $\Sigma^{\prime}$ is the subset of $\{0,1\}^{\mathbb{Z}}$ consisting of elements 
with the majority condition. 
We first discuss the 3-dimensional diffeomorphism $g$ defined in \eqref{eqn_gfpsi}.
From Subsection \ref{ss_quad_majority} and Theorem \ref{thmWD}, 
we already have sequences of integer intervals 
satisfying (DEI) and wandering domains  
satisfying (OCD) in Definition \ref{describable}. That is, 
$f$ is  $\Sigma^{\prime}$-describable. 
Thus it follows immediately from Theorem \ref{p-lemma} that  
$f$ is pluripotent for $\Lambda_{f}^{\prime}$, where $\Lambda_{f}^{\prime}=\mathcal{I}^{-1}_{f}(\Sigma^{\prime})$.

Now we consider the case of $\dim M=n\geq 4$.
Then $M$ has a coordinate neighborhood identified with $(-1,2)^n$ the coordinate of which 
is represented as
$$\boldsymbol{x}=(x,y,z,y_1,y_2,\dots,y_{n-3}).$$
We set $\widetilde{\mathbb{B}}=I_{\varepsilon_0}^n$, $\widetilde{\mathbb{V}}_i=\mathbb{V}_i\times 
I_{\varepsilon_0}^{n-3}$ $(i=0,1)$ and $\widetilde{\mathbb{H}}_{\varepsilon_0}=\mathbb{H}_{\varepsilon_0}\times 
I_{\varepsilon_0}^{n-3}$.
Let $f_0:(-1,2)^3\longrightarrow (-1,2)^3$ is a $C^r$-diffeomorphism satisfying 
\eqref{eqn_f0V} and \eqref{eqn_tang}.
We define a $C^r$-diffeomorphism $\widetilde f_0:M\longrightarrow M$ extending $f_0|_{\mathbb{B}}$ and 
satisfying the following conditions.
\begin{align*}
&(1)\quad \widetilde f_0(\boldsymbol{x})=\bigl(f_0(x,y,z),\lambda_{\mathrm{ss}}y_1,\lambda_{\mathrm{ss}}y_2,\dots,
\lambda_{\mathrm{ss}}y_{n-3}\bigr)\quad \text{for $\boldsymbol{x}\in \widetilde{\mathbb{V}}_0\cup \widetilde{\mathbb{V}}_1$}.\\
&(2)\quad \widetilde f_0^2(\boldsymbol{x})=\bigl(f_0^2(x,y,z),2^{-1}y_1,2^{-1}y_2,\dots,2^{-1}y_{n-3}\bigr)\quad \text{for $\boldsymbol{x}\in \widetilde{\mathbb{H}}_{\varepsilon_0}$}.
\end{align*}
For any element $\widetilde{f}$ of $\mathrm{Diff}^r(M)$ contained in a sufficiently small 
neighborhood of $\widetilde{f}_0$, there exist a stable foliation $\mathcal{F}_{\widetilde f}^{\mathrm{s}}$ on 
$\widetilde{\mathbb{B}}$ satisfying the conditions corresponding to \eqref{F1}--\eqref{F3} in 
Subsection \ref{ss_cone_foliation} and an $\widetilde f$-invariant 1-dimensional foliation  $\widetilde{\mathcal{L}}_{(k;\infty)}$ on $\widetilde{\mathbb{H}}_{\,[k]}
=\widetilde f^{-k}(\widetilde{\mathbb{H}}_{\varepsilon_0})\cap \widetilde{\mathbb B}$
defined as in Subsection \ref{ss_1-dim}.
Then we have a leaf $\widetilde J_k$ of $\widetilde{\mathcal{L}}_{(k;\infty)}$ corresponding to $J_k$ in 
Subsection \ref{ss_setting_WD} and the $n$-dimensional cylinder $\widetilde{D}_k=\bigcup_{{\boldsymbol x}\in \widetilde J_k}\widetilde U_{\rho_k}({\boldsymbol x})$, 
where $\widetilde U_{\rho_k}({\boldsymbol x})$ is the $\rho_k$-neighborhood of ${\boldsymbol x}$ 
in the leaf of $\mathcal{F}_{\widetilde f}^{\mathrm{s}}$ containing $\widetilde{\boldsymbol x}$.
Note that $\widetilde U_{\rho_k}({\boldsymbol x})$ is an $(n-1)$-dimensional disk centered at ${\boldsymbol x}$.
By applying arguments in the proof of Theorem \ref{thmWD}, one can show 
that there exist an element $\widetilde g$ of $\mathrm{Diff}^r(M)$ arbitrarily $C^r$-close to $\widetilde f$ and a positive integer $k_0$ 
satisfying (OCD) in Definition \ref{describable}.
Thus, as in the 3-dimensional case discussed above,  
$\widetilde f$ is proved to be strongly pluripotent for $\Lambda_{\widetilde f}'$.
\end{proof}

\section{Proofs of Theorems \ref{thm_new} and \ref{DrcHst}}\label{S_Proof_thms}
\begin{proof}[Proof of Theorem \ref{thm_new}]
Here we work under the notations and conditions in Subsection  \ref{ss_quad_majority}.
Recall that $(\alpha_k)_{k\geq 1}$ is the increasing sequence of positive integers given in 
the proof of Lemma \ref{subexp}
and $\beta_k=k^2$.
We denote by $\gamma_k$ the greatest integer with $2\gamma_k\leq k^2$ for any $k\geq 2$.
For any positive integer $q$, we set $\mathbb{I}_k^{(q)}=\mathbb{I}_k^{(q)-}\sqcup \mathbb{I}_k^{(q)+}$, where
$$\mathbb{I}_k^{(q)-}= [\alpha_k+q,\alpha_k+\gamma_k-q]\cap 
\mathbb{Z}, \quad\mathbb{I}_k^{(q)+}=[\alpha_k+\gamma_k+q,\alpha_k+2\gamma_k-q]\cap 
\mathbb{Z}$$ 
if $q\leq \gamma_k/2$ and otherwise $\mathbb{I}_k^{(q)\pm}=\emptyset$.
For any integer $N\geq \alpha_1+\beta_1+1$, let $k_N$ be the greatest integer with 
$\alpha_{k_N}+\beta_{k_N}\leq N-1$.
By \eqref{eqn_|nkhat|}, for any sufficiently small $\varepsilon >0$, there exists an integer $N_0=N_0(\varepsilon,q)>0$ 
such that
\begin{equation}\label{eqn_I(q)N}
\begin{split}
\frac{\#\left\{ 0\le n \le N-1 \, ;\,  n \in \bigcup_{k=1}^{k_N} \mathbb{I}_k^{\,(q)}\right\}}{N}
&\geq 
\frac{\sum_{k=1}^{k_{N}}(k^2-4q+1)}{ \sum_{k=1}^{k_N+1}(k^2+O(k))}\\
&=\frac{2k_N^{3}/6+O(k_N^{2})}{2k_N^{3}/6+O(k_N^{2})}>1-\varepsilon
\end{split}
\end{equation}
for any $N\geq N_0$.
This implies that  
\begin{equation}\label{eqn_Nepsilon}
\#\Bigl\{[0,N-1]\cap \mathbb{Z}\setminus \bigcup _{k=1}^{k_N} \mathbb{I}_k^{\,(q)}\Bigr\}<N\varepsilon
\quad\text{if}\quad N\geq N_0.
\end{equation}

Let $z$ be any element of $\Lambda_f$ such that the binary code
$$\mathcal{I}(z)=\underline t=(\ \dots t_{-2}t_{-1}t_0t_1t_2\dots\ )$$
satisfies the following conditions for any $k\geq 2$.
\begin{itemize}
\item
$t_i=0$ for any $i$ with $\alpha_k+1\leq i\leq \alpha_k+\gamma_k$, 
\item 
$t_i=1$ for any $i$ with $\alpha_k+\gamma_k+1\leq i\leq \alpha_k+\beta_k$. 
\end{itemize}
From the definition of $\underline t$ together with \eqref{eqn_Nepsilon}, $\liminf\limits_{n\to\infty}p_n(\underline t)=1/2$, 
and hence $z\in \Lambda_f^{(\mathrm{mj})}$.
Then, by Theorem \ref{mainthm}, there exist an element $g$ of $\mathrm{Diff}^r(M)$ arbitrarily $C^r$-close to 
$f$ and a contracting wandering domain $D$ of $g$ satisfying the following equation.
\begin{equation}\label{eqn_g(z)}
\lim_{n\rightarrow \infty} \frac1{n}\sum_{j=0}^{n-1}\sup_{y\in D}\bigl\{\mathrm{dist}(g^j(y),g^j(z_g))\bigr\}=0,
\end{equation}
where $z_g\in \Lambda_g'$ is the continuation of $z$.

Next we consider another element $x$ of $\Lambda_f$ such that the binary code 
$$\mathcal{I}(x)=\underline v=(\ \dots v_{-2}v_{-1}v_0v_1v_2\dots\ )$$
satisfies the following conditions for any $k\geq 2$.
\begin{itemize}
\item
$v_i=1$ for any $i$ with $\alpha_k+1\leq i\leq \alpha_k+\gamma_k$,
\item 
$v_i=0$ for any $i$ with $\alpha_k+\gamma_k+1\leq i\leq \alpha_k+\beta_k$.
\end{itemize}
Then the continuation $x_g$ of $x$ is also an element of $\Lambda_g'$.

Note that $P_g=\mathcal{I}_g^{-1}(\ \dots 000\dots\ )$, $Q_g=\mathcal{I}_g^{-1}(\ \dots 111\dots\ )$ 
are the fixed points of $g$.
As in the proof of Theorem \ref{p-lemma}, one can choose $q$ so that the following 
condition holds for any positive integer $k$ with $k^2\geq 4q$ and $j\in \mathbb{I}_k^{(q)-}$.
\begin{itemize}
\item
$\{g^j(z_g),g^{j+\gamma_k}(x_g)\}$ and $\{g^{j+\gamma_k}(z_g),g^{j}(x_g)\}$ are contained in the $\varepsilon$-neighborhoods 
of $P_g$ and $Q_g$ in $M$ respectively.
\end{itemize}
In particular, we have 
\begin{align}
&\mathrm{dist}(g^j(z_g), g^{j+\gamma_k}(x_g))<2\varepsilon\quad\text{and}
\quad\mathrm{dist}(g^{j+\gamma_k}(z_g), g^{j}(x_g)) <2\varepsilon,\label{eqn_2e}\\
&\mathrm{dist}(g^j(z_g), g^{j}(x_g))>L-2\varepsilon\quad\text{and}
\quad\mathrm{dist}(g^{j+\gamma_k}(z_g), g^{j+\gamma_k}(x_g)) >L-2\varepsilon\label{eqn_L-2e}
\end{align}
for any $j\in \mathbb{I}_k^{(q)-}$, 
where $L=\mathrm{dist}(P_g,Q_g)$.

By \eqref{eqn_I(q)N}, \eqref{eqn_g(z)} and \eqref{eqn_L-2e}, 
for any sufficiently large $N\in\mathbb{N}$, 
\begin{align*}
\sum_{j=0}^{N-1}\inf_{y\in D}\mathrm{dist} &(g^j(y),g^j(x_g))\\
&>
\sum_{j\in \bigcup_{k=1}^{k_N} \mathbb{I}_k^{(q)}}
\inf_{y\in D}\mathrm{dist}(g^j(y),g^j(x_g))\\
&\geq 
\sum_{j\in \bigcup_{k=1}^{k_N} \mathbb{I}_k^{(q)}}
\inf_{y\in D}\bigl\{\mathrm{dist}(g^j(z_g),g^j(x_g))-\mathrm{dist}(g^j(z_g),g^j(y))\bigr\}\\
&\geq N(1-\varepsilon)(L-2\varepsilon)-\sum_{i=0}^{N-1}\sup_{y\in D}\{\mathrm{dist}(g^i(y),g^i(z_g))\}\\
&>N(1-\varepsilon)(L-2\varepsilon)-N\varepsilon.
\end{align*}
Since one can choose $\varepsilon$ arbitrarily small, this shows
$$\liminf_{n\rightarrow\infty}\frac1{n}\sum_{j=0}^{n-1}\inf_{y\in D}\mathrm{dist}(g^j(y),g^j(x_g))\geq L.$$

On the other hand, by \eqref{eqn_Nepsilon} and \eqref{eqn_g(z)}, 
for any sufficiently large $N$ and any Lipschitz function $\varphi:M\longrightarrow \mathbb{R}$ 
with $\varphi(M)\subset [-1,1]$ and $\mathrm{Lip}(\varphi)\leq 1$, 
\begin{align*}
&\left|\int_M\varphi\,d\delta_{y,g}^N-\int_M\varphi\,d\delta_{x_g,g}^N\right|
=\frac1{N}\biggl|\sum_{j=0}^{N-1}\bigr(\varphi(g^j(y))-\varphi(g^j(x_g))\bigr)\biggr|\\
&\hspace{70pt}\leq \frac1{N}\biggl|\sum_{j=0}^{N-1}\bigr(\varphi(g^j(z_g))-\varphi(g^j(x_g))\bigr)\biggr|
+\frac1{N}\biggl|\sum_{j=0}^{N-1}\bigr(\varphi(g^j(y))-\varphi(g^j(z_g))\bigr)\biggr|\\
&\hspace{70pt}\leq \frac1{N}\biggl|\sum_{j=0}^{N-1}\bigr(\varphi(g^j(z_g))-\varphi(g^j(x_g))\bigr)\biggr|+\varepsilon.
\end{align*}
Here we divide the total sum $\sum\limits_{j=0}^{N-1}$ into 
$\sum\limits_{j\in \bigcup_{k=1}^{k_N} \mathbb{I}_k^{(q)}}$ and 
$\sum\limits_{j\in [0,N-1]\cap\mathbb{Z}\setminus \bigcup_{k=1}^{k_N} \mathbb{I}_k^{(q)}}$.
By \eqref{eqn_2e}, 
\begin{align*}
\frac1{N}\biggl|& \sum_{j\in \bigcup_{k=1}^{k_N} \mathbb{I}_k^{(q)}}
\bigl(\varphi(g^j(z_g))-\varphi(g^j(x_g))\bigr)\biggr|\\
&=  
\frac1{N} \biggl|\sum_{j\in \bigcup_{k=1}^{k_N}\mathbb{I}_k^{(q)-}}
\Bigl(\varphi(g^j(z_g))-\varphi(g^j(x_g))+\varphi(g^{j+\gamma_k}(z_g))-\varphi(g^{j+\gamma_k}(x_g))\Bigr)\biggr|\\
&<\frac1{N} \sum_{j\in \bigcup_{k=1}^{k_N}\mathbb{I}_k^{(q)-}}
|\varphi(g^j(z_g))-\varphi(g^{j+\gamma_k}(x_g))|\\
&\hspace{70pt}
+\frac1{N} \sum_{j\in \bigcup_{k=1}^{k_N}\mathbb{I}_k^{(q)-}}
|\varphi(g^{j+\gamma_k}(z_g))-\varphi(g^{j}(x_g))|\\
&<\frac1{N}\sum_{j\in \bigcup_{k=1}^{k_N}\mathbb{I}_k^{(q)-}}
\mathrm{dist}(g^j(z_g),g^{j+\gamma_k}(x_g))
+\frac1{N} \sum_{j\in \bigcup_{k=1}^{k_N}\mathbb{I}_k^{(q)-}}
\mathrm{dist}(g^{j+\gamma_k}(z_g),g^{j}(x_g))\\
&<\frac1{N}\frac{N}2 2\varepsilon+\frac1{N}\frac{N}2 2\varepsilon=2\varepsilon.
\end{align*}
Since $\varphi(M)\subset [-1,1]$, $|\varphi(g^j(z_g))-\varphi(g^j(x_g))|\leq 2$.
Hence, by \eqref{eqn_Nepsilon},
$$\frac1{N}\biggl| \sum_{j\in [0,N-1]\cap\mathbb{Z}\setminus \bigcup_{k=1}^{k_N} \mathbb{I}_k^{(q)}}
\bigr(\varphi(g^j(z_g))-\varphi(g^j(x_g))\bigr)\biggr|
\leq \frac2{N}\#\Bigl\{[0,N-1]\cap \mathbb{Z}\setminus \bigcup _{k=1}^{k_N} \mathbb{I}_k^{\,(q)}\Bigr\}<2\varepsilon.
$$
By combining these inequalities, we have 
$$\sup_{y\in D}\left\{\sup_{\varphi} \left|\int_M\varphi\,d\delta_{y,g}^N-\int_M\varphi\,d\delta_{x_g,g}^N\right|\right\}\leq 5\varepsilon$$
for any sufficiently large $N$.
It follows that 
$$\lim_{n\to \infty}\sup_{y\in D}d_W(\delta_{y,g}^n,\delta_{x_g,g}^n)=
\lim_{n\to \infty}\sup_{y\in D}\left\{\sup_{\varphi} \left|\int_M\varphi\,d\delta_{y,g}^n-\int_M\varphi\,d\delta_{x_g,g}^n\right|\right\}=0.$$
This completes the proof.
\end{proof}

\begin{proof}[Proof of Theorem \ref{DrcHst}]
First, we give the proof of  \eqref{DrcHst-1}.
Let us focus on one of the saddle fixed points of  the $n$-dimensional diffeomorphism $\widetilde{g}$ 
given in the proof of Theorem \ref{mainthm}.
Then $\widetilde{g}$ has the saddle fixed point $P_{\widetilde g}$ which is the continuation of 
the saddle fixed point $P_{\widetilde f_0}$ of $\widetilde f_0$ with $P_{\widetilde f_0}=(0,\dots,0)\in 
\mathbb{R}^n$.

Consider a binary code satisfying the conditions in Lemma \ref{lem-3-1} for $\widetilde f$ instead of $f$ 
and the quadratic condition \eqref{eqn_k2},
that is, 
the length of the free part $\underline{u}_{k}$ is equal to $k^{2}$.
The binary code is still presented by $\underline{\widehat{w}}_{k}
=\underline{w}^{(n_{0}+Lk)}\underline{u}_{k}\underline{\iota}_{k}\underline{\gamma}^{(m_{k})}$ for simplicity.
Suppose that a sequence   
constructed from $(\underline{\widehat{w}}_{k})_{k\geq 1}$ as in Lemma \ref{l_xkinS}
 is also denoted by $(\widehat{\boldsymbol{x}}_{k})_{k\geq 1}$.
Now we set the free part $\underline{u}_{k}$ of $\underline{\widehat{w}}_{k}$
 such that the $\widetilde g$-orbit of ${\boldsymbol x}$ accumulates the saddle fixed point $P_{\tilde{g}}$.
 In practice, it should be set up as 
 \[
\underline{u}_{k}=\overbrace{
00\dots 0}^{k^{2}}.
\]
This implies that  $\tilde{g}$  
 has the non-trivial Dirac physical measure supported on 
 the saddle fixed point  $P_{\tilde{g}}$. See \cite[Theorem 5.5]{KNS23} for detail calculations.
 This concludes the proof of \eqref{DrcHst-1}.
 \smallskip
 
 Next,  let us prove  \eqref{DrcHst-2}. To implement historic behavior in 
 every forward orbit starting from a contracting wandering domain $\widetilde{D}$, 
 we have to prepare a code that oscillates between different dynamics in each generation and does not converge on any of them. 
 The easiest way might be the following.

\begin{itemize}
\item (Era condition)
We first consider 
an increasing sequence of integers 
 $(k_{s})_{s\in \mathbb{N}}$ 
such that, 
for every $s\in \mathbb{N}$,
\begin{equation}\label{era-ratio}
\sum_{k=k_{s}}^{k_{s+1}-1} k^{2}>s \sum_{k=k_{2}}^{k_{s}-1} k^{2}.
\end{equation}
\end{itemize}
Note that \eqref{era-ratio} provides the situation 
that the new era from 
$k_{s}$ to $k_{s+1}-1$ is so dominant that 
the old era from $k_{2}$ to $k_{s}-1$ is ignored. 
\begin{itemize}
\item (Code condition for oscillation)
Under the condition \eqref{era-ratio}, for each integer $k\ge k_{2}$, 
let $\underline{u}_{k}=(u_{1}u_{2}\ldots u_{k^{2}})$ be the code the entry of which
 satisfies the following rules: 
 \begin{subequations}
\begin{enumerate}[(1)]
\item 
if $s$ is even and $k_{s}\le k< k_{s+1}$, 
\begin{equation}\label{eqn_h1}
u_{i}=\left\{
\begin{array}{ll}
0 & \text{for}\ i=1,\ldots, \left\lfloor 3k^{2}/4\right\rfloor
\\[3pt]
1 & \text{for}\ i=\left\lfloor 3k^{2}/4\right\rfloor+1,\ldots, k^2
\end{array}\right.
\end{equation} 
that is, 
\[\underline{u}_k=
\overbrace{000\ldots\ldots0}^{\left\lfloor 3k^{2}/4\right\rfloor}
\overbrace{1\ldots 1}^{\left\lceil k^{2}/4\right\rceil},\]
\item
if $s$ is odd and $k_{s}\le k< k_{s+1}$,  
\begin{equation}\label{eqn_h2}
u_{i}=\left\{
\begin{array}{ll}
0 & \text{for}\ i=1,\ldots, \left\lfloor 7k^{2}/8\right\rfloor
\\[3pt]
1 & \text{for}\  i=\left\lfloor 7k^{2}/8\right\rfloor+1,\ldots, k^{2},
\end{array}\right.
\end{equation} 
that is, 
\[\underline{u}_k=
\overbrace{000\ldots\ldots0}^{\left\lfloor 7k^{2}/8\right\rfloor}
\overbrace{1\ldots 1}^{\left\lceil k^{2}/8\right\rceil},\]  
\end{enumerate}
\end{subequations}
where $\left\lfloor\cdot \right\rfloor$ and $\lceil\cdot\rceil$ indicate the 
floor and ceiling functions, respectively. 
\end{itemize}

The practical values of the ratios themselves, such as 3/4 or 7/8, are not meaningful, 
but it is important that they differ from each other according as the eras are even or odd.
Let $\underline{v}=(v_j)$ be any element of $\{0,1\}^{\mathbb{Z}}$ the sub-code 
$(v_j)_{j\geq k_2}$ of which satisfies \eqref{eqn_wiva}, 
see Figure \ref{f_9_1} again.
Note that, by \eqref{eqn_h1} and \eqref{eqn_h2}, $\underline{v}$ satisfies the quadratic condition 
\eqref{eqn_k2} and 
the majority condition  in  Definition \ref{m-cnd}. 
In fact, it follows from the equation $(3n)^{2/3}=k^{2}(1+O(k^{-1}))^{2/3}$ of \eqref{eqn_3n} that 
\[
\liminf_{n\to \infty}p_{n}(\underline{v})=\frac{3}{4}>\frac{1}{2}. 
\]
These facts imply that the open cylinder $\widetilde D=\mathrm{Int}\widetilde D_{k_2}$ given in the proof of Theorem \ref{mainthm} 
is a wandering domain of $\widetilde g$ the forward orbit of which 
has historic behavior. 
See \cite[Theorem 5.1]{KNS23}  for detail calculations. 
 This completes the proof of \eqref{DrcHst-2}.
 \end{proof}

\appendix
\section{Curvatures of leaves of 1 and 2-dimensional foliations}\label{Ap_curvature}

The results presented in this section are rather elementary.
Here we will use fundamental notations and results on differential geometry which are covered in standard textbooks, 
for example see \cite{dC76, Ko21} and so on.
For readers familiar with the differential geometry of curves and surfaces, the assertions below would 
be folklore.

For any $f\in \mathcal{O}(f_0)$, 
let $\ell$ be a $C^r$-arc in $\mathbb{H}_{\,[k]}$ with $k\geq 1$ adaptable to $\boldsymbol{C}_\varepsilon^{\mathrm{u}}$.
Then $\ell$ is parametrized as $\boldsymbol{x}(t)=(t,y(t),z(t))$ $(\alpha<t<\beta)$ with
\begin{equation}\label{eqn_y'}
|y'(t)|=O(\varepsilon),\quad |z'(t)|=O(\varepsilon).
\end{equation}
We denote by $\kappa_\ell(\boldsymbol{x}(t))$ and $\kappa_{f(\ell)}(f(\boldsymbol{x}(t)))$ the curvatures of $\ell$ and 
$f(\ell)$ at $\boldsymbol{x}(t)$ and $f(\boldsymbol{x}(t))$ respectively.
Then we have the following lemma.

\begin{lem}\label{l_l_curvature}
For any $t\in (\alpha,\beta)$, $\kappa_{f(\ell)}(f(\boldsymbol{x}(t)))< \dfrac12\kappa_\ell(\boldsymbol{x}(t))+O(\varepsilon)$.
\end{lem}

Note that $O(\varepsilon)$ here is a $C^{r-1}$-function of $\boldsymbol{x}\in \mathbb{H}_{\,[k]}$ satisfying $-C\varepsilon<O(\varepsilon)<C\varepsilon$ for some constant $C>0$ depending only on $\lambda_{\mathrm{u}}$, $\lambda_{\mathrm{ss}}$, 
$\lambda_{\mathrm{cs} 0}$ and $\lambda_{\mathrm{cs} 1}$.

\begin{proof}
Since $\boldsymbol{x}'(t)=(1,y'(t),z'(t))$, $\boldsymbol{x}''(t)=(0,y''(t),z''(t))$, by \eqref{eqn_y'} 
\begin{equation}\label{eqn_kappa_l}
\begin{split}
\kappa_\ell(\boldsymbol{x}(t))&=\frac{\|\boldsymbol{x}'(t)\times \boldsymbol{x}''(t)\|}{\|\boldsymbol{x}'(t)\|^3}=\frac{\sqrt{((y''(t))^2+(z''(t))^2)(1+O(\varepsilon))}}{(1+O(\varepsilon))^3}\\
&=\sqrt{(y''(t))^2+(z''(t))^2}(1+O(\varepsilon)).
\end{split}
\end{equation}
We set $f(\boldsymbol{x})=(f_1(\boldsymbol{x}),f_2(\boldsymbol{x}),f_3(\boldsymbol{x}))$ for $\boldsymbol{x}\in \mathbb{H}_{\,[k]}$.
By \eqref{eqn_f0V} and \eqref{eqn_def_zeta}, 
\begin{equation}\label{eqn_part_f}
\begin{split}
\frac{\partial f_1}{\partial x}(\boldsymbol{x})&=(-1)^i\lambda_{\mathrm{u}}+O(\varepsilon),\quad  
\frac{\partial f_2}{\partial y}(\boldsymbol{x})=(-1)^i\lambda_{\mathrm{ss}}+O(\varepsilon),\\
\frac{\partial f_3}{\partial z}(\boldsymbol{x})&=\lambda_{\mathrm{cs} i}+O(\varepsilon),
\end{split}
\end{equation}
where $i=0$ if $\boldsymbol{x}\in \mathbb{H}_{\,[k]}\cap \mathbb{V}_{0,f}$ and $i=1$ if $\boldsymbol{x}\in \mathbb{H}_{\,[k]}\cap\mathbb{V}_{1,f}$.
On the other hand, $\dfrac{\partial f_j}{\partial x_k}(\boldsymbol{x})=O(\varepsilon)$ for any $j,k\in \{1,2,3\}$ with $j\neq k$, where $(x_1,x_2,x_3)=(x,y,z)$.
We set $f(\boldsymbol{x}(t))=f(t)$ for short.
By the chain rule, 
\begin{align*}
\frac{df_j}{dt}(t)&=\frac{\partial f_j}{\partial x}(t)+\frac{\partial f_j}{\partial y}(t)y'(t)
+\frac{\partial f_j}{\partial z}(t)z'(t),\\
\frac{d^2f_j}{dt^2}(t)&=\frac{\partial^2 f_j}{\partial x^2}(t)+2\frac{\partial^2 f_j}{\partial x\partial y}(t)y'(t)
+2\frac{\partial^2 f_j}{\partial x\partial z}(t)z'(t)
+\frac{\partial^2 f_j}{\partial y^2}(t)(y'(t))^2\\
&\qquad+2\frac{\partial^2 f_j}{\partial y\partial z}(t)y'(t)z'(t)+\frac{\partial^2 f_j}{\partial z^2}(t)(z'(t))^2
+\frac{\partial f_j}{\partial y}(t)y''(t)+\frac{\partial f_j}{\partial z}(t)z''(t)
\end{align*}
for $j=1,2,3$.
Then, by \eqref{eqn_DDf} and \eqref{eqn_part_f}, we have
\begin{align*}
f''(t)=\bigl(O(\varepsilon)y''(t)+O(\varepsilon)z''(t),\ ((-1)^i&\lambda_{\mathrm{ss}}+O(\varepsilon))y''(t)+O(\varepsilon)z''(t),\\
&O(\varepsilon)y''(t)+(\lambda_{\mathrm{cs} i}+O(\varepsilon))z''(t)\bigr)+\boldsymbol{O}(\varepsilon)
\end{align*}
for $\boldsymbol{x}(t)\in \mathbb{V}_{i,f}$, 
where $\boldsymbol{O}(\varepsilon)=(O(\varepsilon),O(\varepsilon),O(\varepsilon))$.
Since $\lambda_{\mathrm{u}}^{-2}<1/4$, it follows from \eqref{eqn_kappa_l} that 
\begin{align*}
\kappa_{f(\ell)}(\boldsymbol{x}(t))&=
\frac{\|f'(t)\times f''(t)\|}{\|f'(t)\|^3}\\
&=\frac{\lambda_{\mathrm{u}}\sqrt{\lambda_{\mathrm{cs} i}^2(z''(t))^2+\lambda_{\mathrm{ss}}^2(y''(t))^2}\,(1+O(\varepsilon))}{\bigl((\lambda_{\mathrm{u}}\bigr)^2+O(\varepsilon)\bigr)^{3/2}}+O(\varepsilon)\\
&=\lambda_{\mathrm{u}}^{-2}\sqrt{\lambda_{\mathrm{cs} i}^2(z''(t))^2+\lambda_{\mathrm{ss}}^2(y''(t))^2}(1+O(\varepsilon))+O(\varepsilon)\\
&\leq \lambda_{\mathrm{u}}^{-2}\sqrt{(z''(t))^2+(y''(t))^2}(1+O(\varepsilon))+O(\varepsilon)< \frac12\kappa_\ell(\boldsymbol{x}(t))+O(\varepsilon).
\end{align*}
This completes the proof.
\end{proof}

\begin{prop}\label{p_curvature_curve}
For any leaf $l$ of $\mathcal{L}_{(0;\infty)}$ 
and any point $\boldsymbol{x}$ of $l$, 
$\kappa_\ell(\boldsymbol{x})=O(\varepsilon)$.
\end{prop}
\begin{proof}
First we consider the case of $\boldsymbol{x}\in W_{\mathrm{loc}}^{\mathrm{u}}(\Lambda_f)$.
Then $l$ is a leaf of $W_{\mathrm{loc}}^{\mathrm{u}}(\Lambda_f)$.
Since $f$ satisfies \eqref{eqn_DDf}, $l$ is a proper $C^r$-submanifold of $\mathbb{B}$ 
with $\kappa_l(\boldsymbol{x})=O(\varepsilon)$ by the stable manifold theorem (and its proof), 
for example, see Robinson \cite[Chapter 10, Theorem 2.1]{Ro99}.  
Next we suppose that $\boldsymbol{x}$ is an element of $\mathbb{H}_{\varepsilon_0}
\setminus W_{\mathrm{loc}}^{\mathrm{u}}(\Lambda_f)$.
Then there exist a positive integer $k$ and an element $\boldsymbol{x}_k\in \mathbb{H}_{\,[k]}\setminus 
f(\mathbb{H}_{\,[k+1]})$ with $f^k(\boldsymbol{x}_k)=\boldsymbol{x}$.
From the construction of $\mathcal{L}_{(k;\infty)}$, 
the leaf $l_k$ of $\mathcal{L}_{(k;\infty)}$ containing $\boldsymbol{x}_k$ is also a leaf of $\mathcal{L}_{(k;k+1)}$.
From the construction of $\mathcal{L}_{(k;k+1)}$, 
we know that  
$\kappa_{l_k}(\boldsymbol{x}_k)$ is an $O(\varepsilon)$-function.
By Lemma \ref{l_l_curvature}, 
$$\kappa_\ell(\boldsymbol{x})<\left(\sum_{i=0}^{k}\frac1{2^i}\right)O(\varepsilon)<2O(\varepsilon).$$
Thus one can complete the proof by regarding $2O(\varepsilon)$ as $O(\varepsilon)$ again.
\end{proof}

Let $F$ be any leaf of $\mathcal{F}_f^{\mathrm{s}}$ and $\boldsymbol{x}_0$ any point of $F\cap f(\mathbb{V}_{0,f}\cup \mathbb{V}_{1,f})$.
For any unit vector $\boldsymbol{u}$ tangent to $F$ at $\boldsymbol{x}_0$, $\widetilde{ \boldsymbol{u}}=D(f^{-1})(\boldsymbol{x}_0)\boldsymbol{u}$ 
is a non-zero vector tangent to $\widetilde F$ at $\widetilde{\boldsymbol{x}}_0=f^{-1}(\boldsymbol{x}_0)$, 
where $\widetilde F$ is the leaf of $\mathcal{F}_f^{\mathrm{s}}$ containing $\widetilde{\boldsymbol{x}}_0$.
Let $\kappa_{\boldsymbol{u}}(\boldsymbol{x}_0)$ (resp.\ $\kappa_{\widetilde{\boldsymbol{u}}}(\widetilde{\boldsymbol{x}}_0)$) be the normal curvature of $F$ (resp.\ 
$\widetilde F$) along $\boldsymbol{u}$ (resp.\ $\widetilde{\boldsymbol{u}}$).
Here we note that the the curvature of any spatial arc is non-negative by the definition.
On the other hand, for any spatial surface $S$ and a point $\boldsymbol{x}\in S$, 
the sign of normal curvature of $S$ along a vector tangent to $S$ at $\boldsymbol{x}$ depends on 
the choice of the normal direction of $S$ at $\boldsymbol{x}$.

As in Lemma \ref{l_l_curvature}, one can prove the following lemma.

\begin{lem}\label{l_F_curvature}
$|\kappa_{\widetilde{\boldsymbol{u}}}(\widetilde{\boldsymbol{x}}_0)|< \dfrac12 |\kappa_{\boldsymbol{u}}(\boldsymbol{x}_0)|+O(\varepsilon)$.
\end{lem}
\begin{proof}
Since $F$ is adaptable to $\boldsymbol{C}_{\varepsilon}^{\mathrm{cs}}$, the normal unit vector $\boldsymbol{N}$ of $F$ at $\boldsymbol{x}_0$ is 
represented as $(1,0,0)+\boldsymbol{O}(\varepsilon)$.
Recall that $\mathbb{B}$ is regarded as a subspace of $\mathbb{R}^3$.
Let $P$ be the plane in $\mathbb{R}^3$ containing $\boldsymbol{x}_0$ and tangent to $\boldsymbol{u}$ and $\boldsymbol{N}$ at $\boldsymbol{x}_0$.
See Figure \ref{f_A_1}.
\begin{figure}[hbtp]
\centering
\scalebox{0.6}{\includegraphics[clip]{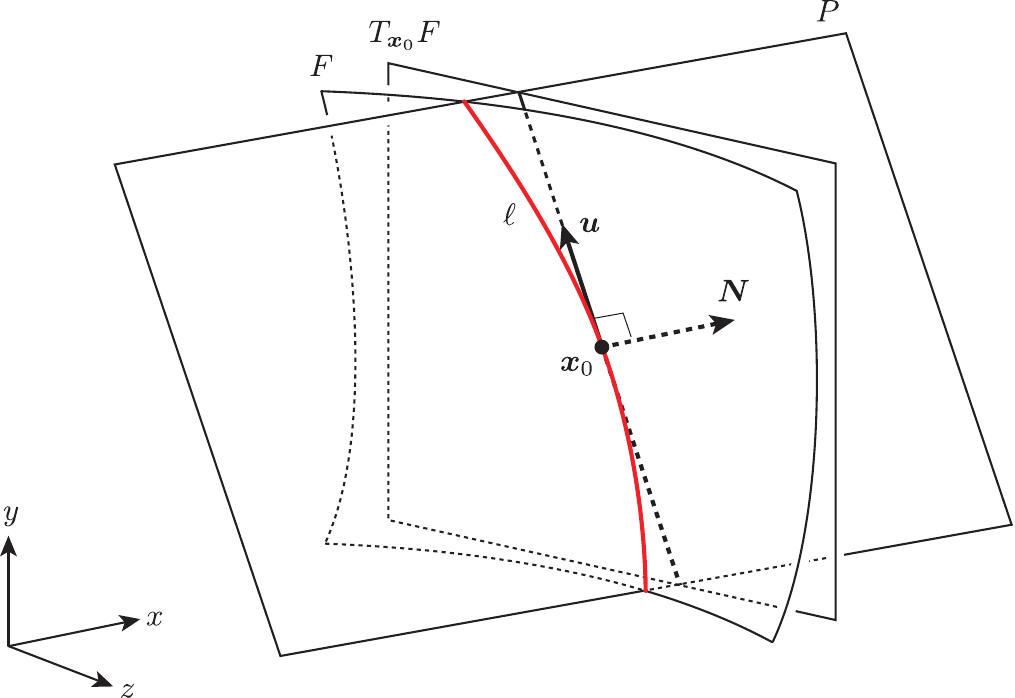}}
\caption{}
\label{f_A_1}
\end{figure}
Then there exists an orthogonal matrix $A$ of order three which has a form $A=E_3+\boldsymbol{O}_3(\varepsilon)$ and satisfies 
$(1,0,0)A=\boldsymbol{N}$ and $(0,a,b)A=\boldsymbol{u}$, where $E_3$ is the unit matrix of order three and 
$\boldsymbol{O}_3(\varepsilon)$ is a square matrix of order three each entry of which is an $O(\varepsilon)$-function and 
$a,b$ are constants with $a^2+b^2=1$.
Since $\ell=P\cap F$ is a curve with $T_{\boldsymbol{x}_0}\ell\ni \boldsymbol{u}$, it is parametrized as 
$$
\boldsymbol{x}(t)=\bigl(x(t)(1+O(\varepsilon)),\ x(t)O(\varepsilon)+at,\ 
x(t)O(\varepsilon)+bt\bigr)+\boldsymbol{x}_0+t\boldsymbol{O}(\varepsilon)
$$
for some $C^r$-function $x(t)$ with $x(0)=x'(0)=0$ 
defined on an open interval containing $0$.
Since the vector $\boldsymbol{O}(\varepsilon)$ here is independent of $t$, the first and second derivatives of $\boldsymbol{x}(t)$ are represented as
\begin{equation}\label{eqn_x'x''}
\begin{split}
\boldsymbol{x}'(t)&=\bigl(x'(t),\ x'(t)O(\varepsilon)+a,\ x'(t)O(\varepsilon)+b\bigr)+\boldsymbol{O}(\varepsilon),\\
\boldsymbol{x}''(t)&=\bigl(x''(t),\ x''(t)O(\varepsilon),\ x''(t)O(\varepsilon)\bigr)=x''(t)\bigl(1,O(\varepsilon),O(\varepsilon)\bigr).
\end{split}
\end{equation}
Here we do not incorporate $x''(t)O(\varepsilon)$ with $O(\varepsilon)$ since we could not exclude 
the case that $|x''(t)|$ is greater than $c\varepsilon^{-1}$ for some constant $c>0$.
The absolute value of the normal curvature $\kappa_{\boldsymbol{u}}(\boldsymbol{x}_0)$ of $F$ along $\boldsymbol{u}$ is equal to the curvature of $\ell$ at $\boldsymbol{x}_0$.
From the forms of $\boldsymbol{x}'(t)$ and $\boldsymbol{x}''(t)$, 
we have
\begin{equation}\label{eqn_kappau}
|\kappa_{\boldsymbol{u}}(\boldsymbol{x}_0)|=\kappa_\ell(\boldsymbol{x}_0)=|x''(0)|(1+O(\varepsilon)).
\end{equation}
On the other hand, the arc $f^{-1}(\ell)$ is parametrized as $\widetilde{\boldsymbol{x}}(t)=f^{-1}(\boldsymbol{x}(t))$.
By \eqref{eqn_f0V} and \eqref{eqn_def_zeta}, 
$$Df^{-1}(\boldsymbol{x})=\mathrm{diag}\bigl(\,(-1)^i\lambda_{\mathrm{u}}^{-1},\ (-1)^i\lambda_{\mathrm{ss}}^{-1},\ 
\lambda_{\mathrm{cs} i}^{-1}\,\bigr)+\widehat{\boldsymbol{O}}_3(\varepsilon)$$
if $\boldsymbol{x}\in f(\mathbb{V}_{i,f})$ for $i=0,1$, 
where $\widehat{\boldsymbol{O}}_3(\varepsilon)$ is a square matrix of order three each entry $\widehat O(\varepsilon)$ of which 
is a $C^{r-1}$ $O(\varepsilon)$-function on $\boldsymbol{x}=(x_1,x_2,x_3)$.
Moreover, by \eqref{eqn_DDf}, $\partial \widehat O(\varepsilon)/\partial x_j$ is also an $O(\varepsilon)$-function for $j=1,2,3$.
By these facts together with \eqref{eqn_x'x''}, the first and second derivatives of $\widetilde{\boldsymbol{x}}(t)$ are represented as
\begin{align*}
\widetilde{\boldsymbol{x}}'(t)&=
x'(t)\bigl((-1)^i\lambda_{\mathrm{u}}^{-1},\ (-1)^i\lambda_{\mathrm{ss}}^{-1}O(\varepsilon),\ \lambda_{\mathrm{cs} i}^{-1}O(\varepsilon)\bigr)\\
&\hspace{70pt}+\bigl(0,\ (-1)^i\lambda_{\mathrm{ss}}^{-1}a,\ \lambda_{\mathrm{cs} i}^{-1}b\bigr)+\boldsymbol{x}'(t)\widehat{\boldsymbol{O}}_3(\varepsilon)+\boldsymbol{O}(\varepsilon),\\
\widetilde{\boldsymbol{x}}''(t)&=x''(t)
\bigl((-1)^i\lambda_{\mathrm{u}}^{-1},\ (-1)^i\lambda_{\mathrm{ss}}^{-1}O(\varepsilon),\ \lambda_{\mathrm{cs} i}^{-1}O(\varepsilon)\bigr)
+\boldsymbol{x}''(t)\widehat{\boldsymbol{O}}_3(\varepsilon)+\boldsymbol{O}^\flat(\varepsilon).
\end{align*}
Here $\boldsymbol{O}^\flat(\varepsilon)$ represents $\boldsymbol{x}'(t)\dfrac{d\,\widehat{\boldsymbol{O}}_3(\varepsilon)(\boldsymbol{x}(t))}{dt}$, 
which is still an $O(\varepsilon)$-vector.
So we have  
\begin{align*}
\widetilde{\boldsymbol{x}}'(t)\times \widetilde{\boldsymbol{x}}''(t)&=x''(t)\bigl(O(\varepsilon),\ (-1)^i\lambda_{\mathrm{u}}^{-1}\lambda_{\mathrm{cs} i}^{-1}b,\ 
(-1)^{i+1}\lambda_{\mathrm{u}}^{-1}\lambda_{\mathrm{ss}}^{-1}a\bigr)\\
&\hspace{140pt}+x''(t)\boldsymbol{O}(\varepsilon)+\boldsymbol{O}(\varepsilon).
\end{align*}
This implies that 
\begin{align*}
\|\widetilde{\boldsymbol{x}}'(t)\|&=\sqrt{|x'(t)|^2\lambda_{\mathrm{u}}^{-2}+\varGamma_i}+O(\varepsilon)\geq \sqrt{\varGamma_i}+O(\varepsilon),\\
\|\widetilde{\boldsymbol{x}}'(t)\times \widetilde{\boldsymbol{x}}''(t)\|&=|x''(t)|\lambda_{\mathrm{u}}^{-1}(\sqrt{\varGamma_i}+O(\varepsilon))+O(\varepsilon),
\end{align*}
where $\varGamma_i=\lambda_{\mathrm{ss}}^{-2}a^2+\lambda_{\mathrm{cs} i}^{-2}b^2$.
Since $\lambda_{\mathrm{u}}>2$ and $\varGamma_i>1$, by \eqref{eqn_kappau} we have 
\begin{align*}
|\kappa_{\widetilde{\boldsymbol{u}}}(\widetilde{\boldsymbol{x}}_0)|&\leq \kappa_{f^{-1}(\ell)}(\widetilde{\boldsymbol{x}}_0)=\frac{\|\widetilde{\boldsymbol{x}}'(0)\times \widetilde{\boldsymbol{x}}''(0)\|}{\|\widetilde{\boldsymbol{x}}'(0)\|^3}
<\frac{|x''(0)|\lambda_{\mathrm{u}}^{-1}(\sqrt{\varGamma_i}+O(\varepsilon))+O(\varepsilon)}{(\sqrt{\varGamma_i}+O(\varepsilon))^3}\\
&<\lambda_{\mathrm{u}}^{-1}|x''(0)|+O(\varepsilon)<\frac12 |\kappa_{\boldsymbol{u}}(\boldsymbol{x}_0)|+O(\varepsilon).
\end{align*}
Here the first inequality is immediately obtained from the definition of normal curvature, for example, see 
\cite[Chapter 3, Definition 3]{dC76} or \cite[Section 2.2]{Ko21}.
This completes the proof.
\end{proof}

\begin{prop}\label{p_curvature_plane}
For any leaf $F$ of $\mathcal{F}_f^{\mathrm{s}}$ and any unit tangent vector $\boldsymbol{u}\in T_{\boldsymbol{x}} F$ 
with $\boldsymbol{x}\in F$, 
the absolute value $|\kappa_{\boldsymbol{u}}(\boldsymbol{x})|$ of the normal curvature of $F$ at $\boldsymbol{x}$ along $\boldsymbol{u}$ is $O(\varepsilon)$.
In particular, the principal curvatures $\kappa_{F,1}(\boldsymbol{x})$ and $\kappa_{F,2}(\boldsymbol{x})$ of $F$ at $\boldsymbol{x}$ 
satisfy $|\kappa_{F,i}(\boldsymbol{x})|=O(\varepsilon)$ for $i=1,2$.
\end{prop}
\begin{proof}
If $\boldsymbol{x}\in W_{\mathrm{loc}}^{\mathrm{s}}(\Lambda_f)$, then we have as in the proof of 
Proposition \ref{p_curvature_curve} $|\kappa_{\boldsymbol{u}}(\boldsymbol{x})|=O(\varepsilon)$.
Let $\mathbb{G}_0$ be the component of $\mathbb{B}\setminus W_{\mathrm{loc}}^{\mathrm{s}}(\Lambda_f)$ 
containing $\mathbb{H}_{\varepsilon_0}$.
One can choose $\mathcal{F}_f^{\mathrm{s}}$ so that $|\kappa_{\boldsymbol{u}}(\boldsymbol{x})|
=O(\varepsilon)$ 
if $\boldsymbol{x}\in \mathbb{G}_0$.
Intuitively, such a foliation on $\mathbb{G}_0$ is obtained by pushing the two leaves of $W_{\mathrm{loc}}^{\mathrm{s}}(\Lambda_f)$ adjacent to $\mathbb{G}_0$ toward $\mathbb{H}_{\varepsilon_0}$ 
with the same ratio 
along the lines in $\mathbb{B}$ parallel to the $x$-axis.
If $\boldsymbol{x}\in \mathbb{B}\setminus (W_{\mathrm{loc}}^{\mathrm{s}}(\Lambda_f)\cup \mathbb{G}_0)$, 
then there exists a positive integer $k$ such that 
$f^j(\boldsymbol{x})=\boldsymbol{x}_{j}$, $Df^j(\boldsymbol{x})\boldsymbol{u}=\boldsymbol{u}_{j}$ $(j=1,\dots,k)$ with $\boldsymbol{x}_k\in \mathbb{G}_0$.
By Lemma \ref{l_F_curvature}, $|\kappa_{\boldsymbol{u}_{j-1}}(\boldsymbol{x}_{j-1})|<\dfrac12 |\kappa_{\boldsymbol{u}_{j}}(\boldsymbol{x}_{j})|+O(\varepsilon)$, 
where $\boldsymbol{x}_0=\boldsymbol{x}$ and $\boldsymbol{u}_0=\boldsymbol{u}$.
Since $|\kappa_{\boldsymbol{u}_k}(\boldsymbol{x}_k)|=O(\varepsilon)$, 
we have 
$|\kappa_{\boldsymbol{u}}(\boldsymbol{x})|<\biggl(\sum\limits_{j=0}^{k-1}\dfrac1{2^j}\biggr)O(\varepsilon)<2O(\varepsilon).$
Thus one can complete the proof by regarding $2(\varepsilon)$ as $O(\varepsilon)$.
\end{proof}

Suppose that $(\widehat{\boldsymbol{x}}_k)_{k\geq 1}$ with $\widehat{\boldsymbol{x}}_k\in S_{\widehat{\underline{w}}_k}^{\mathrm{cs}}$ is the sequence given in Lemma \ref{l_xkinS} and 
$g=f\circ \psi_n$ is the diffeomorphism of \eqref{eqn_gfpsi}, 
which satisfies the conclusion of Theorem \ref{thmWD} if $n$ is sufficiently large.
Note that $f^{\widehat n_k}(J_k)$ is the leaf of $\mathcal{L}_{(0,\infty)}$ containing $f^{\widehat n_k}(\widehat{\boldsymbol{x}}_k)$.
Since $\psi_n$ is $C^r$-close to the identity by Lemma \ref{l_psi_n}, 
Proposition \ref{p_curvature_curve} implies that the curvature of $\psi_n\circ f^{\widehat n_k}(J_k)$ 
is an $O(\varepsilon)$-function.
Let $\widehat{\boldsymbol{y}}(t)=(t,y(t),z(t))$ $(-\alpha<t<\beta)$ be a parametrization of $\psi_n\circ f^{\widehat n_k}(J_k)$ with 
$\widehat{\boldsymbol{y}}(0)=\widehat{\boldsymbol{y}}_{k+1}=\psi_n\circ f^{\widehat n_k}(\widehat{\boldsymbol{x}}_k)$ for some $\alpha,\beta>0$.
Since $f$ is sufficiently $C^r$-close to $f_0$, 
it follows from the form \eqref{eqn_tang} of $f_0^2$ on $\mathbb{H}_{\varepsilon_0}$ that 
$g^{\widehat n_k+2}(J_k)=f^2\circ \psi_n\circ f^{\widehat n_k}(J_k)$ has a parametrization such as 
$$\widehat{\boldsymbol{x}}(t)=f^2(\widehat{\boldsymbol{y}}(t))=\bigl(-a_1t^2+a_2z(t),\ a_2y(t),\ a_4t\bigr)+\widehat{\boldsymbol{x}}_{k+1}+\boldsymbol{O}(\varepsilon),$$
where the $i$-th entry $O_i(\boldsymbol{x})$ of $\boldsymbol{O}(\varepsilon)$ is an $O(\varepsilon)$-function of $\boldsymbol{x}=(x_1,x_2,x_3)\in \mathbb{H}_{\varepsilon_0}$ with 
$\partial O_i(\boldsymbol{x})/\partial x_j=O(\varepsilon)$, $\partial^2 O_i(\boldsymbol{x})/\partial x_j\partial x_k=O(\varepsilon)$ for any $i,j,k\in \{1,2,3\}$.
By \eqref{eqn_y'} and Proposition \ref{p_curvature_curve}, $y'(t),z'(t),y''(t),z''(t)=O(\varepsilon)$ 
if necessary supposing that $\widehat n_1$ is greater than the integer $k_0$ given in 
Proposition \ref{p_curvature_curve}.
This shows that 
\begin{equation}\label{eqn_whx'}
\widehat{\boldsymbol{x}}'(t)=(-2a_1t, 0, a_4)+\boldsymbol{O}(\varepsilon),\quad \widehat{\boldsymbol{x}}''(t)=(-2a_1, 0, 0)+\boldsymbol{O}(\varepsilon).
\end{equation}

Then we have the following proposition.

\begin{prop}\label{p_kappa_fn}
For any $-\alpha<t<\beta$, 
$$
\kappa_{g^{\widehat n_k+2}(J_k)}(\widehat{\boldsymbol{x}}(t))=\frac{2a_1|a_4|}{(4a_1t^2+a_4^2)^{3/2}}+O(\varepsilon).$$
Moreover the unit normal vector $\boldsymbol{N}(\widehat{\boldsymbol{x}}_{k+1})$ of $g^{\widehat n_k+2}(J_k)$ at $\widehat{ \boldsymbol{x}}_{k+1}=\widehat{\boldsymbol{x}}(0)$ is 
$(-1,0,0)+\boldsymbol{O}(\varepsilon)$.
\end{prop}
\begin{proof}
The form of $\kappa_{g^{\widehat n_k+2}(J_k)}(\widehat{\boldsymbol{x}}_{k+1})$ as above is obtained immediately from \eqref{eqn_whx'}.
The arc length of $\widehat{\boldsymbol{x}}(t)$ is given as 
$$s=\int_0^t\sqrt{4a_1u^2+a_4^2 +O(\varepsilon)}\,du.$$
Then an elementary calculation shows 
that $\dfrac{d^2\widehat{\boldsymbol{x}}}{ds^2}(0)=\dfrac1{a_4^2}(-2a_1,0,0)+\boldsymbol{O}(\varepsilon)$.
Since we supposed that $a_1>0$, we have $\boldsymbol{N}(\widehat{\boldsymbol{x}}_{k+1})=(-1,0,0)+\boldsymbol{O}(\varepsilon)$ 
by unitizing $\dfrac{d^2\widehat{\boldsymbol{x}}}{ds^2}(0)$.
\end{proof}

\bibliographystyle{plain}

\bibliography{ref}

\printindex

\end{document}